\documentclass[11pt]{amsart}
\usepackage[utf8]{inputenc}
\usepackage{amssymb,amsmath, amsthm, mathabx}
\usepackage{color, enumerate}
\usepackage{hyperref}
\usepackage{caption}
\usepackage{subcaption}
\usepackage{a4wide}
\usepackage{accents}
\usepackage{tikz}
\usepackage{graphicx}
\usepackage{chngcntr}
\usepackage{datetime2}
\usepackage{scalerel}    

\usepackage{xifthen}

\counterwithin{figure}{section}
\counterwithin{equation}{section}
\allowdisplaybreaks

\newtheorem{thm}{Theorem}[section]

\newtheorem{lem}[thm]{Lemma}
\newtheorem{cor}[thm]{Corollary}
\newtheorem{defn}[thm]{Definition}
\theoremstyle{remark}
\newtheorem{rem}{Remark}[thm]

\newcommand*{\cl}[1]{\overline{#1}}

\def\a{\mathbf{a}}
\def\b{\mathbf{b}}
\newcommand{\ca}[2]{\Fab{\a}{#2}{#1}}
\newcommand{\cb}[2]{\Fab{\b}{#2}{#1}}

\def\w{\omega}
\newcommand{\cw}[1]{\Fa{\w}{#1}}
\newcommand{\csw}[1]{\Fa{\wh{\w}}{#1}}

\def\Blk{\Omega}
\def\BlkBd{\wh{\Omega}}

\def\wp{\omega}
\newcommand{\cwp}[2]{\Fabs{\wp}{#1}{#2}{\a, \b}}
\newcommand{\cwph}[2]{\Fab{\wp}{#1}{#2}}

\def\P{\bfP}
\def\E{\bfE}
\newcommand{\cP}{\P}
\newcommand{\cE}{\E}

\newcommand{\cPf}[1]{\Fbs{\P}{#1}{\a, \b}}

\newcommand{\csPf}[2]{\Fbs{\wh{\P}}{#1}{\a, \b, #2}}

\def \G{\mathrm G}
\newcommand{\cG}[1]{\Fa{\G}{#1}}
\newcommand{\cGi}[2]{\Fab{\G}{#1}{#2}}
\newcommand{\csG}[1]{\Fa{\wh{\G}}{#1}}
\newcommand{\csGi}[2]{\Fab{\wh{\G}}{#1}{#2}}

\def\H{\mathrm{H}}
\newcommand{\cH}[2]{\Fa{\H}{#1, #2}}
\newcommand{\cHp}[2]{\Fas{\H}{#1, #2}{\a, \b}}

\def\prt{\sigma}
\newcommand{\cprt}[2]{\Fa{\prt}{#1, #2}}
\newcommand{\cprtp}[2]{\Fas{\prt}{#1, #2}{\a, \b}}

\def\Flux{\mathrm{F}}
\newcommand{\cFlux}[2]{\Fa{\Flux}{#1, #2}}
\newcommand{\cFluxp}[2]{\Fas{\Flux}{#1, #2}{\a, \b}}

\def\lH{\mathrm{h}}
\newcommand{\clH}[2]{\Fas{\lH}{#1, #2}{\alpha, \beta, \aml, \bml}}

\def\lFlux{\mathrm{f}}
\newcommand{\clFlux}[2]{\Fas{\lFlux}{#1, #2}{\alpha, \beta, \aml, \bml}}

\def\HE{\wh{\mathrm H}}
\def\VE{\wh{\mathrm V}}
\newcommand{\cHE}[1]{\Fa{\HE}{#1}}
 \newcommand{\cHEi}[2]{\Fab{\HE}{#1}{#2}}
\newcommand{\cVE}[1]{\Fa{\VE}{#1}}
 \newcommand{\cVEi}[2]{\Fab{\VE}{#1}{#2}}

\def\I{\mathrm I}
\def\J{\mathrm J}
\newcommand{\csI}[1]{\Fa{\wh{\I}}{#1}}
\newcommand{\csJ}[1]{\Fa{\wh{\J}}{#1}}

\def \Gp{\mathrm G}
\newcommand{\cGp}[1]{\Fas{\Gp}{#1}{\a, \b}}

\newcommand{\csGp}[2]{\Fas{\wh{\Gp}}{#1}{\a, \b, #2}}

\def\EStShp{\mathrm M} 
\newcommand{\cEStShp}[2]{\Fas{\EStShp}{#1}{\a, \b, #2}}
\newcommand{\cEStShphs}[3]{\Fas{\EStShp}{#1}{\tau_{#3}\a, \b, #2}}
\newcommand{\cEStShpvs}[3]{\Fas{\EStShp}{#1}{\a, \tau_{#3}\b, #2}}

\def\EShp{\mathcal M}
\newcommand{\cEShp}[1]{\Fas{\EShp}{#1}{\a, \b}}
\newcommand{\cEShphs}[2]{\Fas{\EShp}{#1}{\tau_{#2}\a, \b}}
\newcommand{\cEShpvs}[2]{\Fas{\EShp}{#1}{\a, \tau_{#2}\b}}

\def \Iz {I}
\newcommand{\cIz}[1]{\Fbs{\Iz}{#1}{\a, \b}}
\def \aml {\mfa}
\def \bml {\mfb}

\def \ams{\mathfrak{A}}
\def \bms{\mathfrak{B}}

\def\Shp{\Gamma}
\def\shp{\gamma}
\newcommand{\cStShp}[2]{\Fabs{\Shp}{#1}{#2}{\alpha, \beta}}
\newcommand{\cShp}[1]{\Fas{\shp}{#1}{\alpha, \beta, \aml, \bml}}

\def\MinVar{\mathrm C}
\newcommand{\cMinVar}[1]{\Fas{\MinVar}{#1}{\bfa, \bfb}}
\newcommand{\cMinVarhs}[2]{\Fas{\MinVar}{#1}{\tau_{#2}\bfa, \bfb}}
\newcommand{\cMinVarvs}[2]{\Fas{\MinVar}{#1}{\bfa, \tau_{#2}\bfb}}

\def\MinDis{\Delta}
\newcommand{\cMinDis}[1]{\Fas{\MinDis}{#1}{\bfa, \bfb}}

\newcommand{\cMinDisvs}[2]{\Fas{\MinDis}{#1}{\bfa, \tau_{#2}\bfb}}

\def\MaMin{\pmb{\zeta}}
\newcommand{\cMaMin}[1]{\Fas{\MaMin}{#1}{\alpha, \beta, \aml, \bml}}

\def\LimShp{\mathcal{R}}
\def\LimShpBd{\mathcal{I}}
\newcommand{\cLimShp}{\Fs{\LimShp}{\alpha, \beta, \aml, \bml}}

\newcommand{\cLimShpBd}[1]{\Fs{\LimShpBd}{\alpha, \beta, #1}}

\def\A{\mathrm A}
\def\B{\mathrm B}
\newcommand{\cA}[1]{\Fas{\A}{#1}{\alpha}}
\newcommand{\cB}[1]{\Fas{\B}{#1}{\beta}}

\newcommand{\bbC}{\mathbb{C}}

\newcommand{\bbQ}{\mathbb{Q}}
\newcommand{\bbR}{\mathbb{R}}

\newcommand{\bbZ}{\mathbb{Z}}

\newcommand{\bfE}{\mathbf{E}}

\newcommand{\bfP}{\mathbf{P}}

\newcommand{\bfa}{\mathbf{a}}
\newcommand{\bfb}{\mathbf{b}}

\newcommand{\bfk}{\mathbf{k}}

\newcommand{\bfu}{\mathbf{u}}

\newcommand{\mfa}{\mathfrak{a}}
\newcommand{\mfb}{\mathfrak{b}}

\newcommand{\sC}{\mathcal{C}}

\newcommand{\sG}{\mathcal{G}}

\newcommand{\sM}{\mathcal{M}}

\newcommand{\sX}{\mathcal{X}}

\newcommand{\one}{\mathbf{1}}

\newcommand{\lc}{\lceil}
\newcommand{\rc}{\rceil}

\newcommand{\Var}{\mathbf{Var}}

\DeclareMathOperator\supp{supp}

\def\Exp{\mathrm {Exp}}

\def\dd{\mathrm d}

\def \I{\mathrm I}
\def \J{\mathrm J}

\def \T{\mathrm T}

\def\sG{\wh{\mathrm G}}

\definecolor{darkblue}{rgb}{.1,.1,.6}

\newcommand{\wt}[1]{\widetilde{#1}}
\newcommand{\wh}[1]{\widehat{#1}}

\newcommand{\Fabs}[4]
{
\ifthenelse{\isempty{#2}}
{
{#1}_{#3}^{#4}
}
{
{#1}_{#3}^{#4}(#2)
}
}
\newcommand{\Fab}[3]
{
\Fabs{#1}{#2}{#3}{}
}
\newcommand{\Fas}[3]
{
\Fabs{#1}{#2}{}{#3}
}
\newcommand{\Fbs}[3]
{
\Fabs{#1}{}{#2}{#3}
}
\newcommand{\Fa}[2]
{
\Fabs{#1}{#2}{}{}
}

\newcommand{\Fs}[2]
{
\Fabs{#1}{}{}{#2}
}

\def\EA{\mathrm A}
\def\EB{\mathrm B}
\newcommand{\cEA}[2]{\Fabs{\EA}{#1}{#2}{\bfa}}
\newcommand{\cEB}[2]{\Fabs{\EB}{#1}{#2}{\bfb}}

\def\EMin{\zeta}

\newcommand{\cEMin}[1]{\Fas{\EMin}{#1}{\bfa, \bfb}}

\newcommand{\cEMinhs}[2]{\Fas{\EMin}{#1}{\tau_{#2}\bfa, \bfb}}
\newcommand{\cEMinvs}[2]{\Fas{\EMin}{#1}{\bfa, \tau_{#2}\bfb}}

\newcommand{\cMiEh}[1]{\Fa{\MiEh}{#1}}
\newcommand{\cMiEv}[1]{\Fa{\MiEv}{#1}}

\def\MiEh{\wh{\mathrm{H}}}
\def\MiEv{\wh{\mathrm{V}}}

\def\HDis{\mathrm{d}_{\mathrm{H}}}
\newcommand{\cHDis}[2]{\Fa{\HDis}{#1, #2}}

\def\Clus{\mathrm{R}}
\newcommand{\cClus}[1]{\Fbs{\Clus}{#1}{\a, \b}}

\def\StCon{\mathcal{S}}
\newcommand{\cRc}{\Fs{\StCon}{\alpha, \beta, \aml, \bml}}

\newcommand{\cRh}{\Fs{\mathcal{H}}{\alpha, \beta, \aml, \bml}}

\newcommand{\cRv}{\Fs{\mathcal{V}}{\alpha, \beta, \aml, \bml}}

\def\Cat{\mathrm{C}}

\newcommand{\bCat}[2]{\Cat_{#1}(#2)}

\definecolor{darkgreen}{rgb}{.1,.6,.1}

\def\w{\omega} 
 
\def\Z{\mathbb Z} 
\newcommand{\be}{\begin{equation}}   
\newcommand{\ee}{\end{equation}}

\usepackage{todonotes, xargs}
\newcommandx{\note}[2][1=]{\todo[linecolor=yellow,backgroundcolor=yellow!25,bordercolor=yellow,#1]{#2}}

\usetikzlibrary{decorations.pathreplacing}
\usetikzlibrary{calc}

\title[Corner growth model]
{Flats, spikes and crevices: the evolving shape of the inhomogeneous corner growth model}

\author[E.~Emrah]{Elnur Emrah}
\address{Elnur Emrah\\ KTH Royal Institute of Technology \\ Department of Mathematics \\ SE-100 44 Stockholm\\ Sweden.}
\email{elnur@kth.se}
\urladdr{https://sites.google.com/view/elnur-emrah}
\thanks{E.\ Emrah was partially supported by the grant KAW 2015.0270 from the Knut and Alice Wallenberg Foundation and by the Mathematical Sciences Department at Carnegie Mellon University through a postdoctoral position.} 

\author[C.~Janjigian]{Christopher Janjigian}
\address{Christopher Janjigian\\ Purdue University\\  Mathematics Department\\  150 N. University Street
\\   West Lafayette, IN 47907\\ USA.}
\email{cjanjigi@purdue.edu}
\urladdr{http://www.math.purdue.edu/~cjanjigi}
\thanks{C.\ Janjigian was partially supported by National Science Foundation grant DMS-1954204 and by a postdoctoral grant from the Fondation Sciences Math\'ematiques de Paris while working at Universit\'e Paris Diderot.}

\author[T.~Sepp\"al\"ainen]{Timo Sepp\"al\"ainen}
\address{Timo Sepp\"al\"ainen\\ University of Wisconsin-Madison\\  Mathematics Department\\ Van Vleck Hall\\ 480 Lincoln Dr.\\   Madison WI 53706-1388\\ USA.}
\email{seppalai@math.wisc.edu}
\urladdr{http://www.math.wisc.edu/~seppalai}
\thanks{T.\ Sepp\"al\"ainen was partially supported by  National Science Foundation grant  DMS-1602486 and DMS-1854619, and by the Wisconsin Alumni Research Foundation.} 

\thanks{}

\date{\DTMnow. } 

\begin{document}
\begin{abstract}
We study the macroscopic evolution of the growing cluster in the exactly solvable corner growth model with independent exponentially distributed waiting times. The rates of the exponentials are given by an addivitely separable function of the site coordinates. When computing the growth process (last-passage times) at each site, the horizontal and vertical additive components of the rates are allowed to also vary respectively with the column and row number of that site. This setting includes several models of interest from the literature as special cases. 
Our main result provides simple explicit variational formulas for the a.s.\ first-order asymptotics of the growth process under a decay condition on the rates. Formulas of similar flavor were conjectured in \cite{rain-00}, which we also establish. Subject to further 
mild conditions, we prove the existence of the limit shape and describe it explicitly. We observe that the boundary of the limit shape can develop flat segments adjacent to the axes and spikes along the axes. Furthermore, we record the formation of persistent macroscopic spikes and crevices in the cluster that are nonetheless not visible in the limit shape. As an application of the results for the growth process, we compute the flux function and limiting particle profile for the TASEP with the step initial condition and disorder in the jump rates of particles and holes. Our methodology is based on concentration bounds and estimating the boundary exit probabilities of the geodesics in the increment-stationary version of the model, with the only input from integrable probability being the distributional invariance of the last-passage times under permutations of columns and rows.         
\end{abstract}

\maketitle



\section{Introduction}

\subsection{Some background and the contribution of the present work} \label{SsCont}

Stochastic growth far from equilibrium arises in models of diverse phenomena such as propagation of burning fronts, spread of infections, colonial growth of bacteria, liquid penetration 
into porous media and vehicular traffic flow \cite{halp-zhan-95, krug-spoh}. 
In these models, a growth process describes the time-evolution of a randomly growing cluster that represents, for example, a tissue of infected cells. Mathematical study of growth processes dates at least back to Eden's model \cite{eden-61}. A fundamental type of result in this subject is that the cluster associated with a given growth process acquires a deterministic \emph{limit shape} in a suitable scaling limit. Understanding the geometric properties of 
limit shapes has been one of the main research themes, for example, in percolation theory since at least the seminal work of D. Richardson \cite{rich73}. See this brief introduction \cite{damr-rass-sepp-16} and survey articles \cite{auff-damr-hans-16, damr-18, mart-06}. 

Of particular interest is to determine whether and in what manner local inhomogeneities in a growth model are manifested in the limit shape. This line of inquiry was pursued in both mathematics and physics literature with some early work in the 1990s, particularly on disordered exclusion and related growth processes \cite{benj-ferr-land, jano-lebo-92, jano-lebo-94, kand-muka-92, krug-ferr, schu-doma-93, sepp-krug, trip-barm-98, wolf-tang-90}. It has now been rigorously observed in various settings that suitably introduced inhomogeneity into the parameters of a growth model can create new geometric features in the limit shape including flat segments \cite{baha-bodi-18, barr-15, emra-16-ecp, ciec-geor-18, grav-trac-wido-02b, grav-trac-wido-02a, kniz-petr-saen-18, sepp-krug, sly-16}, spikes \cite{beff-sido-vares-10}, corners \cite{ciec-geor-18} and pyramids \cite{ahlb-damr-sido-16}.
Such features are often an indication of a \emph{phase transition} at the level of fluctuations of the growth process as observed, for example, in \cite{baik-bena-pech-05, barr-15, bena-corw, grav-trac-wido-02b, grav-trac-wido-02a, kniz-petr-saen-18}.

The present paper revisits the exactly solvable, inhomogeneous CGM from \cite{boro-pech, joha-08} that has independent and exponentially distributed waiting times with possibly distinct rates given by an addivitely separable function of the site coordinates. Hence, the inhomogeneity can be represented in terms of real parameters (namely the additive components of the rates) attached to the columns and rows. The model arises naturally in several contexts, including from the totally asymmetric simple exclusion process (TASEP) with the step initial condition and particlewise and holewise disorder. As elaborated on in Subsection \ref{SsiCGM}, we generalize the model slightly by computing the growth process at each site from a distinct collection of waiting times. With this enhancement, the model in particular  
unifies the following somewhat disparate settings from the literature. 
\begin{enumerate}[(\romannumeral1)]
\item Random rates from some work in the 1990s on TASEP with the step initial condition and particlewise disorder \cite{benj-ferr-land, krug-ferr, sepp-krug} and more recently in \cite{emra-16-ecp, emra-janj-17}. 
\item Macroscopically inhomogeneous rates as in \cite{cald-15, ciec-geor-18} in the special case that the speed function is additively separable. A discrete version of such a model (with geometrically distributed waiting times) appeared recently in \cite{kniz-petr-saen-18}.  
\item Fixed defective rates on the south or west boundaries as in \cite{bala-cato-sepp}, on a thickened west boundary (a few columns) as in \cite{baik-bena-pech-05} and, more generally, on thickened south and west boundaries (a few columns and rows) as in \cite{bena-corw}.
\item Suitably rescaled defective rates in a few columns and rows considered in \cite{bena-corw, boro-pech}. The discrete version of the model in \cite{bena-corw} appeared later in \cite{corw-liu-wang-16}.   
\item Defective rates near north or east boundaries used in \cite{sepp-cgm-18}. 
\item Growing rates as in \cite{joha-08}. 
\end{enumerate} 
(Precise connections to the above models are explained in the longer version of this article \cite[Subsection 3.9]{emra-janj-sepp-19-shp-long}). 

When the rates are identical, a well-known result of H.\ Rost \cite{rost} identifies the limit shape as a certain explicit parabolic region. 
In this homogeneous case, the limit shape completely governs the growth of the cluster to the leading order in time. The purpose of this work is to study the macroscopic evolution of the cluster in presence of inhomogeneity. 

We find that the inhomogeneity can influence the cluster qualitatively in two aspects, and the limit shape can fail to capture the full picture of growth at the macroscopic scale. We focus on the columns in the following discussion as analogous remarks hold also for the rows. First, when the smaller column parameters are sufficiently rare, the cluster evolves into an approximately flat shape near the vertical axis. This behavior creates a flat segment in the boundary of the limit shape adjacent to the vertical axis, and has been observed earlier in \cite{emra-16-ecp, sepp-krug}. Second, the cluster can grow at distinct speeds across columns leading to persistent macroscopic variations in its height profile. As a result, after a while the cluster visually resembles a structure with large \emph{crevices} and \emph{spikes}. 
Out of these features, the limit shape only remembers the maximal size of the spikes, and registers this information as a spike (line segment) along the vertical axis emanating from the vertical intercept of its boundary inside the quadrant. A systematic treatment of the formation of spikes and crevices in the CGM seems to be new.

In this paper, we describe the macroscopic evolution of the cluster and elucidate the growth behavior outlined above. 
Our main result provides the first-order asymptotics of the growth process in terms of an explicit variational formula under a mild condition ensuring at least linear growth. With further reasonable assumptions, the formula leads to an exact description of the limit shape, which controls the macroscopic growth asymptotically at sites increasingly away from both axes. The formula also gives the leading order growth of the cluster along a fixed set of columns or rows. This information is, in general, not encoded in the limit shape. 

As an application of the results for the CGM, we also describe the macroscopic evolution of the particles in the associated disordered TASEP. In particular, we derive the flux function and limiting particle profile from the limit shape. 
Subsection \ref{SsResDis} provides a more detailed account of our results.

It has come to our attention that a statement somewhat similar to our main result was conjectured by E.\ Rains in \cite[Conjecture 5.2]{rain-00}, 
 which also contains analogous claims for various other integrable percolation models. Although peripheral to the present work, we reformulate and prove the part of the conjecture pertinent to our setting to highlight the connection.  

A sequel \cite{emra-janj-sepp-19-buse} to the present  paper will study geometric features such as Busemann limits, geodesics, and the competition interface in the inhomogeneous  LPP model, and also apply these results to an inhomogeneous tandem of queues.  

There is also some technical novelty in the treatment of the model. In the present setting, the existence of the limit shape does not follow from standard subadditive arguments and takes up a significant part of the paper to establish. 
We attain this through concentration bounds for the growth process, development of which utilizes explicit increment-stationary versions of the process. 

As proved in \cite{boro-pech}, the CGM studied here is connected to Schur measures \cite{okou-01} and thereby possesses a determinantal structure. In particular, the one-point distribution of the growth process can be written in terms of a Fredholm determinant with an explicit kernel. We take advantage of one feature that follows from this representation, namely, the distributional invariance of the growth process under permutations of columns and rows (stated in Lemma \ref{LGDisPerInv}). Apart from this point, our methodology (described in Section \ref{SsMeth}) does not rely on integrable probability.  

\subsection{Limit shape in the CGM}
\label{SsLimShpCGM}

The general two-dimensional CGM consists of a given collection of nonnegative real-valued random waiting times $\{\omega(i, j): i, j \in \bbZ_{>0}\}$ and a corner growth process $\{\G(i, j): i, j \in \bbZ_{>0}\}$ defined through the recursion  
\begin{align}
\label{EGP}
\G(i, j) = \max \{\one_{\{i > 1\}}\G(i-1, j), \one_{\{j > 1\}}\G(i, j-1)\} + \omega(i, j) \quad \text{ for } i, j \in \bbZ_{>0}. 
\end{align}
This process represents a randomly growing cluster in the first  quadrant of the plane, given at time $t \in \bbR_{\ge 0}$ by 
\begin{align}
\label{EClus}
\Clus(t) = \{(x, y) \in \bbR_{>0}^2: \G(\lc x \rc, \lc y \rc) \le t\}. 
\end{align}
In other words, the unit square $(i-1, i] \times (j-1, j]$ is added to the cluster at time $t = \G(i, j)$ for $i, j \in \bbZ_{>0}$. The closure of \eqref{EClus} in $\bbR_{\ge 0}^2$ is given by 
\begin{align*}
\overline{\Clus(t)} = \{(x, y) \in \bbR_{\ge 0}^2: \G(\lc x \rc + \one_{\{x = 0\}}, \lc y \rc + \one_{\{y = 0\}}) \le t\} \quad \text{ for } t \in \bbR_{\ge 0}.  
\end{align*}
The limit shape of the cluster is defined as the limiting set 
\begin{align}
\label{ELimShpDef}
\LimShp = \lim_{t \rightarrow \infty} t^{-1}\cl{\Clus(t)}, 
\end{align}
with respect to the Hausdorff metric (on nonempty, closed, bounded subsets of $\bbR^2_{\ge 0}$, see Appendix \ref{SsHaus}) provided that the limit exists\footnote{The definition of the cluster in some literature can differ slightly from \eqref{EClus}. For example in \cite{mart-04}, the growth process lives on $\bbZ_{\ge 0}^2$ and the cluster at a given time is defined as a closed subset of $\bbR_{\ge 0}^2$ with the floor function instead of the ceiling function. These variations do not have any impact on the limit shape.}. The definition will be slighly modified in Subsection \ref{SsLimShp} via suitable truncation in the case of superlinear growth in time. 

The first instance of the CGM in the literature had i.i.d.\ exponential waiting times and appeared in connection with a fundamental interacting particle system, TASEP, in a pioneering work \cite{rost} of H.\ Rost. 
Recall that the standard TASEP \cite{spit-70} is a continuous-time Markov process on particle configurations on $\bbZ$ that permit at most one particle per site (exclusion), and evolves as follows: Each particle independently attempts to jump at a common rate $c > 0$ to the next site to its right. Per the exclusion rule, the jump is allowed only if the next site is vacant. The dynamics is unambiguously defined since simultaneous jump attempts a.s.\ never happen. To connect with the CGM, start the TASEP from the step initial condition meaning that the particles occupy the sites of $\bbZ_{\le 0}$ at time zero. Label the particles with positive integers from right to left such that particle $j$ is initially at site $-j+1$ for $j \in \bbZ_{>0}$. Let $\T(i, j)$ denote the time of the $i$th jump of particle $j$, and write  
\begin{align}
\omega'(i, j) = \T(i, j) - \max \{\one_{\{i > 1\}}\T(i-1, j), \one_{\{j > 1\}}\T(i, j-1)\} \quad \text{  for } i, j \in \bbZ_{>0}. \label{ETP}
\end{align} 
By the strong Markov property, $\omega'(i, j) \sim \Exp(c)$ and are jointly independent for $i, j \in \bbZ_{>0}$. Then, since the recursions in \eqref{EGP} and \eqref{ETP} are the same, the $\T$-process is equal in distribution to the $\G$-process defined with i.i.d.\ $\Exp(c)$-distributed waiting times. 

A celebrated result in \cite{rost}, based on the above correspondence with TASEP, identifies the limit shape of the CGM with i.i.d.\ Exp$(c)$ waiting times as the parabolic region given by 
\begin{align}
\label{ERostLimShp}
\LimShp = \{(x, y) \in \bbR_{\ge 0}^2:  \sqrt{x}+\sqrt{y} \le \sqrt{c}\}. 
\end{align}
If the waiting times are i.i.d.\ and geometrically distributed, $\LimShp$ is also explicitly known as a certain elliptic region \cite{cohn-elki-prop-96, jock-prop-shor-98, sepp98mprf}. More generally, for i.i.d.\ waiting times subject to mild conditions, the limit in \eqref{ELimShpDef} still exists and is a concave region with the boundary inside $\bbR_{>0}^2$ extending continuously to the axes \cite{mart-04}. Furthermore, the limit can be characterized in terms of variational formulas over certain infinite dimensional spaces \cite{geor-rass-sepp-16}. 
However, these formulas presently do not yield detailed geometric information about the limit shape except in the above exactly solvable cases. For example, it is unclear precisely when the limit shape has flat segments in the boundary. If the waiting times attain their maxima frequently enough to create an infinite cluster of oriented percolation, the boundary of the limit shape becomes flat in a cone symmetric around the diagonal of the plane \cite{geor-rass-sepp-16}. In the context of 
undirected first-passage percolation, 
 this phenomenon goes back to the classic paper of R.\ Durrett and T.\ Liggett \cite{durr-ligg-81}, and was subsequently studied in \cite{auff-damr-13-ptrf, marc-02}. It is not known whether this is the only mechanism to produce flat segments with i.i.d.\ waiting times. Due to the limited knowledge in the general i.i.d.\ case, a natural starting point as a homogeneous setting for our study into the effects of inhomogeneity is the i.i.d.\ exponential model. 

\subsection{Simulations of flat segments, spikes and crevices}
\label{SsSim}

Varying the rates of the exponential waiting times can create flat spots, spikes and crevices in the evolving shape of the cluster. Let us illustrate these features through some simulations of the CGM deferring their further discussion to Subsection \ref{SsResDis}. 
  
The simulations below share a common sample of independent $\Exp(1)$-distributed waiting times $\{\w(i, j): i, j \in [N]\}$ where $N = 4000$. Each simulation constructs waiting times with specific rates $\lambda_{m, n}(i, j) > 0$ by setting
\begin{align*}
\w_{m, n}(i, j) = \frac{\w(i, j)}{\lambda_{m, n}(i, j)} \sim \Exp(\lambda_{m, n}(i, j)) \quad \text{ for } m, n \in [N], i \in [m], j \in [n]. 
\end{align*}
The value $\G(m, n)$ of the growth process at each site $(m, n) \in [N]^2$ is then computed through \eqref{EGP} with $\w_{m, n}(i, j)$ in place of $\w(i, j)$ for $i \in [m]$, $j \in [n]$. Finally, the cluster $\Clus(t)$ is computed from \eqref{EClus} at time $t = 1000$.  

Figure \ref{FSim} depicts a realization of $\Clus(t)$ together with the boundary of the limit shape approximation $t\LimShp$ in four cases. For comparison, Figure \ref{FRost} covers the homogeneous case where the rates are $1$ and $\LimShp$ is given by \eqref{ERostLimShp} with $c = 1$. In the remaining cases, $\LimShp$ is the subset of 
$\bbR_{\ge 0}^2$ given by 
\begin{align}
\LimShp &= \{x \ge y \text{ and } \sqrt{x} + \sqrt{y} \le 1\} \cup \{x \le y \text{ and } 2(x+y) \le 1\} \cup \{(0, y):  1/2 \le y \le 1\}. \label{ELimShpEg}
\end{align}

\begin{figure}
\begin{subfigure}[t]{.5\textwidth}
  \centering
  \includegraphics[scale=0.6]{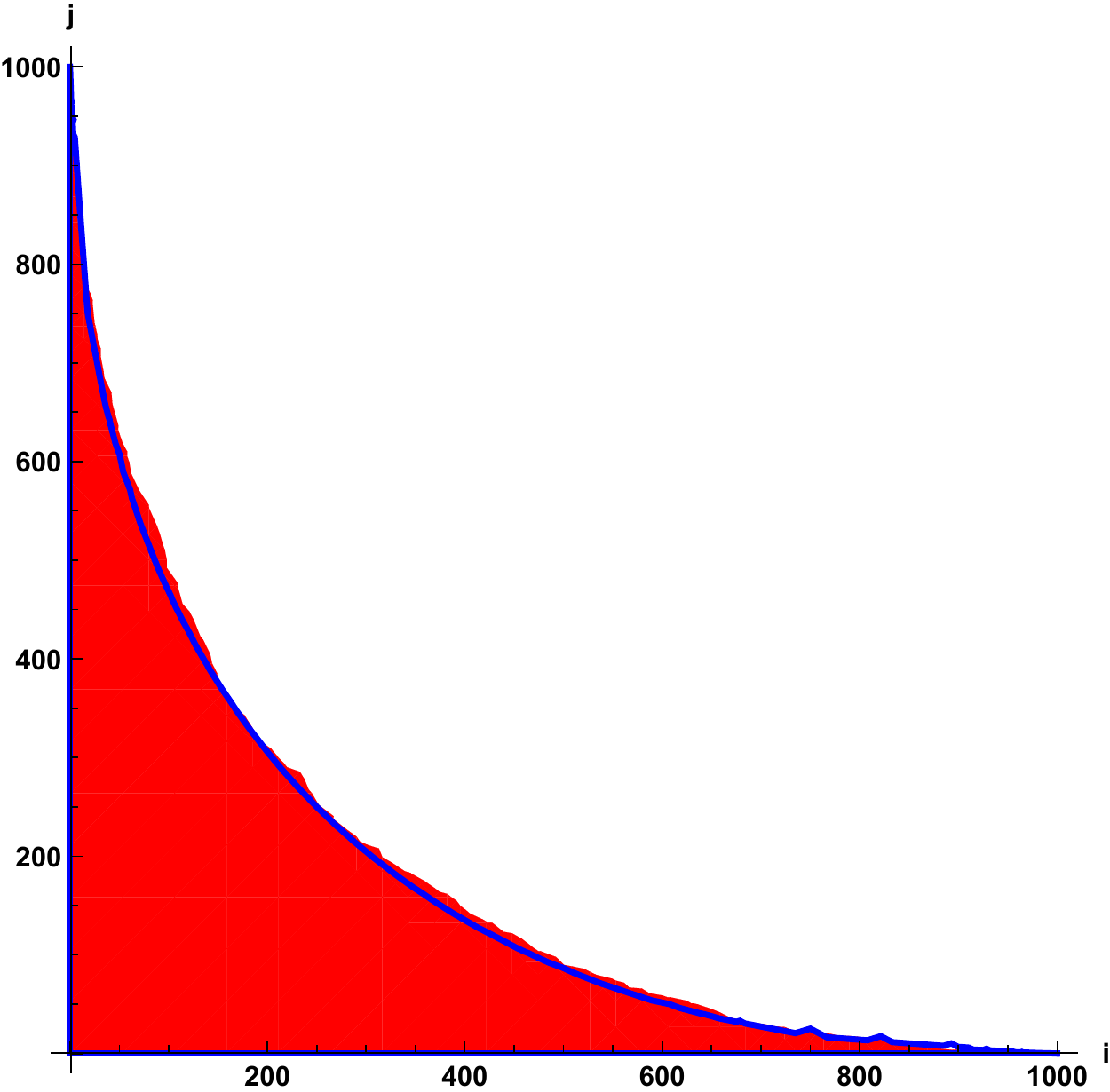}
  \caption{$\lambda_{m, n}(i, j) = 1$.}
  \label{FRost}
\end{subfigure}%
\begin{subfigure}[t]{.5\textwidth}
  \centering
  \includegraphics[scale=0.6]{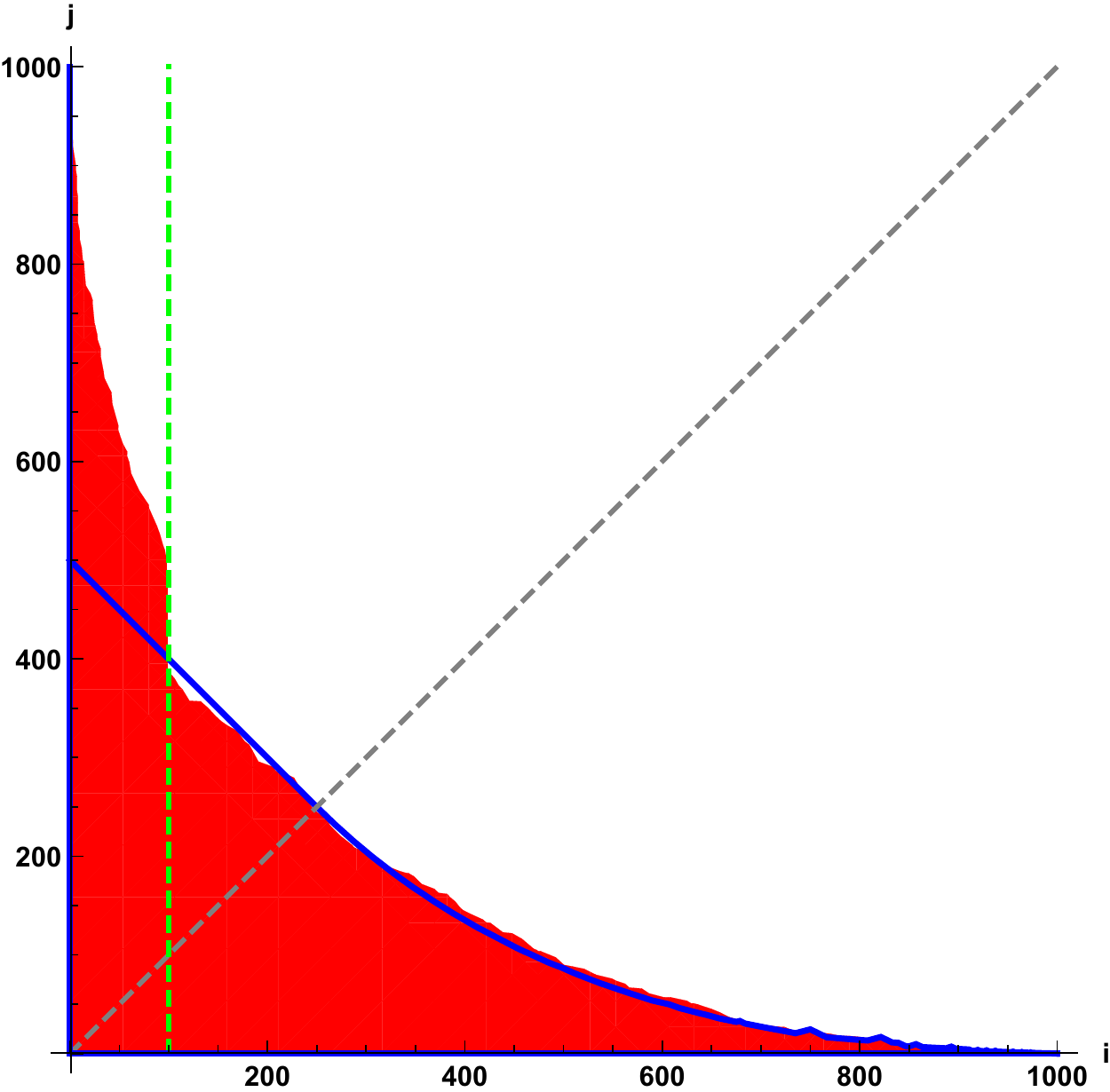}
  \caption{$\lambda_{m, n}(i, j) = \begin{cases}0.5 \quad &\text{ if } i = 100 \\ 1 \quad &\text{ otherwise. }\end{cases}$}
  \label{FFlatSpike}
\end{subfigure}
\begin{subfigure}[t]{.5\textwidth}
  \centering
  \includegraphics[scale=0.6]{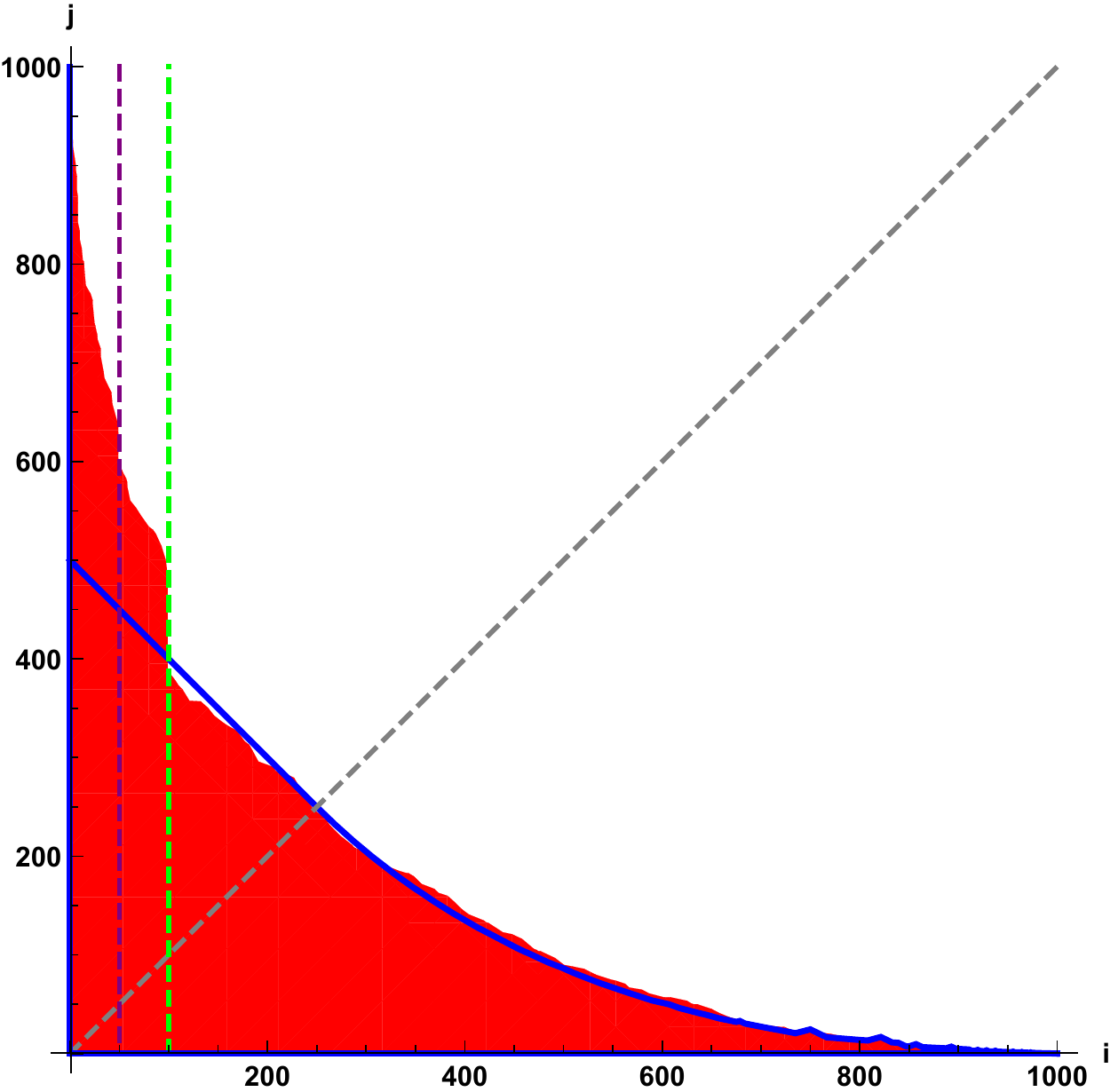}
  \caption{$\lambda_{m, n}(i, j) = \begin{cases}0.75 \quad &\text{ if } i = 50 \\ 0.50 \quad &\text{ if } i = 100 \\ 1 \quad &\text{ otherwise. }\end{cases}$}
  \label{FSpike2}
\end{subfigure}%
\begin{subfigure}[t]{.55\textwidth}
  \centering
  \includegraphics[scale=0.6]{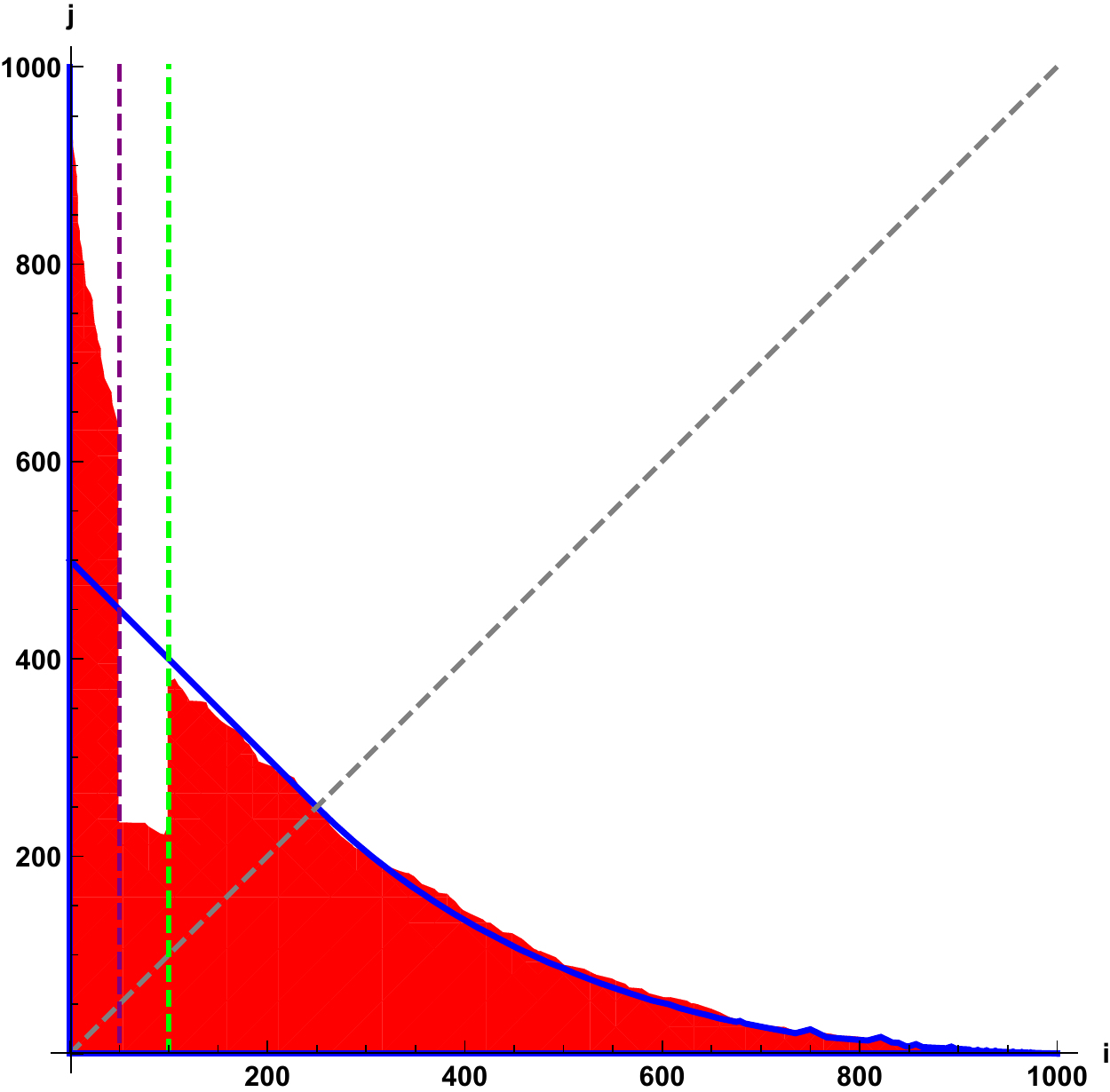}
  \caption{$\lambda_{m, n}(i, j) = \begin{cases}0.25 \quad &\text{ if } i = 50 \text{ and } m < 100 \\ 0.5 \quad &\text{ if } i = 100 \text{ and } m \ge 100 \\ 1 \quad &\text{ otherwise. }\end{cases}$}
  \label{FCrevice}
\end{subfigure}%
\caption{The cluster $\Clus(t)$ (red) and the boundary of the region $t\LimShp$ (blue) at time $t = 1000$ in four simulations of CGM with indicated rates. (a) Homogeneous case. $\LimShp$ is given by \eqref{ERostLimShp} with $c = 1$. (b) Flat spot above the diagonal (dashed gray) to the right of column $100$ (dashed green) and spikes to the left of column $100$. $\LimShp$ is given by \eqref{ELimShpEg} in this and subsequent cases. (c) Larger spikes to the left of column $50$ (dashed purple) and smaller spikes between columns $50$ and $100$. (d) Crevices between columns $50$ and $100$. }
\label{FSim}
\end{figure}

\subsection{Exponential CGM with inhomogeneous rates}
\label{SsiCGM}

In the basic version of our setting, $\cw{i, j}{} \sim \Exp(a_i+b_j)$ for $i, j \in \bbZ_{>0}$ for some real parameter sequences $\a = (a_i)_{i \in \Z_{>0}}$ and $\b = (b_j)_{j \in \Z_{>0}}$. To have positive rates the parameters are assumed to satisfy $a_i + b_j > 0$ for $i, j \in \bbZ_{>0}$. The earliest appearances of this CGM were perhaps in \cite{boro-pech, joha-08}. The model arises via a limit transition from the CGM considered earlier in \cite{joha00b}. The latter has independent geometric waiting times with multiplicatively separable inhomogeneity in fail parameters and comes from a Schur measure \cite{okou-01}. As proved in \cite{boro-pech}[Theorem 1], the present model is closely linked to the complex Wishart ensemble 
(also known as the Laguerre ensemble) in the sense that the square root of the largest singular value of a natural generalization of an $m \times n$ sized realization of this ensemble has the same distribution as $\G(m, n)$ for $m, n \in \bbZ_{>0}$. This correspondence was observed earlier in \cite[Proposition 6.1]{baik-bena-pech-05} when $\a$ or $\b$ is a constant sequence, and generalized later to the process level in \cite{diek-warr-09}. 

The model can also be naturally motivated as a TASEP with the step initial condition, and particlewise and holewise disorder. The disorder in rates translates to the feature that the attempts for the $i$th jump of particle $j$ occur at rate $a_i+b_j$ for $i, j \in \bbZ_{>0}$. For an alternative viewpoint, imagine the holes (empty sites) as another class of particles labeled with positive integers such that hole $i$ is at site $i$ for $i \in \bbZ_{>0}$ at time zero. Hole $i$ moves by exchanging positions with the particle to its immediate left at rate $a_i$ and particle $j$ moves by doing the same with the holes to its immediate right at rate $b_j$ for $i, j \in \bbZ_{>0}$. Hence, when they encounter, hole $i$ and particle $j$ exchange positions at net rate $a_i+b_j$, and this exchange is precisely the $i$th jump of particle $j$ for $i, j \in \bbZ_{>0}$. 

In this paper, we consider the following slightly more general setting that will permit us to simultaneously treat the models alluded to in items  (\romannumeral1)-(\romannumeral6) in Section \ref{SsCont}. 
Fix two collections of real parameters 
\begin{align}\a = \{\ca{m}{i}: m \in \bbZ_{>0} \text{ and } i \in [m]\} \quad \text{ and } \quad \b = \{\cb{n}{j}: n \in \bbZ_{>0} \text{ and } j \in [n]\}.\label{EPar}\end{align} 
Abbreviate $\ca{m}{} = (\ca{m}{i})_{i \in [m]}$ and $\cb{n}{} = (\cb{n}{j})_{j \in [n]}$ for $m, n \in \bbZ_{>0}$. Assume that 
\begin{align}
\ca{m}{i}+\cb{n}{j} > 0 \quad \text{ for } m, n \in \bbZ_{>0} \text{ and } i \in [m], j \in [n]. \label{AsmParam}
\end{align} 
For each $m, n \in \bbZ_{>0}$, let 
\begin{align}
\cwp{i, j}{m, n} \sim \Exp(\a_{m}(i) + \b_{n}(j)) \text{ and jointly independent for } i \in [m] \text{ and } j \in [n], \label{EWP}
\end{align} 
and define $\cGp{m, n}$ from these waiting times by \eqref{EGP}. Then, for $t \in \bbR_{\ge 0}$, define the region $\cClus{t}$ from the growth process $\{\cGp{m, n}: m, n \in \bbZ_{>0}\}$ by \eqref{EClus}. A few points worth emphasizing: The $\cGp{}$-process itself does not necessarily satisfy the recursion \eqref{EGP} because the waiting times in \eqref{EWP} are allowed to vary with $m, n$. By the same token, $\cGp{}$-process need not be coordinatewise nondecreasing. Therefore, $\cClus{t}$ may no longer be a connected subset of $\bbR_{>0}^2$ although we shall continue to call it a cluster. 

As corollaries we obtain results for the height process and cumulative particle current (flux process) of TASEP. 
 In terms of the CGM these are  defined respectively by
\begin{align}
\cHp{n}{t} &= \max \{\sup\{m \in \bbZ_{>0}: \cGp{m, n} \le t\}, 0\} \label{EHp} \\
\cFluxp{m}{t} &= \max \{\sup \{n \in \bbZ_{>0}: \cGp{m+n-1, n} \le t, 0\} \label{EFp}
\end{align}
for $m, n \in \bbZ_{>0}$ and $t \in \bbR_{\ge 0}$. In the absence of $m, n$-dependence in \eqref{EWP} for $m, n \in \bbZ_{>0}$, \eqref{EHp} gives the number of jumps executed by particle $n$ by time $t$ and also the \emph{height} (namely, length) of the $n$th row of the cluster at time $t$, while \eqref{EFp} counts the number of particles that have jumped from site $m-1$ to site $m$ by time $t$.  
These interpretations, although not valid in the full generality of \eqref{EWP}, justify the names of the processes in \eqref{EHp}--\eqref{EFp}. 


\subsection{Discussion of the main results}
\label{SsResDis}
The main contributions of this paper are exact first-order asymptotics of the growth process that lead to fairly explicit descriptions of the limit shape and the limiting flux function. Precise results are stated in Section \ref{SRes}. For the moment, we summarize some key points. 

\emph{An explicit centering for the growth process (Theorem \ref{TEmpShp}).} The central result of the paper
 computes an explicit, deterministic approximation to the first order (a \emph{centering} for short, see Definition \ref{DCent} below for the precise meaning) for the growth process under a mild growth condition on the means of the waiting times. More specifically, assuming that $\min \ca{m}{}+\min \cb{n}{}$ does not decay too fast as $m+n$ grows, 
\begin{align}
\cGp{m, n} \stackrel{\text{a.s.}}{\approx} \inf_{-\min \ca{m}{} < z < \min \cb{n}{}} \bigg\{\sum_{i=1}^m \frac{1}{\ca{m}{i}+z} + \sum_{j=1}^n \frac{1}{\cb{n}{j}-z}\bigg\} \quad \text{ for large } m+n. \label{ECent}
\end{align}
Here, $z$ serves as a convenient parameter indexing the increment-stationary versions of the growth process. 
This result is obtained by first developing summable concentration bounds for the growth process. 

\emph{Resolution of a conjecture due to E.\ Rains (Theorem \ref{TRains}).} A formula similar to \eqref{ECent} appeared in \cite{rain-00} within the continuous counterpart of Conjecture 5.2, which is not stated explicitly but can be discerned from the context. We state the part of the conjecture related to the present model and prove it by means of concentration bounds. 

\emph{Shape function (Theorem \ref{TShpFun}).} 
Assume further that the running minima $\min \ca{n}{}$ and $\min \cb{n}{}$ converge to some $\aml, \bml \in \bbR \cup \{\infty\}$ with $\aml + \bml > 0$, respectively, and the empirical distributions associated with $\ca{n}{}$ and $\cb{n}{}$ converge vaguely  to some subprobability measures $\alpha$ and $\beta$, respectively, on $\bbR$ as $n \rightarrow \infty$.  
Then \eqref{ECent} leads to the simpler a.s.\ approximation 
\begin{align}
\cGp{m, n} \stackrel{\text{a.s.}}{\approx}  \inf_{-\aml < z < \bml} \bigg\{m\int_{\bbR} \frac{\alpha(\dd a)}{a+z} + n\int_{\bbR} \frac{\beta(\dd b)}{b-z}\bigg\} \quad \text{ for large } m, n. \label{EShpF}
\end{align}
The centering in \eqref{EShpF} is unique with the property that it extends to a positive-homogeneous and continuous function on $\bbR_{\ge 0}^2$. This 
extension is the shape function (see Definition \ref{DShpFun}) of the growth process. In many variants of the exponential CGM from the literature, the shape function can be either represented as or derived from \eqref{EShpF}, see \cite[Section 3.9]{emra-janj-sepp-19-shp-long} for numerous corollaries to this effect. 

Theorem \ref{TShpFun} is a considerable strengthening of \cite[Theorem 2.1]{emra-16-ecp} which derived the approximation \eqref{EShpF} as $m, n$ grow large along a fixed direction and when the parameters in \eqref{EPar} are not $m, n$-dependent and are randomly chosen subject to a joint ergodicity condition. This condition enabled  \cite{emra-16-ecp} to utilize subadditive ergodic theory to obtain the existence of the shape function and then compute it through convex analysis from the shape functions of the increment-stationary growth processes. An obstruction to implementing this approach in the present setting is that the waiting times in \eqref{EWP} are not stationary with respect to lattice translations and, therefore, the existence of the shape function is no longer guaranteed by standard subadditive ergodic theory. 

\emph{Growth near the axes (Theorem \ref{TNarShp}).} Another consequence of \eqref{ECent} is that 
\begin{align}
\cGp{m, n} \stackrel{\text{a.s.}}{\approx} 
\begin{cases} 
n \int_{\bbR}  (b+\min \ca{m}{})^{-1} \beta(\dd b)&\quad \text{ when $n$ is large and $m/n$ is small} \\ m \int_{\bbR}  (a+\min \cb{n}{})^{-1} \alpha(\dd a)&\quad \text{ when $m$ is large and $n/m$ is small} \end{cases} \label{ENarShp}
\end{align}
provided that $\inf \a + \inf \b > 0$ and the appropriate half of the vague convergence assumption above holds. In particular, \eqref{ENarShp} describes the asymptotics of the growth process along a fixed column or row. The result demonstrates the possibility of macroscopically uneven growth in the cluster, for example, across columns and reveals the underlying reason for this as the variations in the $\min \ca{m}{}$ sequence. Near the axes, the right-hand side of \eqref{EShpF} is approximately given by 
\begin{align*}
\begin{cases} n \int_{\bbR}  (b+\aml)^{-1} \beta(\dd b)&\quad \text{ when $n$ is large and $m/n$ is small} \\ m \int_{\bbR}  (a+\bml)^{-1} \alpha(\dd a)&\quad \text{ when $m$ is large and $n/m$ is small} \end{cases} 
\end{align*}
in contrast with \eqref{ENarShp}. Discrepancies in these approximations are manifested as macroscopic spikes and crevices in the cluster relative to the boundary of the limit shape near the axes as shown in Figure \ref{FSpikeCrevice}. See Subsection \ref{SsSpikes} for a precise calculation in support of the figure. 

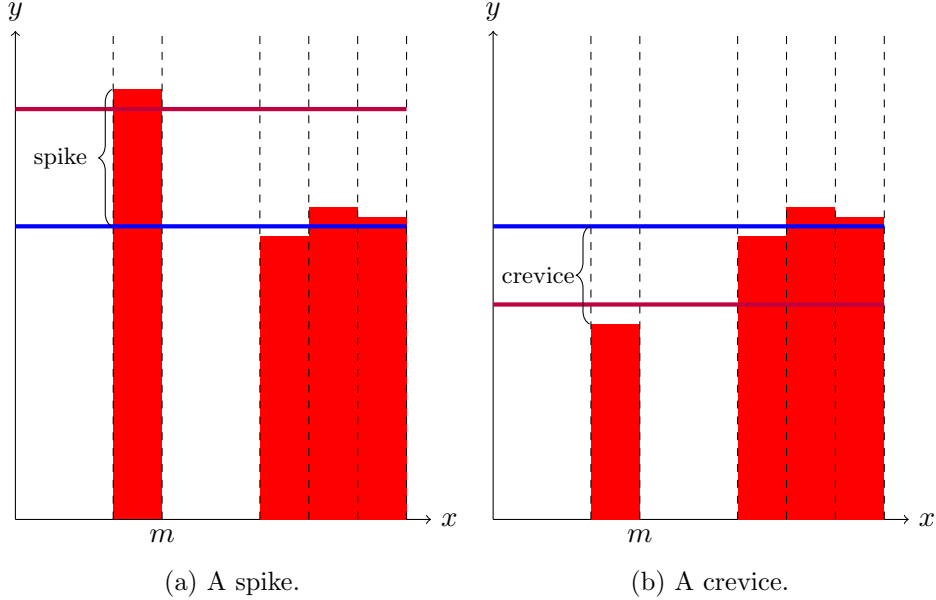
\begin{figure}[!h]  
\centering 
 \begin{subfigure}[b]{0.4\linewidth}
\centering    
\begin{tikzpicture}[scale=1.3]   
      \draw[->] (0,0) -- (4.25,0) node[right] {$x$};  
      \draw[->] (0,0) -- (0,5) node[above] {$y$};
\fill[red] (1, 0)rectangle(1.5, 4.4); 
\fill[red] (2.5, 0)rectangle(3, 2.9);
\fill[red] (3, 0)rectangle(3.5, 3.2);
\fill[red] (3.5, 0)rectangle(4.0, 3.1);
\draw[dashed](1, 0)--(1, 5);  
\draw[dashed](1.5, 0)node[below]{$m$}--(1.5, 5); 
\draw[dashed](2.5, 0)--(2.5, 5);  
\draw[dashed](3, 0)--(3, 5); 
\draw[dashed](3.5, 0)--(3.5, 5);  
\draw[dashed](4, 0)--(4, 5);
\draw[blue, ultra thick](0, 3)--(4, 3);
\draw[purple, ultra thick](0, 4.2)--(4, 4.2);
\draw [decorate,decoration={brace,amplitude=5pt, raise=0.5pt},yshift=0pt]
(1, 3) -- (1,4.4) node [black,midway,xshift=-0.7cm] {\footnotesize
spike};
    \end{tikzpicture}%
    \caption{A spike. } \label{FSpik}  
  \end{subfigure}
\begin{subfigure}[b]{0.4\linewidth}
\centering
    \begin{tikzpicture}[scale=1.3]   
      \draw[->] (0,0) -- (4.25,0) node[right] {$x$};  
      \draw[->] (0,0) -- (0,5) node[above] {$y$};
\fill[red] (1, 0)rectangle(1.5, 2); 
\fill[red] (2.5, 0)rectangle(3, 2.9);
\fill[red] (3, 0)rectangle(3.5, 3.2);
\fill[red] (3.5, 0)rectangle(4.0, 3.1);
\draw[dashed](1, 0)--(1, 5);  
\draw[dashed](1.5, 0)node[below]{$m$}--(1.5, 5); 
\draw[dashed](2.5, 0)--(2.5, 5);  
\draw[dashed](3, 0)--(3, 5); 
\draw[dashed](3.5, 0)--(3.5, 5);  
\draw[dashed](4, 0)--(4, 5);
\draw[blue, ultra thick](0, 3)--(4, 3);
\draw[purple, ultra thick](0, 2.2)--(4, 2.2);
\draw [decorate,decoration={brace,amplitude=5pt,raise=0.5pt},yshift=0pt]
(1, 2) -- (1,3) node [black,midway,xshift=-0.7cm] {\footnotesize
crevice};
    \end{tikzpicture}%
\caption{A crevice.} \label{FCrev}  
\end{subfigure}
\caption{An illustration of the cluster along column $m \in \bbZ_{>0}$ and after columns with much larger indices at time $t$ (red) in the case $\min \bfa_m \neq \aml$. The lines $y \int_{\bbR} (b+\aml)^{-1}\beta(\dd b) = t $ (blue) and $y\int_{\bbR} (b+\min \a_m)^{-1}\beta(\dd b) = t$ (purple) are shown. (a) A spike forms when $\min \bfa_m > \mathfrak{a}$ (b) A crevice forms when $\min \bfa_m < \mathfrak{a}$. }
\label{FSpikeCrevice}
\end{figure}

\emph{Limit shape (Theorem \ref{TLimShp}).} With the aid of \eqref{EShpF}-\eqref{ENarShp} and assuming $\alpha, \beta \neq 0$, the limit shape (in the sense of \eqref{ELimShpDef}) can be identified as the union of the sublevel set 
\begin{align}\bigg\{(x, y) \in \bbR_{\ge 0}^2: \inf_{z \in (-\aml, \bml)} \bigg\{x \int_{\bbR} \frac{\alpha(\dd a)}{a+z} + y \int_{\bbR} \frac{\beta(\dd b)}{b-z} \bigg\} \le 1\bigg\} \label{ELimShpBlk}\end{align} 
of the shape function, and the line segments 
\begin{align}
\bigg\{(x, 0) \in \bbR_{\ge 0}^2: x\int_{\bbR} \dfrac{\alpha(\dd a)}{a+\bms} \le 1\bigg\} \quad \text{ and } \quad \bigg\{(0, y) \in \bbR_{\ge 0}^2: y \int_{\bbR} \dfrac{\beta(\dd b)}{b+\ams} \le 1\bigg\}, \label{ELimShpAxes}
\end{align}
where $\ams = \sup_{m \in \bbZ_{>0}} \min \ca{m}{}$ and $\bms = \sup_{n \in \bbZ_{>0}} \min \cb{n}{}$. See Figure \ref{FLimShp} for an illustration. (When $\alpha = 0$ or $\beta = 0$, the above set is unbounded but can still be viewed as the limit shape in a truncated sense, see Subsection \ref{SsLimShp}). Computations behind the subsequent discussion are either omitted or postponed to Subsections \ref{SsLimShp}-\ref{SsFlat}. 
The statements pertinent to the vertical axis have obvious analogues for the horizontal axis. 

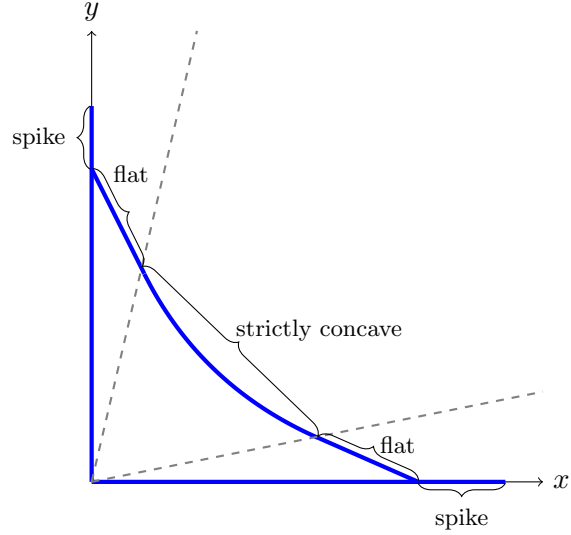
\begin{figure}[!h]
\centering
\begin{tikzpicture}
\draw[->] (0,0) -- (6,0) node[right] {$x$};  
\draw[->] (0,0) -- (0,6) node[above] {$y$};
\draw[ultra thick, blue] (0, 0)--(5.5, 0);
\draw[ultra thick, blue] (0, 0)--(0, 5);
\draw[blue, ultra thick]   plot[smooth,domain=0.7:3] (\x, {(2.5-sqrt(\x))^2});
\draw[blue, ultra thick]   plot[smooth,domain=0:0.7] (\x, {-1.99*(\x-0.7)+2.77});
\draw[blue, ultra thick]   plot[smooth,domain=0:0.6] ({-2.27*(\x-0.6)+2.98}, \x);
\draw[thick, gray, dashed] (0, 0)--(3, 0.6)--(6, 1.2);
\draw[thick, gray, dashed] (0, 0)--(0.7, 3)--(1.4, 6);
\draw [decorate,decoration={brace,amplitude=5pt, raise=0.5pt},yshift=0pt]
(0, 4.16) -- (0,5) node [black,midway,xshift=-0.7cm] {\footnotesize
spike};
\draw [decorate,decoration={brace,amplitude=5pt, raise=0.5pt, mirror},xshift=0pt]
(4.34, 0) -- (5.5,0) node [black,midway,yshift=-0.5cm] {\footnotesize
spike};
\draw [decorate,decoration={brace,amplitude=5pt, raise=0.5pt},xshift=0pt]
(0.65, 2.85) -- (3,0.59) node [black,midway,yshift=0.3cm, xshift=1.2cm] {\footnotesize
strictly concave};
\draw [decorate,decoration={brace,amplitude=5pt, raise=0.5pt, mirror},xshift=0pt]
(0.65, 2.85) -- (0, 4.16) node [black,midway,yshift=0.6cm, xshift=0.2cm] {\footnotesize
flat};
\draw [decorate,decoration={brace,amplitude=5pt, raise=0.5pt, mirror},xshift=0pt]
(4.34, 0) -- (3,0.59) node [black,midway,yshift=0.2cm, xshift=0.4cm] {\footnotesize
flat};
\end{tikzpicture}
\caption{An illustration of the boundary of the limit shape (blue) and the boundary of the region \eqref{ECurvReg} (dashed gray). The strictly concave part and the (possibly empty) flat segments and spikes of the limit shape are indicated.}
\label{FLimShp}
\end{figure}

The boundary of the limit shape inside $\bbR_{>0}^2$ connects to the axes at the points 
\begin{align}
\bigg(\bigg\{\int_{\bbR} \frac{\alpha(\dd a)}{a+\bml}\bigg\}^{-1}, 0\bigg) \quad \text{ and } \quad \bigg(0, \bigg\{\int_{\bbR} \frac{\beta(\dd b)}{b+\aml}\bigg\}^{-1}\bigg). \label{EIntcp}
\end{align}
Comparing with \eqref{ELimShpAxes} shows that the limit shape has a vertical \emph{spike}, namely, a vertical line segment above the second intercept in \eqref{EIntcp}, if and only if $\ams > \aml$. The latter is the precise condition for the occurrence of a vertical spike in the cluster. Thus, the limit shape retains some residual memory of the spikes in the cluster by encoding their maximal size (to the first-order asymptotics) as the lengths of its spikes. However, the crevices and non-maximal spikes of the cluster, despite being persistent macroscopic scale structures, are not visible in the limit shape.

The boundary of the limit shape can be explicitly parametrized. It is curved (strictly concave) inside the nonempty conic region given by 
\begin{align}
\bigg\{(x, y) \in \bbR_{>0}^2: x \int_{\bbR} \frac{\alpha(\dd a)}{(a-\aml)^{2}} > y \int_{\bbR} \frac{\beta(\dd b)}{(b+\aml)^2} \ \text{ and } \ x\int_{\bbR} \frac{\alpha(\dd a)}{(a+\bml)^2} < y \int_{\bbR} \frac{\beta(\dd b)}{(b-\bml)^2}\bigg\}, \label{ECurvReg}
\end{align}
and is flat elsewhere. In particular, the boundary has a flat segment inside $\bbR_{>0}^2$ adjacent to the vertical axis if and only if $\int_{\bbR} (a-\aml)^{-2}\alpha(\dd a) < \infty$. This condition indicates that the small parameters in $\ca{m}{}$ become sufficiently infrequent as $m \rightarrow \infty$. It holds precisely when $\aml < \inf \supp \alpha$ or $\int_{\bbR} (a-\inf \supp \alpha)^{-2}\alpha(\dd a) < \infty$ since $\aml \le \inf \supp \alpha$. 
The formation of flat segments can be understood geometrically in terms of the coalescence of geodesics in the associated LPP model, which we leave to the sequel \cite{emra-janj-sepp-19-buse}. 

Formulas \eqref{ELimShpBlk}-\eqref{ELimShpAxes} illuminate how the inhomogeneity introduced through the parameters $\a$ and $\b$ at the microscopic scale propogates to the limit shape. This happens by means of three partially independent mechanisms: the limiting empirical measures $\alpha, \beta$, the limiting running minima $\aml, \bml$ and the maximal running minima $\ams, \bms$. The dependence on the running minima implies that changing the means of the waiting times in a single column or row can alter the limit shape. This feature is reminiscent of the sensitivity of the flux function to a slow bond in TASEP \cite{basu-sido-sly-16}. 




\emph{The limiting height and flux functions (Theorem \ref{TLimFlux}).} The knowledge of the shape function also leads to following centerings for the height and flux processes. 
\begin{align*}
\cHp{n}{t} &\stackrel{\text{a.s.}}{\approx} \max \bigg\{\sup_{z \in (-\aml, \bml)}\bigg\{\frac{t-n\int_{\bbR} (b-z)^{-1}\beta(\dd b)}{\int_{\bbR} (a+z)^{-1}\alpha(\dd a)}\bigg\}, 0\bigg\} \\
\cFluxp{m}{t}&\stackrel{\text{a.s.}}{\approx} \max \bigg\{\sup_{z \in (-\aml, \bml)}\bigg\{\frac{t-m\int_{\bbR}(a+z)^{-1}\alpha(\dd a)}{\int_{\bbR} (a+z)^{-1}\alpha(\dd a)+\int_{\bbR}(b-z)^{-1}\beta(\dd b)}\bigg\}, 0\bigg\}, 
\end{align*}
for sufficiently large $m, n$ and all $t$. The formulas above are obtained from \eqref{EShpF} assuming further that the measures $\alpha, \beta$ are nonzero. 

%

\emph{Height of a fixed row (Theorem \ref{TNarHeight}).} Similarly, it follows from \eqref{ENarShp} that 
\begin{align*}\cHp{n}{t} \stackrel{\text{a.s.}}{\approx} t\bigg\{\int_{\bbR} \frac{\alpha(\dd a)}{a+\min \cb{n}{}}\bigg\}^{-1} \quad \text{ for fixed $n$ and large $t$.} \end{align*}

\subsection{Methodology}
\label{SsMeth}

We study the growth process $\cGp{}$ through couplings with its increment-stationary versions $\csGp{}{z}$ indexed by the $z$-parameter in \eqref{ECent}. The horizontal $\csGp{}{z}$-increments are independent along any row and exponentially distributed with explicit rates that are invariant under vertical translations, and an analogous statement holds for the vertical increments. This feature is sometimes referred to as the \emph{Burke property} in reference to Burke's theorem from queueing theory since the increments correspond to inter-arrival/departure times of customers in an interpretation of the CGM as M/M/1 queues in tandem. The details can be found, for example, in \cite[Section 7]{mart-06}. 

Due to the distributional structure of the increments,  
the $\csGp{}{z}$-process concentrates around the expression inside the infimum in \eqref{ECent} with overwhelming probability. Utilizing the coupling with $z$ chosen as the unique minimizer $\EMin$ in \eqref{ECent}, 
we then derive similar concentration bounds for the $\G$-process. The right tail bound comes easily since the $\cGp{}$-process is dominated by the $\csGp{}{z}$-process for each $z$. For the left tail bound, we first show that in the LPP representation of $\csGp{}{\EMin}$ the geodesic from the origin to $(m, n)$ exits the boundary close to the origin with overwhelming probability assuming monotonicity of the parameters. In this case, a left tail bound for $\cGp{m, n}$ can be extracted from that of $\csGp{}{\EMin}$. 
On the other hand, the distributional invariance of the $\cGp{m, n}$ under permutations of the parameters 
implies that the bound continues to hold without the monotonicity condition. The bounds obtained in this manner are not sharp but suffice for the purposes of first-order asymptotics. 

In the context of percolation and directed polymer models, the idea of coupling with increment-stationary processes to compute limit shapes dates back to \cite{sepp98mprf}.

An alternative path to the results proved in this work is to utilize the determinantal structure in the model. For example, formula \eqref{ECent} can be predicted from the correlation kernel. To obtain asymptotics in the strength of the present work, one would likely still turn to summable tail bounds for the growth process. Developing such bounds from the correlation kernel appears more involved than the more elementary arguments used here.

\subsection{Outline}
\label{SsOut}

The remainder of this text is organized as follows. Section \ref{SiLPP} casts the growth process as an LPP process with inhomogeneous exponential weights. This section also constructs the TASEP with the step initial condition and disorder in particles and holes from the growth process. The main results are formulated precisely in Section \ref{SRes}. Concentration bounds for the growth process are developed in Section \ref{SConBdLPP}. The centerings \eqref{ECent} and \eqref{EShpF} are derived in Sections \ref{SEmpShp} and \ref{SPrShpFun}, respectively. Section \ref{SThinRec} obtains approximations to the growth process near the axes. Section \ref{SPrLimShp} computes the limit shape. Section \ref{SPrFlux} computes the limiting flux and height functions for the disordered TASEP. 
Some standard and auxiliary facts are recorded in Appendix \ref{SApdx}. 

\subsection{Notation and conventions}
\label{SsNot}

Let $\bbZ$, $\bbQ$, $\bbR$, and $\bbC$ denote the spaces of integers, rational, real and complex numbers, respectively. For $a \in \bbR$, define $\bbZ_{\ge a} = \{i \in \bbZ: i \ge a\}$ and make analogous definitions if the set is replaced with $\bbR$ or the subscript is replaced with $> a$, $\le a$ or $< a$. Write $\emptyset$ for the empty set. For $n \in \Z_{>0}$, $[n]=\{1,2,\dotsc,n\}$ with the convention that $[n]=\varnothing$ for $n\in\Z_{\le0}$. For $x \in \bbR$, $\lc x \rc = \inf \bbZ_{\ge x}$ and $x_+ = \max \{x, 0\}$. 

For a real sequence $(c_i)_{i \in \bbZ_{>0}}$, write $c_{p, n}^{\min} = \min \{c_i: i \in [n] \smallsetminus [p-1]\}$ for $n \in \bbZ_{>0}$ and $p \in [n]$, and abbreviate $c_{1, n}^{\min} = c_n^{\min}$ for $n \in \bbZ_{>0}$. For $k \in \bbZ_{\ge 0}$, denote by $\tau_k$ the shift map $(c_i)_{i \in \bbZ_{>0}} \mapsto (c_{i+k})_{i \in \bbZ_{>0}}$. 

A function $f: \bbR^2_{>0} \rightarrow \bbR$ is positive-homogeneous if $f(cx, cy) = cf(x, y)$ for $x, y, c > 0$. Being an \emph{increasing} or \emph{decreasing} function is understood in the strict sense.  For any set $X$ and subset $A \subset X$, write $\one_{A}$ for the indicator function of $A$ that equals $1$ on $A$ and $0$ on the complement $X \smallsetminus A$. 
For any function $f: A \rightarrow \bbR \cup \{\infty, -\infty\}$, the product $\one_{A}f$ equals $f$ on $A$ and $0$ on $X \smallsetminus A$. If $X$ is a topological space, $\cl{A}$ denotes the closure of $A$ in $X$. 

The support of a Borel measure $\mu$ on $\bbR$ is the set $\supp \mu = \bbR \smallsetminus U$ where 
$U \subset \bbR$ is the largest open set with $\mu(U) = 0$. For $\lambda \in \bbR_{>0}$, the exponential distribution with rate $\lambda$, denoted $\Exp(\lambda)$, is the Borel measure on $\bbR$ with density $x \mapsto \one_{\{x > 0\}} \lambda e^{-\lambda x}$. Its mean and variance are $\lambda^{-1}$ and $\lambda^{-2}$, respectively. The statement $X \sim \Exp(\lambda)$ means that the random variable $X$ is $\Exp(\lambda)$-distributed. For $x \in \bbR$, the Dirac measure $\delta_{x}$ is the Borel probability measure on $\bbR$ such that $\delta_x\{x\} = 1$. 

A sequence of events $(E_n)_{n \in \bbZ_{>0}}$ in a probability space occurs with overwhelming probability if for any $p \in \bbZ_{>0}$ the probability of $E_n$ is at least $1-C_pn^{-p}$ for $n \in \bbZ_{>0}$ for some constant $C_p > 0$ dependent only on $p$. 

\subsection*{Acknowledgement}
The authors are grateful to an anonymous referee for helpful comments. 

\section{LPP with inhomogeneous exponential weights}
\label{SiLPP}

In this section, we reintroduce the model from the percolation perspective and mention its special features due to the exponential weights (waiting times) that contribute to our analysis. We also discuss the disordered TASEP associated with the growth process. 

\subsection{Last-passage times, geodesics and exit points}

A finite sequence $\pi = (\pi_i)_{i \in [p]}$ in $\bbZ^2$ is an \emph{up-right path} if $\pi_{i}-\pi_{i-1} \in \{(1, 0), (0, 1)\}$ for $1 < i \le p$. For $k, l, m, n \in \bbZ$, write $\Pi_{k, l}^{m, n}$ for the set of all up-right paths with $\pi_1 = (k, l)$ and $\pi_p = (m, n)$. Let $\Blk = \bbR^{\bbZ_{>0}^2}$ and $\BlkBd = \bbR^{\bbZ_{\ge 0}^2}$. Define the \emph{last-passage times} on $\Blk$ by 
\begin{align}
\cGi{m, n}{k, l} &= \max \limits_{\pi \in \Pi_{k, l}^{m, n}} \sum_{(i, j) \in \pi} \cw{i, j} \quad \text{ for } m, n, k, l \in \bbZ_{>0} \text{ and } \w \in \Blk, \label{ELPT}
\end{align}
and on $\BlkBd$ by
\begin{align}
\csGi{m, n}{k, l} &= \max \limits_{\pi \in \Pi_{k, l}^{m, n}} \sum_{(i, j) \in \pi} \csw{i, j} \quad \text{ for } k, l, m, n \in \bbZ_{\ge 0} \text{ and } \csw{} \in \BlkBd. \label{ELPT2} 
\end{align}
We  work with $k \le m$ and $l \le n$ in the sequel, in which case $\Pi_{k, l}^{m, n}$ is nonempty and the maxima above are finite. Any maximizer $\pi \in \Pi_{k, l}^{m, n}$ in \eqref{ELPT} or \eqref{ELPT2} is called a \emph{geodesic}. Being finite and nonempty, $\Pi_{k, l}^{m, n}$ contains at least one geodesic. We abbreviate $\cG{m, n} = \cGi{m, n}{1, 1}$ (consistently with \eqref{EGP}) and  $\csG{m, n} = \csGi{m, n}{0, 0}$.  

For $m, n \in \Z_{>0}$, define the \emph{horizontal} and \emph{vertical exit points} by 
\begin{align}
\cHE{m, n} &= \max \{i \in \bbZ_{\ge 0}: i \le m \text{ and } \csG{m, n} = \csGi{i, 0}{} + \csGi{m, n}{i, 1}\}\label{EhorExit}\\
\cVE{m, n} &=  \max \{j \in \bbZ_{\ge 0}: j \le n \text{ and } \csG{m, n} = \csGi{0, j}{} + \csGi{m, n}{1, j}\},  \label{EverExit}
\end{align}
respectively.
$\cHEi{m, n}{}$ is the maximal $i \in \{0, \dotsc, m\}$ such that $(i, 0) \in \pi$ for some geodesic $\pi \in \Pi_{0, 0}^{m, n}$, and then  $(i, 0)$ is the site where $\pi$ exits the horizontal boundary $\bbZ_{\ge 0} \times \{0\}$. Likewise for $\cVEi{m, n}{}$. 
\subsection{Bulk LPP process}
\label{SsbLPP}

Let $\cP$ denote the Borel probability measure on $\Blk$ under which
\begin{align*}
\{\cw{i, j}: i, j \in \Z_{>0}\} \text{ are independent and } \cw{i, j} \sim \Exp(1) \text{ for } i, j \in \bbZ_{>0}. 
\end{align*}
Write $\cE$ for the corresponding expectation. For $m, n \in \bbZ_{>0}$, let $\{\cwph{i, j}{m, n}: (i, j) \in [m] \times [n]\}$ be a collection of independent $\Exp(1)$-distributed random variables on the probability space $(\Blk, \cP)$. No assumption is made about the joint distribution of $\cwph{i, j}{m ,n}$ and $\cwph{i', j'}{m', n'}$ if $(m, n) \neq (m', n')$.  

Introduce \emph{inhomogeneity (disorder)} through real parameters in \eqref{EPar} assumed to satisfy \eqref{AsmParam}. 
A weight (waiting time) process with property \eqref{EWP} can be defined on $(\Blk, \cP)$ by setting 
\begin{align}
\cwp{i, j}{m, n} = \frac{\cwph{i, j}{m, n}}{\ca{m}{i}+\cb{n}{j}} \quad \text{ for } m, n \in \bbZ_{>0} \text{ and } i \in [m], j \in [n]. \label{EWPdef}
\end{align}
Let $\cPf{m, n}$ denote the distribution of the weights $\{\cwp{i, j}{m, n}: i \in [m], j \in [n]\}$ under $\cP$ for $m, n \in \bbZ_{>0}$. In other words, $\cPf{m, n}$ is the Borel probability measure on $\bbR^{[m] \times [n]}$ under which 
\begin{align}
&\{\cw{i, j}: (i, j) \in [m] \times [n]\} \text{ are independent and } \nonumber\\
&\cw{i, j} \sim \Exp(\ca{m}{i}+\cb{n}{j}) \text{ for } (i, j) \in [m] \times [n]. \label{EblkwDis}
\end{align}

Define the \emph{bulk LPP process} via \eqref{ELPT} using the weights in \eqref{EWPdef}. Namely, the value of the process at site $(m, n)$ is given by 
\begin{align*}
\cGp{m, n} = \max \limits_{\pi \in \Pi_{1, 1}^{m, n}} \sum_{(i, j) \in \pi} \cwp{i, j}{m, n} \quad \text{ for } m, n \in \bbZ_{>0} \text{ and } i \in [m], j \in [n]. 
\end{align*}
This is a particular construction of the corner growth process discussed in Subsection \ref{SsiCGM}. 

\subsection{Stationary last-passage increments}

The horizontal and vertical $\csG{}$-increments are defined by 
\begin{align}
\csI{m, n} &= \one_{\{m > 0\}}(\csG{m, n}-\csG{m-1, n}) \label{EHorInc}\\
\csJ{m, n} &= \one_{\{n > 0\}}(\csG{m, n}-\csG{m, n-1}) \label{EVerInc} 
\end{align}
respectively, for $m, n \in \bbZ_{\ge 0}$. From the definitions,  $\csI{m, 0}=\csw{m, 0}$ and $\csJ{0, n}=\csw{0, n}$ for $m, n \in \bbZ_{>0}$. 

Let $\cIz{m, n}=(-\min \ca{m}{}, \min \cb{n}{})$ for $m, n \in \bbZ_{\ge 0}$ with the convention $\min \ca{0}{} = \min \cb{0}{} = \infty$. For $m, n \in \bbZ_{\ge 0}$ and $z \in \cIz{m, n}$, let $\csPf{m, n}{z}$ denote the Borel probability measure on $\bbR^{([m] \cup \{0\}) \times ([n] \cup \{0\})}$ under which 
\be
\label{Estw}
\begin{aligned}
&\{\csw{i, j}: i \in [m] \cup \{0\}, j \in [n] \cup \{0\}\} \text{ are independent, } \csw{}(0, 0) \stackrel{\text{a.s.}}{=} 0, \text{ and } \\ &\text{ for } i \in [m], j \in [n],  \quad \csw{i, j} \sim \Exp(\ca{m}{i}+\cb{n}{j}), \\ 
&\csw{i, 0} \sim \Exp(\ca{m}{i}+z) \ \text{and }  \ \csw{0, j} \sim \Exp(\cb{n}{j}-z). 
\end{aligned}
\ee
Under $\csPf{m, n}{z}$  the bulk weights   $\{\csw{i, j}: i \in [m], j \in [n]\}$  have distribution  $\cPf{m, n}$ described in \eqref{EblkwDis} and $\csG{}$-increments are stationary in the sense that 
\be\label{EBurke1}\begin{aligned}
&\text{$\{\csI{i, n}: i \in [m]\}$ are independent with $\csI{i, n}\sim$ Exp$(\ca{m}{i}+z)$}\\
\text{and } 
&\text{$\{\csJ{m, j}: j \in [n]\}$ are independent with $\csJ{m, j}\sim$ Exp$(\cb{n}{j}-z)$.}
\end{aligned}\ee
A stronger version of this property is in \cite[Lemma 4.2]{bala-cato-sepp} for constant parameters. The extension to the general case is sketched in \cite{emra-16-ecp}. 

We study the bulk LPP process mainly through the coupling $\cw{i, j} = \csw{i, j}$ for $i, j \in \bbZ_{>0}$. Then $\cGi{m, n}{k, l} = \csGi{m, n}{k, l}$ for $m, n, k, l \in \bbZ_{>0}$, and  
\begin{align}
\label{EexitId}
\csG{m, n} = \begin{cases} \csG{\cHE{m, n}, 0} + \cGi{m, n}{\cHE{m, n}, 1}, &\text{if } \cHE{m, n}>0 \\[3pt]
\csG{0, \cVE{m, n}} + \cGi{m, n}{1, \cVE{m, n}}, &\text{if } \cVE{m, n}>0  \end{cases} 
\end{align}
for $m, n \in \bbZ_{>0}$ follows from definitions \eqref{EhorExit}--\eqref{EverExit}. Utilizing the stationarity and the independence structure of $\sG$-increments, we establish sufficient control over the exit points and then gain access to the bulk LPP process via \eqref{EexitId}. 

\subsection{TASEP with particlewise and holewise disorder}
\label{SsTASEP}

Assume now that $\cw{i, j} \ge 0$ for $i, j \in \bbZ_{>0}$. Define the height of an \emph{interface} over site $n \in \bbZ_{>0}$ at time $t \in \bbR_{\ge 0}$ by 
\begin{align}
\cH{n}{t} = \max \{\sup \{m \in \bbZ_{> 0}: \cG{m, n} \le t\}, 0\}. \label{EH}
\end{align}
Since the weights are nonnegative, $\cG{}$ is coordinatewise nondecreasing. Hence, $\cH{n}{t}$ is nonincreasing in $n$ and nondecreasing in $t$. Note that $\cH{n}{t}$ also measures the width of the cluster in \eqref{EClus} at level $n$ and time $t$.  

The height variables also represent evolving configurations of particles on $\bbZ$ as follows: The position of particle $n \in \bbZ_{>0}$ at time $t \in \bbR_{\ge 0}$ is given by 
\begin{align}
\cprt{n}{t} = \cH{n}{t}-n+1. \label{Eprt}
\end{align}
Since $\cprt{n}{t}$ is decreasing in $n$ and nonincreasing in $t$, the particles move right over time retaining their order. In particular, each site is occupied by at most one particle at any time. Assume further that $\cw{i, j} > 0$ for $i, j \in \bbZ_{>0}$. Then each particle jumps one step at a time and the particles start from the step initial condition i.e.\ $\sigma(n, 0) = -n+1$ for $n \in \bbZ_{>0}$. 

Define the (total) flux over the time interval $[0, t]$ through site $i \in \bbZ_{>0}$ by 
\begin{align}
\cFlux{i}{t} &= \max \{\sup \{j \in \bbZ_{>0}: \cG{i+j-1, j} \le t\}, 0\}. \label{Eflux} 
\end{align}
Note from definitions \eqref{EH} and \eqref{Eprt} that 
\begin{align*}
\{j \in \bbZ_{>0}: \cG{i+j-1, j} \le t\} = \{j \in \bbZ_{>0}: \cH{j}{t} \ge j+i-1\} = \{j \in \bbZ_{>0}: \cprt{j}{t} \ge i\}. 
\end{align*}
Since the particles are initially at negative sites, it follows that $\cFlux{i}{t}$ counts the number of particles that have jumped from $i-1$ to $i$ by time $t$. 

Define the height process $\{\cHp{n}{t}: n \in \bbZ_{>0}, t \in \bbR_{\ge 0}\}$, the particle process $\{\cprtp{n}{t}: n \in \bbZ_{>0}, t \in \bbR_{\ge 0}\}$ and the flux process $\{\cFluxp{m}{t}: m \in \bbZ_{>0}, t \in \bbR_{\ge 0}\}$ through \eqref{EH}, \eqref{Eprt} and \eqref{Eflux}, respectively, using the bulk LPP process $\cGp{}$ in place of $\G$. 
The disordered TASEP arises in the special case $\ca{m}{i} = a_i$ and $\cb{n}{j} = b_j$ for $m, n \in \bbZ_{>0}$, $i \in [m]$, $j \in [n]$. 

\section{Main results}
\label{SRes}

We state our main results in this section. Throughout, fix two collections of real parameters $\a = \{\ca{m}{i}: m \in \bbZ_{>0}, i \in [m]\}$ and $\b = \{\cb{n}{j}: n \in \bbZ_{>0}, j \in [n]\}$ subject to condition \eqref{AsmParam}. 

\subsection{An explicit centering for the LPP process}
\label{SsCent}
\begin{defn}\label{DCent}
We call a (deterministic) function $F: \bbZ_{>0}^2 \rightarrow \bbR$ a \emph{centering} for the $\cGp{}$-process if for any $\epsilon > 0$, $\cP$-a.s., there exists a random $L \in \bbZ_{>0}$ such that 
\begin{align}
|\cGp{m, n}-F(m, n)| \le \epsilon (m+n) \quad \text{ for } m, n \in \bbZ_{\ge L}. \label{ECentDef}
\end{align}
\end{defn}
The definition does not determine $F$ uniquely since, for any function $f: \bbZ_{>0}^2 \rightarrow \bbR$ with 
$
\sup_{\substack{m, n \in \bbZ_{\ge l}}}(m+n)^{-1}f(m, n) \stackrel{l \rightarrow \infty}{\rightarrow} 0, 
$
the function $F+f$ also satisfies \eqref{ECentDef}. The results of this subsection provide an explicit centering under a mild condition on the parameters.  

Condition \eqref{AsmParam} implies that the interval $\cIz{m, n} = (-\min\ca{m}{}, \min \cb{n}{})$ 
is nonempty for $m, n \in \bbZ_{\ge 0}$. When $m, n \in \bbZ_{>0}$,  the length $|\cIz{m, n}| = \min \ca{m}{} + \min \cb{n}{}$ of this interval equals the reciprocal of the maximal mean of the weights in the rectangle $[m] \times [n]$: 
\begin{align*}
|\cIz{m, n}|^{-1} = \max_{i \in [m], j \in [n]} \cE[\cwp{i, j}{m, n}] \quad \text{ for } m, n \in \bbZ_{>0}. 
\end{align*}
Define 
\begin{align}
\label{EMeanSt}
\cEStShp{m, n}{z} = \sum_{i=1}^m \frac{1}{\ca{m}{i}+z} + \sum_{j=1}^n \frac{1}{\cb{n}{j}-z} \quad \text{ for } m, n \in \bbZ_{\ge 0}
\end{align}
and $z \in \bbC \smallsetminus (\{-\ca{m}{i}: i \in [m]\} \cup \{\cb{n}{j}: j \in [n]\})$. When $z \in \cIz{m, n}$, \eqref{EMeanSt} gives the mean of $\csG{m, n}$ under $\csPf{m, n}{z}$ 
as can be seen from \eqref{EHorInc}, \eqref{EVerInc} and \eqref{EBurke1}. 
Next define 
\begin{align}
\label{EECenter}
\cEShp{m, n} = \inf_{z \in \cIz{m, n}} \cEStShp{m, n}{z} \quad \text{ for } m, n \in \bbZ_{>0}.
\end{align}

Our first result bounds the difference $\cGp{m, n}-\cEShp{m, n}$ with error terms that are of smaller order than $m+n$ provided that $|\cIz{m, n}|$ does not decay too quickly. Note that the bounds require only one of $m$ and $n$ to be sufficiently large.  
\begin{thm}
\label{TEmpShp}
Let $p > 0$. Then, $\cP$-a.s., there exists a random $L \in \bbZ_{>0}$ such that
\begin{align*}
\cGp{m, n} &\le \cEShp{m, n} + |\cIz{m, n}|^{-1}(m+n)^{1/2+p} \\
\cGp{m, n} &\ge \cEShp{m, n} - |\cIz{m, n}|^{-1}(m+n)^{9/10+p}
\end{align*}
for $m, n \in \bbZ_{>0}$ with $m+n \ge L$. 
\end{thm}

In particular, \eqref{EECenter} is a centering for the $\cGp{}$-process under a mild condition. 
\begin{cor}
\label{CEmpShp}
Assume that, for some $c > 0$ and $\eta > 0$, 
\begin{align}
\label{AsmMin}
|\cIz{m, n}| \ge c(m+n)^{-1/10+\eta} \quad \text{ for } m, n \in \bbZ_{>0}. 
\end{align}
Let $\epsilon > 0$. Then, $\cP$-a.s., there exists a random $L \in \bbZ_{>0}$ such that
\begin{align*}
|\cGp{m, n}-\cEShp{m, n}| \le \epsilon (m+n) \quad \text{ for } m, n \in \bbZ_{\ge 0} \text{ with } m+n \ge L. 
\end{align*}
\end{cor}

\begin{rem} Assumption \eqref{AsmMin} is not claimed to be sharp.
The result does  fail if $|\cIz{m, n}|$ is allowed to decay too fast.  For example, let  $\ca{n}{i} = \cb{n}{i} = 2^{-i}$ for $n \in \bbZ_{>0}$ and $i \in [n]$. Then  $\cEShp{n, n}\le \cEStShp{n, n}{0}=2\sum_{i=1}^n 2^i \le 2^{n+2}$, while $\G(n, n)\ge \w(n,n) \ge 100\cdot 2^n$ happens with at least probability $e^{-200}$ for each $n$. 
\end{rem}

\subsection{A conjecture due to E.\ Rains}
\label{SsRains}
The next result can be viewed as a variant of Corollary \ref{CEmpShp} on account of the similarity of the centering \eqref{EECenter} to the limit \eqref{ERainsCent} below. The statement is a reformulation of a conjecture due to E.\ Rains \cite{rain-00}. 

\begin{thm}
\label{TRains}
Let $\a = (a_i)_{i \in \bbZ_{>0}}$ and $\b = (b_j)_{j \in \bbZ_{>0}}$ be real sequences subject to 
\begin{align}
&\inf \a + \inf \b > 0 \label{AsmInf},\\ 
&\cEStShp{\infty, \infty}{z} = \sum_{i=1}^\infty \frac{1}{a_i+z} + \sum_{j=1}^\infty \frac{1}{b_j-z} < \infty \quad \text{ for } -\inf \a  < z < \inf \b, \label{AsmRainsConv}
\end{align}
where the equality is a definition. Also define 
\begin{align}
\cEShp{\infty, \infty} = \inf_{-\inf \a < z < \inf \b} \cEStShp{\infty, \infty}{z}. \label{ERainsCent}
\end{align}
On $(\Blk, \cP)$, consider the weights 
\begin{align}
\w_n(i, j) = \frac{\cw{i, j}}{a_{\lc i/n \rc} + b_{\lc j/n \rc}} \quad \text{ for } i, j, n \in \bbZ_{>0}. \label{ERainsWp}
\end{align}
For each $n \in \bbZ_{>0}$, define the last-passage times $\{\G_n(i, j): i, j \in \bbZ_{>0}\}$ via \eqref{ELPT} from the weights $\{\w_n(i, j): i, j \in \bbZ_{>0}\}$. Then, $\cP$-a.s., 
\begin{align*}
\lim_{n \rightarrow \infty} n^{-1}\sup_{i, j \in \bbZ_{>0}} \G_n(i, j) = \cEShp{\infty, \infty}. 
\end{align*}
\end{thm}
\begin{rem}
Since the conjecture addressed above  is somewhat dispersed within the text of \cite{rain-00}, we explain how to locate it. The statement is a special case of the continuous analogue of Conjecture 5.2. To obtain it, replace $m(\alpha, p_+, p_-)$ in (5.5) with $m_c(z; \rho^+, u; \rho^-)$ from (5.29) and set $u = 0$. The parameters $\rho^\pm = (\rho_i^\pm)$ and $a$ play the role of $\a, \b$ and $z$ in our setting. Assumptions (2.14)--(2.15) there correspond to our \eqref{AsmInf}--\eqref{AsmRainsConv}. The infimum in (5.5) is now to be taken over $z \in (-\inf \a, \inf \b)$. After these changes, the right-hand side of (5.5) becomes $\cEShp{\infty, \infty}$. In direct analogy with (5.7), the weights are chosen as in \eqref{ERainsWp}. Finally, the quantity $\lambda_1$ in (5.5) is precisely $\sup_{i, j \in \bbZ_{>0}} \G_n(i, j)$, which can be inferred from the discussion preceding \cite[Theorem 2.4]{rain-00}.
\end{rem}

\subsection{Shape function}
\label{SShpFun}

\begin{defn}\label{DShpFun} We call a deterministic  function $\shp: \bbR_{>0}^2 \rightarrow \bbR_{\ge 0}$ the {\rm shape function}  of the $\cGp{}$-process if $\shp$ is coordinatewise nondecreasing, positive-homogeneous {\rm(}see Subsection \ref{SsNot}{\rm)} and its restriction to $\bbZ_{>0}^2$ is a centering for $\cGp{}$ in the sense of 
Definition \ref{DCent}. 
\end{defn}
Definition \ref{DShpFun} is consistent with the notion of shape function from earlier literature, see Remark \ref{RDirLim} below. It can be seen from the definition that the shape function, if it exists, is necessarily unique and continuous. 
This subsection provides an explicit formula for the shape function under some natural sufficient conditions for its existence.  

Let $\alpha$ and $\beta$ be finite, nonnegative Borel measures on $\bbR$. Introduce the functions 
\begin{alignat}{2}
\cA{z} &= \int_{\bbR} \frac{\alpha(\dd a)}{a+z} && \quad \text{ for } z \in \bbC \smallsetminus (-\supp \alpha), \label{EAdef} \\
\cB{z} &= \int_{\bbR} \frac{\beta(\dd b)}{b-z} && \quad \text{ for } z \in \bbC \smallsetminus \supp \beta. \label{EBdef}
\end{alignat}
The integrals above are well-defined, and $\cA{}$ and $\cB{}$ are holomorphic functions. For each $z \in \bbC \smallsetminus ((-\supp \alpha) \cup \supp \beta)$, define 
\begin{align}
\label{EStShpFun} \cStShp{x, y}{z} &= x\cA{z} + y\cB{z} = x\int_\bbR \frac{\alpha(\dd a)}{a+z} + y\int_\bbR \frac{\beta(\dd b)}{b-z} \quad \text{ for } x, y > 0. 
\end{align}
As recorded in the last case of \cite[Corollary 3.16]{emra-janj-sepp-19-shp-long}, the shape function 
of the increment-stationary LPP process can be given in the form \eqref{EStShpFun}. This fact will not be used in the sequel, however. 

Assume that 
\begin{align}
\inf \supp \alpha + \inf \supp \beta > 0.  \label{EInfSuppCond}
\end{align}
Zero measures are acceptable here: if $\alpha=0$ then $\inf \supp \alpha=\infty$ and the same for $\beta$.  
Let $\aml, \bml \in \bbR \cup \{\infty\}$ satisfy
\begin{align}
&\aml + \bml > 0 \label{AsmLB}\\
&\aml \le \inf \supp \alpha \text{ and } \bml \le \inf \supp \beta. \label{Eab}
\end{align}
Then the intervals $(-\aml, \infty)$ and $(-\infty, \bml)$ are contained in the domains of \eqref{EAdef} and \eqref{EBdef}, respectively.  In particular, each $z \in (-\aml, \bml)$ is a legitimate parameter in \eqref{EStShpFun}. Therefore, the following definition is sensible: 
\begin{align}
\cShp{x, y} = \inf_{z \in (-\aml, \bml)} \cStShp{x, y}{z} \quad \text{ for } x, y \in \bbR_{>0}. \label{EShpFun}
\end{align}
As will be made precise with Lemma \ref{LEShpConv} below, one can view \eqref{EShpFun} as the continuous analogue of \eqref{EECenter} obtained after letting $m, n \to \infty$. 
From \eqref{EStShpFun}, \eqref{AsmLB} and \eqref{EShpFun} it follows that  the function $\cShp{}$ is nonnegative, concave, finite, coordinatewise nondecreasing and positive-homogeneous. 

The next result shows that the shape function of the $\cGp{}{}$-process exists and can be represented in the form \eqref{EShpFun} under mild conditions. In the statement, 
\begin{align*}
\alpha_m^{\bfa} = \frac{1}{m} \sum_{i=1}^m \delta_{\ca{m}{i}} \quad \text{ and } \quad \beta_n^{\bfb} = \frac{1}{n} \sum_{j=1}^n \delta_{\cb{n}{j}} \quad \text{ for } n \in \bbZ_{>0}
\end{align*}
denote the empirical distributions associated with the parameters $\a$ and $\b$, respectively. 

\begin{thm}
\label{TShpFun}
Assume that 
\begin{align}
&\lim_{m \rightarrow \infty}\alpha_m^{\bfa} = \alpha \quad \text{ and } \quad \lim_{n \rightarrow \infty} \beta_n^{\bfb} = \beta \quad \text{ in the vague topology} \label{AsmVagueConv} \\
&\lim_{m \rightarrow \infty} \min \ca{m}{} = \aml \quad \text{ and } \quad \lim_{n \rightarrow \infty} \min \cb{n}{}  = \bml \label{AsmMinConv}
\end{align} 
for some subprobability measures $\alpha, \beta$ on $\bbR$ and $\aml, \bml \in \bbR \cup \{\infty\}$ such that  \eqref{AsmLB} is satisfied. 
Then for any $\epsilon > 0$, $\cP$-a.s., there exists a random $L \in \bbZ_{>0}$ such that 
\begin{align*}
|\cGp{m, n}-\cShp{m, n}| \le \epsilon (m+n) \quad \text{ for } m, n \in \bbZ_{\ge L}.
\end{align*}
\end{thm}
\begin{rem}
Assumptions \eqref{AsmVagueConv}--\eqref{AsmMinConv} imply \eqref{Eab}. Hence, the shape function in the theorem given by \eqref{EShpFun} makes sense. 
\end{rem}
\begin{rem}
\label{RShpFunBd}
The condition $m, n \ge L$ can be weakened to $m+n \ge L$ if and only if 
\begin{align}
\cA{\min \cb{n}{}} = \cA{\bml} \text{ and } \cB{-\min \ca{m}{}} = \cB{-\aml} \quad \text{ for } m, n \in \bbZ_{>0}.  \label{EBdExt2}
\end{align}
This is recorded in \cite[Corollary 3.9]{emra-janj-sepp-19-shp-long}. Condition \eqref{EBdExt2} is equivalent to 
\begin{align}
\min \ca{m}{} = \aml \text{ for } m \in \bbZ_{>0} \text{ when } \beta \neq 0 \quad \text{ and } \quad \min \cb{n}{} = \bml \text{ for } n \in \bbZ_{>0} \text{ when } \alpha \neq 0 \label{EBdExt}
\end{align}
by the strict monotonicity of $\cA{}$ and $\cB{}$. (Recall from \eqref{AsmVagueConv} that the excluded case $\alpha = 0$ above corresponds to the mass of $\alpha_m^\a$ escaping to infinity as $m \to \infty$. Similarly for the case $\beta = 0$).  
\end{rem}
\begin{rem} 
When either $\alpha$ or $\beta$ is the zero measure, $\shp$ becomes a one-variable function and can be readily computed from \eqref{EShpFun} by monotonicity. 
\begin{align}
\label{E46}
\cShp{x, y} = \begin{cases}x\cA{\bml} \quad &\text{ if } \beta = 0 \\ y\cB{-\aml} \quad &\text{ if } \alpha = 0\end{cases} \quad \text{ for } x, y > 0. 
\end{align}
Hence, the shape function is identically zero in three cases: 
\begin{align}
\label{EShpFunZero}
\cShp{} = 0 \quad \text{ if } \aml = \infty \text{ or } \bml = \infty \text{ or } \alpha = \beta = 0. 
\end{align}
If none of the conditions in \eqref{EShpFunZero} holds then the shape function is nonzero on $\bbR_{>0}^2$ since 
\begin{align}
\cShp{x, y} \ge x\cA{\bml}+y\cB{-\aml} \ge \max \{x\cA{\bml}, y\cB{-\aml}\} \quad \text{ for } x, y > 0. \label{EShpBdLB}
\end{align} 
\end{rem}
\begin{rem}
The hypotheses of the theorem can be motivated from the following 
statement that holds 
under the weaker assumption \eqref{AsmInf}:  Let $(m_i)_{i \in \bbZ_{>0}}$ and $(n_j)_{j \in \bbZ_{>0}}$ be increasing sequences in $\bbZ_{>0}$. By the (sequential) compactness of $[-\infty, \infty]$ and the space of subprobability measures (see Lemma \ref{LSpmVagComp}), there exist increasing sequences $(k_i)_{i \in \bbZ}$ and $(l_j)_{j \in \bbZ_{>0}}$ in $\bbZ_{>0}$ such that 
\begin{align*}
&\lim_{i \rightarrow \infty}\alpha_{m_{k_i}}^{\bfa} = \alpha \quad \text{ and } \quad \lim_{j \rightarrow \infty} \beta_{n_{l_j}}^{\bfb} = \beta \quad \text{ in the vague topology} \\
&\lim_{i \rightarrow \infty} \min \ca{m_{k_i}}{} = \aml \quad \text{ and } \quad \lim_{j \rightarrow \infty} \min \cb{n_{l_j}}{}  = \bml 
\end{align*} 
for some subprobability measures $\alpha, \beta$ on $\bbR$ and $\aml, \bml \in \bbR \cup \{\infty\}$ with \eqref{AsmLB}. Repeating the argument in Section \ref{SPrShpFun} with $m = m_{k_i}$ and $n = n_{l_j}$ yields that for any $\epsilon > 0$, $\cP$-a.s., there exists a random $L \in \bbZ_{>0}$ such that 
\begin{align*}
|\cGp{m_{k_i}, n_{l_j}}-\cShp{m_{k_i}, n_{l_j}}| \le \epsilon (m_{k_i} + n_{l_j}) \quad \text{ for } i, j \in \bbZ_{>0} \text{ with } i, j \ge L. 
\end{align*}
\end{rem}

\begin{rem}
\label{RDirLim}
Theorem \ref{TShpFun} together with positive-homogeneity and continuity of the shape function implies that there exists a single $\cP$-a.s.\ event on which  
\begin{align}
\lim_{k \rightarrow \infty} k^{-1} \cGp{m_k^x, n_k^y} = \cShp{x, y} \label{ELPPDirLim}
\end{align}
holds for all choices of $x, y > 0$ and all sequences $\{m_k^x, n_k^y\}_{k \,\in\, \bbZ_{>0}}  \subset  \bbZ_{>0}$ such that  $m_k^x/k \rightarrow x$ and $n_k^y/k \rightarrow y$ as $k \rightarrow \infty$. In literature, \eqref{ELPPDirLim} with $m_k^x = \lc k x \rc$ and $n_k^y = \lc k y\rc$ is often taken as the definition of the shape function \cite{mart-06, sepp-cgm}.  
\end{rem}

\begin{rem}
Unlike the usual sequence of arguments, we do not first establish the a.s.\ limit in \eqref{ELPPDirLim} to obtain Theorem \ref{TShpFun}. Instead, we derive Theorem \ref{TShpFun} from Corollary \ref{CEmpShp} by approximating the centering \eqref{EECenter} by the shape function.
\end{rem}

\subsection{Growth near the axes}

The next theorem provides uniform approximations to the growth process at sites where one coordinate is large and the other is relatively small. In particular, the result describes the asymptotics along a fixed column or row. 
\begin{thm}
\label{TNarShp}
Let $S \subset \bbZ_{>0}$ and $\epsilon > 0$. 
\begin{enumerate}[\normalfont (a)]
\item Assume that the first limit in \eqref{AsmVagueConv} holds, and $\inf \a + \inf \limits_{n \in S} \min \cb{n}{} > 0$. Then, $\cP$-a.s., there exist a deterministic $\delta > 0$ and a random $K \in \bbZ_{>0}$ such that 
\begin{align*}
|\cGp{m, n}-m\cA{\min \cb{n}{}}| \le \epsilon m \quad \text{ for } m \in \bbZ_{\ge K} \text{ and } n \in S \text{ with } n \le \delta m.  
\end{align*}
\item Assume that the second limit in \eqref{AsmVagueConv} holds, and $\inf \b+\inf \limits_{m \in S} \min \ca{m}{} > 0$. Then, $\cP$-a.s., there exist a deterministic $\delta > 0$ and a random $K \in \bbZ_{>0}$ such that 
\begin{align*}
|\cGp{m, n}-n\cB{-\min \ca{m}{}}| \le \epsilon n \quad \text{ for } n \in \bbZ_{\ge K} \text{ and } m \in S \text{ with } m \le \delta n. 
\end{align*}
\end{enumerate}
\end{thm}

With the assumptions of Theorem \ref{TShpFun}, the approximations in Theorem \ref{TNarShp} can be replaced with simpler one-dimensional linear functions at sites where both coordinates are large but one is small relative to the other.  
\begin{cor}
\label{CThinShp}
Assume \eqref{AsmLB}, \eqref{AsmVagueConv} and \eqref{AsmMinConv}. Let $\epsilon > 0$. Then, $\cP$-a.s., there exist a deterministic $\delta > 0$ and a random $L \in \bbZ_{>0}$ such that the following hold for $m, n \in \bbZ_{\ge L}$.  
\begin{enumerate}[\normalfont (a)]
\item $|\cGp{m, n}-m\cA{\bml}| \le \epsilon m$ if $n \le \delta m$. 
\item $|\cGp{m, n}-n\cB{-\aml}| \le \epsilon n$ if $m \le \delta n$.
\end{enumerate}
\end{cor}

\subsection{Limit shape}
\label{SsLimShp}
Denote the growing cluster associated with the bulk LPP process by 
\begin{align}\cClus{t} = \{(x, y) \in \bbR_{>0}^2: \cGp{\lc x \rc, \lc y \rc} \le t\} \quad \text{ for } t \in \bbR_{\ge 0}. \label{EClus2}\end{align} 
Our purpose is to describe the \emph{limit shape}, namely, the linear scaling limit of \eqref{EClus2} with respect to the Hausdorff metric $\HDis$ (constructed from the Euclidean metric on $\bbR_{\ge 0}^2$), see Appendix \ref{SsHaus}. To this end, define 
\begin{align}
\cLimShp &= \{(x, y)\in \bbR_{>0}^2: \cShp{x, y} \le 1\} \label{ELimShp} \\
\cLimShpBd{w, z} &= \{(x, 0): x \in \bbR_{\ge 0} \text{ and } x \cA{w} \le 1\} \cup \{(0, y): y \in \bbR_{\ge 0} \text{ and } y \cB{z} \le 1\}. \label{ELimShpBd}
\end{align}
for $\alpha, \beta, \aml, \bml$ subject to \eqref{EInfSuppCond}--\eqref{Eab}, $w > -\inf \supp \alpha$ and $z < \inf \supp \beta$. The following result identifies the limit shape in terms of \eqref{ELimShp}-\eqref{ELimShpBd}. The statement involves a truncation that restricts to a fixed bounded set when the limit shape is unbounded. 

\begin{thm}
\label{TLimShp}
Assume \eqref{AsmLB}, \eqref{AsmVagueConv} and \eqref{AsmMinConv}. Let $0 < \epsilon < 1$. Let $C > 0$ and 
\begin{align}
S &= \bigg\{(x, y) \in \bbR_{\ge 0}^2: x \le \frac{C}{\one_{\{\alpha = 0\} \cup \{\bml = \infty\}}} \text{ and } y \le \frac{C}{\one_{\{\beta = 0\} \cup \{\aml = \infty\}}}\bigg\}  \label{ETrU}
\end{align}
with the convention $1/0 = \infty$. Then, 
$\cP$-a.s., 
\begin{align*}
\lim_{t \rightarrow \infty}\cHDis{S \cap \overline{t^{-1}\cClus{t}}}{S \cap (\cLimShp \cup \cLimShpBd{\ams, \bms})} = 0, 
\end{align*}
where 
\begin{align}
\ams = \sup \limits_{m \in \bbZ_{>0}} \min \ca{m}{} \quad \text{ and } \quad \bms = \sup \limits_{n \in \bbZ_{>0}} \min \cb{n}{}. \label{EAB}
\end{align} 
\end{thm}

\begin{rem}
It follows from \eqref{EShpBdLB} and an omitted complementary upper bound (\cite[Lemma 8.1]{emra-janj-sepp-19-shp-long})
\begin{align}
\lim_{t \rightarrow 0} \cShp{x, t} = x \cA{\bml} \quad \text{ and } \quad \lim_{t \rightarrow 0} \cShp{t, y} = y \cB{-\aml} \quad \text{ for } x, y \in \bbR_{>0}. \label{ShpFunBdLim}
\end{align}
These limits imply that the closure of $\cLimShp$ in $\bbR_{\ge 0}^2$ is $\cl{\LimShp}^{\alpha, \beta, \aml, \bml} = \cLimShp \cup \cLimShpBd{\aml, \bml}$. When $\cLimShp$ is bounded, $\cl{\LimShp}^{\alpha, \beta, \aml, \bml}$ coincides with the limit shape in part (a) if and only if 
\begin{align*}
\min \ca{m}{} \le \aml \text{ for } m \in \bbZ_{>0} \text{ when } \beta \neq 0 \quad \text{ and } \quad \min \cb{n}{} \le \bml \text{ for } n \in \bbZ_{>0} \text{ when } \alpha \neq 0. 
\end{align*} 
This should be compared with the case of i.i.d.\ weights, \cite[Theorem 5.1(i)]{mart-04}, where the limit shape is the closure in $\bbR_{\ge 0}^2$ of the sublevel-$1$ set of the shape function. 
\end{rem}

\subsection{Flat segments}
\label{SsFlat}

To describe the finer structure of the limit shape, define the regions 
\begin{align}
\cRv  &= \bigg\{(x, y) \in \bbR_{>0}^2: x \int_{\bbR} \frac{\alpha(\dd a)}{(a-\aml)^2} \le y \int_{\bbR} \frac{\beta(\dd b)}{(b+\aml)^2}\bigg\} \label{Eflatv}\\
\cRh  &= \bigg\{(x, y) \in \bbR_{>0}^2: x \int_{\bbR} \frac{\alpha(\dd a)}{(a+\bml)^2} \ge y \int_{\bbR} \frac{\beta(\dd b)}{(b-\bml)^2}\bigg\} \label{Eflath}\\
\cRc &= \bbR_{>0}^2 \smallsetminus \{\cRh \cup \cRv\}. \nonumber
\end{align}
To avoid trivialities, assume that $\alpha$ and $\beta$ are both nonzero. In particular, $\aml, \bml < \infty$. The regions $\cRh$, $\cRv$ and $\cRc$ are pairwise disjoint. Note that 
\begin{align*}
\cRv \neq \emptyset \quad \text{ if and only if } \quad \int_{\bbR} \frac{\alpha(\dd a)}{(a-\aml)^{2}} < \infty \\
\cRh \neq \emptyset \quad \text{ if and only if } \quad \int_{\bbR} \frac{\beta(\dd b)}{(b-\bml)^{2}} < \infty. 
\end{align*}
(The other two integrals in \eqref{Eflatv}--\eqref{Eflath} are finite by \eqref{AsmLB}). Also, $\cRc \neq \emptyset$ by the Cauchy-Schwarz inequality and assumption \eqref{AsmLB}. 
While the curved part is nonempty with the hypotheses of Theorem \ref{TShpFun}, it is also possible to generate completely flat limit shapes as in \cite[Corollary 3.19]{emra-janj-sepp-19-shp-long}. 

The justification for the following assertions can be seen from the results in Section \ref{SPrShpFun}. 
The function $\cShp{}$ is affine on the regions $\cRh$ and $\cRv$, and is given by 
\begin{align*}
\cShp{x, y} = \begin{cases}\displaystyle x \int_{\bbR} \frac{\alpha(\dd a)}{a-\aml} + y \int_{\bbR} \frac{\beta(\dd b)}{b+\aml} &\quad \text{ for } (x, y) \in \cRv \\\\
\displaystyle x \int_{\bbR} \frac{\alpha(\dd a)}{a+\bml} + y \int_{\bbR} \frac{\beta(\dd b)}{b-\bml} &\quad \text{ for } (x, y) \in \cRh. \end{cases}
\end{align*}
Hence, the boundaries (inside $\bbR_{>0}^2$) of the regions $\cLimShp \cap \cRv$ and $\cLimShp \cap \cRh$ are flat segments. 
On the other hand, the function $\cShp{}$ is strictly concave on $\cRc$ and the boundary of the region $\cLimShp \cap \cRc$ is curved (non-flat). The entire boundary of $\cLimShp$ is the image of a continuously differentiable curve and does not have any corners.

The boundary of the limit shape admits an exact parametrization. Indeed, the curve
\begin{align*}
\Phi(z) = \frac{(\partial \cB{z}, -\partial \cA{z})}{\cA{z} \partial \cB{z}-\cB{z}\partial\cA{z}} \quad \text{ for } z \in (-\aml, \bml) 
\end{align*}
parametrizes the curved part 
The flat boundaries inside $\bbR_{>0}^2$ are the line segments from $(0, \cB{-\aml}^{-1})$ to $\Phi(-\aml)$, and from $(\cA{\bml}^{-1}, 0)$ to $\Phi(\bml)$. Finally, the boundary along the axes are the line segments from the origin to $(0, \cB{-\ams}^{-1})$ and $(\cA{\bms}^{-1}, 0)$. 

\subsection{Spikes and crevices}
\label{SsSpikes}

In the absence of condition \eqref{EBdExt}, large (in macroscopic scale) \emph{spikes/crevices} form near the axes that are not visible in the limit shape. To demonstrate this, consider the case $\beta \neq 0$ and $\min \ca{m}{} \neq \aml$ for some $m \in \bbZ_{>0}$, namely, that the first part of \eqref{EBdExt} fails. The other case is analogous. The precise claim is that, for any $\epsilon > 0$ with $2\epsilon < |\cB{-\min \ca{m}{}}^{-1} -\cB{-\aml}^{-1}|/2$, $\cP$-a.s., there exists a random $T > 0$ such that
\begin{alignat}{2}
\bigg\{\frac{m}{t}\bigg\} \times \bigg[\frac{1}{\cB{-\aml}}+\epsilon, \frac{1}{\cB{-\min \ca{m}{}}}-\epsilon\bigg] &\subset t^{-1}\cClus{t} \smallsetminus \cLimShp \quad &&\text{ if } \min \ca{m}{} > \aml \label{ESpike}\\
\bigg\{\frac{m}{t}\bigg\} \times \bigg[\frac{1}{\cB{-\min \ca{m}{}}}+\epsilon, \frac{1}{\cB{-\aml}}-\epsilon\bigg] &\subset \cLimShp \smallsetminus  t^{-1}\cClus{t}\quad &&\text{ if } \min \ca{m}{} < \aml. \label{ECrevice}
\end{alignat}
for $t \ge T$. The line segments in \eqref{ESpike} and \eqref{ECrevice} can be visualized as a vertical spike or crevice, respectively, in the rescaled cluster near the vertical axis. 

To verify \eqref{ESpike} for example, pick $0 < \delta < \cB{-\aml}$ such that 
\begin{align*}
u = \frac{1}{\cB{-\aml}-\delta} \ge \frac{1}{\cB{-\aml}}-\epsilon \quad \text{ and } \quad v = \frac{1}{\cB{-\min \ca{m}{}}+\delta} \le \frac{1}{\cB{-\min \ca{m}{}}}+\epsilon. 
\end{align*}
By Theorem \ref{TNarShp}, $\cP$-a.s., there exists a random $K \in \bbZ_{>0}$ such that $\cGp{m, n} \le n(v+\delta)^{-1}$ for $n \in \bbZ_{\ge K}$. Therefore, $K \le n \le t(v+\delta)$ implies that $(m, n) \in \cClus{t}$. Consequently, 
\begin{align*}
\{t^{-1}m\}  \times [t^{-1}K, v]  \subset t^{-1}\cClus{t} \quad \text{ for } t \ge T = \max \{Kv^{-1}, \delta^{-1}\}. 
\end{align*}
On the other hand, by \eqref{ShpFunBdLim}, the shape function $\cShp{t^{-1}m, u} \stackrel{t \rightarrow \infty}{\rightarrow} u\cB{-\aml} > 1$. Hence, 
\begin{align*}
\{t^{-1}n\} \times [u, v] \subset t^{-1}\cClus{t} \smallsetminus \cLimShp \quad \text{ for } t \ge T
\end{align*}
after increasing $T$ if necessary. 

\subsection{Centerings for the height and flux processes}
\label{SsTASEPRes}

The limit shape for the bulk LPP process leads to the following explicit centerings for the height and flux processes, respectively, discussed in Subsection \ref{SsTASEP}. Assuming \eqref{AsmLB}, \eqref{Eab} and that $\alpha, \beta \neq 0$, introduce the limiting height and flux functions by  
\begin{alignat}{2}
\clH{y}{t} &= \max \bigg\{\sup_{z \in (-\aml, \bml)} \bigg\{\frac{t-y\cB{z}}{\cA{z}}\bigg\}, 0\bigg\} &&\quad \text{ for } (y, t) \in \bbR_{\ge 0}^2 \label{ELimHeight} \\
\clFlux{x}{t} &= \max \bigg\{\sup_{z \in (-\aml, \bml)}\bigg\{\frac{t-x\cA{z}}{\cA{z}+\cB{z}}\bigg\}, 0\bigg\} &&\quad \text{ for } (x, t) \in \bbR_{\ge 0}^2, \label{ELimFlux}
\end{alignat}
respectively. 

\begin{thm}
\label{TLimFlux}
Assume \eqref{AsmLB}, \eqref{AsmVagueConv}, \eqref{AsmMinConv} and $\alpha, \beta \neq 0$. Let $\epsilon > 0$. Then, $\cP$-a.s., there exists a {\rm(}random{\rm)} $L \in \bbZ_{>0}$ such that 
\begin{align*}
|\cHp{n}{t}-\clH{n}{t}| &\le \epsilon(t+n),  \quad |\cprtp{n}{t}-\clH{n}{t}+n| \le \epsilon (t+n) \text{ and } \\
|\cFluxp{m}{t}-\clFlux{m}{t}| &\le \epsilon (t+m)
\end{align*}
for $m, n \in \bbZ_{\ge L}$ and $t \in \bbR_{\ge 0}$. 
\end{thm}

The next result approximates the location of particle with a fixed label after a long time. 
\begin{thm}
\label{TNarHeight}
Assume the first limits in \eqref{AsmVagueConv}, \eqref{AsmMinConv} and that $\alpha \neq 0$. Fix $n \in \bbZ_{>0}$ and assume that $\aml + \min \cb{n}{} > 0$. Then, $\cP$-a.s., there exists a {\rm(}random{\rm)} $T \in \bbZ_{>0}$ such that 
\begin{align*}
\bigg|\cHp{n}{t}-\frac{t}{\cA{\min \cb{n}{}}}\bigg| \le \epsilon t \quad \text{ and } \quad \bigg|\cprtp{n}{t}-\frac{t}{\cA{\min \cb{n}{}}}+n\bigg| \quad \text{ for } t \ge T. 
\end{align*}
\end{thm}

\section{Concentration bounds for the last-passage times}
\label{SConBdLPP}

As a main step towards the proof of Theorem \ref{TEmpShp}, we begin to derive concentration bounds for $\cG{m, n}$ around $\cEShp{m, n}$ under $\cPf{m, n}$ for each site $(m, n) \in \bbZ_{>0}^2$. 

Writing $\partial_z$ for the $z$-derivative note that, for $m, n \in \bbZ_{\ge 0}$, 
\begin{align}
\partial_z \cEStShp{m, n}{z} &= -\sum_{i=1}^m \frac{1}{(\ca{m}{i}+z)^2} + \sum_{j=1}^n \frac{1}{(\cb{n}{j}-z)^2} \nonumber\\
\partial_z^2 \cEStShp{m, n}{z} &= \sum_{i=1}^m \frac{2}{(\ca{m}{i}+z)^3} + \sum_{j=1}^n \frac{2}{(\cb{n}{j}-z)^3} \label{Eder2}
\end{align}
If $m$ or $n$ is nonzero then the function $z \mapsto \cEStShp{m, n}{z}$ is strictly convex on $\cIz{m, n}$ due to the strict positivity of \eqref{Eder2}. Also, if $m, n > 0$ then $\cEStShp{m, n}{z} \rightarrow \infty$ as $z$ approaches the boundary values $\{-\min \ca{m}{}, \min \cb{n}{}\}$ within $\cIz{m, n}$. Hence, there exists a unique minimizer $\EMin = \cEMin{m, n}$ in \eqref{EECenter}. This is the unique $z \in \cIz{m, n}$ that satisfies the implicit equation 
\begin{align}
\label{EShpMin}
\sum_{i=1}^m \frac{1}{(\ca{m}{i}+z)^2} = \sum_{j=1}^n \frac{1}{(\cb{n}{j}-z)^2}. 
\end{align}
Note from definition \eqref{Estw} that the sums in \eqref{EShpMin} are precisely the variances of $\sG(m, 0)$ and $\sG(0, n)$, respectively,  under $\csPf{m, n}{z}$. Thus, $\EMin$ is the unique $z$-value for which $\sG(m, 0)$ and $\sG(0, n)$ have the same variance given by  
\begin{align}
\label{EMinVar}
\cMinVar{m, n} = \sum_{i=1}^m \frac{1}{(\ca{m}{i}+\EMin)^2} = \sum_{j=1}^n \frac{1}{(\cb{n}{j}-\EMin)^2}. 
\end{align}
The deviations of $\cG{m, n}$ from the centering $\cEShp{m, n}$ are naturally expressed below in terms of this variance.  

For brevity, introduce the functions 
\begin{align*}
\cEA{m}{z} &= \cEStShp{m, 0}{z} = \sum_{i=1}^m \frac{1}{\ca{m}{i}+z} \quad \text{ for } m \in \bbZ_{\ge 0} \text{ and } z \in \bbC \smallsetminus \{-\ca{m}{i}: i \in [m]\},\\
\cEB{n}{z} &= \cEStShp{0, n}{z} = \sum_{j=1}^n \frac{1}{\cb{n}{j}-z} \quad \text{ for } n \in \bbZ_{\ge 0} \text{ and } z \in \bbC \smallsetminus \{\cb{n}{j}: j \in [n]\}
\end{align*}
with the convention that $\cEA{0}{z} = \cEB{0}{z} = 0$ for $z \in \bbC$. 
Because the concentration bounds below for site $(m, n) \in \bbZ_{>0}^2$ depend on the parameters $\a$ and $\b$ only through $\ca{m}{}$ and $\cb{n}{}$, it causes no loss in generality to assume in this section that 
\begin{align*}
\ca{m}{i} = a_i \quad \text{ and } \quad \cb{n}{j} = b_j \quad \text{ for } m, n \in \bbZ_{>0} \text{ and } i \in [m], j \in [n]
\end{align*}
for some real sequences $(a_i)_{i \in \bbZ_{>0}}$ and $(b_j)_{j \in \bbZ_{>0}}$.

We first record a concentration bound for $\csG{m, n}$ around $\cEShp{m, n}$ under $\csPf{m, n}{\EMin}$. 
\begin{lem}
\label{LStLPPCon}
Let $m, n \in \bbZ_{>0}$ and $\EMin = \cEMin{m, n}$. Let $s > 0$ and $p \in \bbZ_{> 0}$. Then there exists a constant $C_p > 0$ {\rm(}depending only on $p${\rm)} such that
\begin{align*}
\csPf{m, n}{\EMin}\bigl\{\,|\sG(m, n)-\cEShp{m, n}| \ge s\bigl(\cMinVar{m, n}\bigr)^{1/2\,}\bigr\} \le C_ps^{-p}. 
\end{align*}
\end{lem}
\begin{proof}
Abbreviate $\MinVar = \cMinVar{m, n}$. By the triangle inequality, a union bound and \eqref{EBurke1}, the probability in the statement is at most 
\begin{align*}
&\csPf{m, n}{\EMin}\bigl\{|\sG(m, 0)-\cEA{m}{\EMin}| + |\sG(m, n)-\sG(m, 0)-\cEB{n}{\EMin}| \ge s \sqrt{\MinVar}\bigr\} \\
&\le \csPf{m, n}{\EMin}\bigg\{|\sG(m, 0)-\cEA{m}{\EMin}| \ge \frac{s}{2}\sqrt{\MinVar}\bigg\} + \csPf{m, n}{\EMin}\bigg\{|\sG(m, n)-\sG(m, 0)-\cEB{n}{\EMin}| \ge \frac{s}{2}\sqrt{\MinVar}\bigg\} \\
&= \csPf{m, n}{\EMin}\bigg\{|\sG(m, 0)-\cEA{m}{\EMin}| \ge \frac{s}{2}\sqrt{\MinVar}\bigg\} + \csPf{m, n}{\EMin}\bigg\{|\sG(0, n)-\cEB{n}{\EMin}| \ge \frac{s}{2}\sqrt{\MinVar}\bigg\}.
\end{align*}
Now the result readily follows from definition \eqref{EMinVar} and Lemma \ref{LExpCon}. 
\end{proof}

An immediate consequence is the next right tail bound. 
\begin{lem}
\label{LBulkRTBnd}
Let $m, n \in \bbZ_{>0}$. Let $s > 0$ and $p \in \bbZ_{> 0}$. Then there exists a constant $C_p > 0$ {\rm(}depending only on $p${\rm)} such that 
\begin{align*}
\cPf{m, n}\{\G(m, n)-\cEShp{m, n} \ge s\bigl(\cMinVar{m, n}\bigr)^{1/2}\} \le C_ps^{-p}.
\end{align*}
\end{lem}
\begin{proof}
Write $\EMin = \cEMin{m, n}$. Since $\cPf{m, n}$ is a projection of $\csPf{m, n}{\EMin}$ and $\G(m, n) \le \sG(m, n)$ a.s.\ under $\csPf{m, n}{\EMin}$, the result follows from Lemma \ref{LStLPPCon}.  
\end{proof}

We now turn to developing a corresponding left tail bound. To this end, first note the following right tail bound for the last-passage times defined on paths constrained to enter the bulk at a specific boundary site. 
\begin{lem}
\label{LBulkBdRTBnd}
Let $m, n \in \bbZ_{>0}$ and $\EMin = \cEMin{m, n}$. Let $k \in [m]$, $l \in [n]$, $s > 0$ and $p \in \bbZ_{\ge0}$. Then there exists a constant $C_p > 0$ {\rm(}depending only on $p${\rm)} such that 
\begin{align*}
\csPf{m, n}{\EMin}\{\sG(k, 0) + \G_{k, 1}(m, n) \ge \cEShphs{m-k+1, n}{k-1}+\cEA{k}{\EMin}+s\bigl(\cMinVar{m, n}\bigr)^{1/2}\} &\le C_ps^{-p} \\
\csPf{m, n}{\EMin}\{\sG(0, l) + \G_{1, l}(m, n) \ge \cEShpvs{m, n-l+1}{l-1}+\cEB{l}{\EMin}+s\bigl(\cMinVar{m, n}\bigr)^{1/2}\} &\le C_ps^{-p}. 
\end{align*}
\end{lem}
\begin{proof}
By Lemma \ref{LExpCon} and since $\sum_{i=1}^k (a_i+\EMin)^{-2} \le \cMinVar{m, n}$, 
\begin{align*}
\csPf{m, n}{\EMin}\{\sG(k, 0)-\cEA{k}{\EMin} \ge s \bigl(\cMinVar{m, n}\bigr)^{1/2}\} \le C_p s^{-p}. 
\end{align*}
Write $\wt{\EMin} = \cEMinhs{m-k+1, n}{k-1}$. The ordering $\wt{\EMin} \le \EMin$ is clear from definition \eqref{EShpMin}. Then  
\begin{align*}
\cMinVarhs{m-k+1, n}{k-1} = \sum_{j=1}^{n} \frac{1}{(b_j-\wt{\EMin})^2} \le \sum_{j=1}^n \frac{1}{(b_j-\EMin)^2} = \cMinVar{m, n}. 
\end{align*}
Hence, by Lemma \ref{LBulkRTBnd}, 
\begin{align*}
\cPf{m, n}\{\G_{k, 1}(m, n)-\cEShphs{m-k+1, n}{k-1} \ge s\bigl(\cMinVar{m, n}\bigr)^{1/2}\} \le C_ps^{-p}. 
\end{align*}
The first of the claimed bounds now follows from the triangle inequality and a union bound. The proof of the second bound is analogous. 
\end{proof}

Some of the statements below require an ordering condition on the parameters $\bfa$ and $\bfb$. This does not limit the scope of the tail bounds of $\G$, however, because the distribution of $\G$ is invariant under permutations of the parameters as stated in the next lemma. 
\begin{lem}
\label{LGDisPerInv}
Let $m, n \in \bbZ_{>0}$. Let $\sigma$ and $\tau$ be permutations on $[m]$ and $[n]$, respectively. Then, for $x \in \bbR$,  
\begin{align*}
\cPf{m, n}\{\G(m, n) \le x\} = \Fbs{\P}{m, n}{\sigma\bfa, \tau\bfb}\{\G(m, n) \le x\}. 
\end{align*}
\end{lem}
\begin{proof}
One can readily observe this from the formula \cite[(12)]{boro-pech} for the distribution of $\G$.
\end{proof}

Next comes a comparison of $\cEShp{}$ with essentially (up to a single boundary term) the centering for the LPP values considered in Lemma \ref{LBulkBdRTBnd}. 
Denote the distance from the minimizer  $\EMin$ to the boundary of $\cIz{m, n}$ by 
\begin{align}
\cMinDis{m, n} = \min\{a_m^{\min}+\EMin, b_n^{\min}-\EMin\}. \label{EMinDis}
\end{align}
\begin{lem}
\label{LShpComp}
Let $m, n \in \bbZ_{>0}$, $\EMin = \cEMin{m, n}$, $\MinVar = \cMinVar{m, n}$ and $\MinDis = \cMinDis{m, n}$. There exists an absolute constant $c > 0$ such that the following statements hold. 
\begin{enumerate}[\normalfont (a)]
\item Let $k \in [m]$. If $(a_i)_{i \in [m]}$ is nondecreasing then 
\begin{align*}
\cEA{k-1}{\EMin}+\cEShphs{m-k+1, n}{k-1} - \cEShp{m, n}  \le -c\MinDis\MinVar \bigg(\frac{k-1}{m}\bigg)^2. 
\end{align*}
\item Let $l \in [n]$. If $(b_j)_{j \in [n]}$ is nondecreasing then 
\begin{align*}
\cEB{l-1}{\EMin}+\cEShphs{m, n-l+1}{l-1} - \cEShp{m, n} \le -c\MinDis\MinVar \bigg(\frac{l-1}{n}\bigg)^2.
\end{align*}
\end{enumerate}
\end{lem}
\begin{proof}
We prove only (a) since the proof of (b) is similar. For $z \in \cIz{m, n}$, note the identities 
\begin{align}
\cEStShp{m, n}{z}-\cEShp{m, n} &= \sum_{i=1}^m \frac{1}{a_i+z}-\frac{1}{a_i+\EMin} + \sum_{j=1}^n \frac{1}{b_j-z}-\frac{1}{b_j-\EMin} \nonumber\\
&= (z-\EMin)\bigg\{-\sum_{i=1}^m \frac{1}{(a_i+z)(a_i+\EMin)} + \sum_{j=1}^n \frac{1}{(b_j-z)(b_j-\EMin)}\bigg\} \nonumber\\
&= (z-\EMin)^2 \bigg\{\sum_{i=1}^m \frac{1}{(a_i+z)(a_i+\EMin)^2}+ \sum_{j=1}^n \frac{1}{(b_j-z)(b_j-\EMin)^2}\bigg\}. \label{E35}
\end{align}
The last equality is obtained by adding $\displaystyle \sum_{i=1}^m \frac{1}{(a_i+\EMin)^2}-\sum_{j=1}^n \frac{1}{(b_j-\EMin)^2} = 0$ to the previous line. Next note that, for any $z \in (-a_{k, m}^{\min}, b_n^{\min})$, 
\begin{align*}
\cEA{k-1}{\EMin}+\cEShphs{m-k+1, n}{k-1} &\le \cEA{k-1}{\EMin}+\cEStShphs{m-k+1, n}{z}{k-1} \\
&= \cEA{k-1}{\EMin}-\cEA{k-1}{z}+\cEStShp{m, n}{z} \\
&= -\sum_{i=1}^{k-1}\frac{\EMin-z}{(a_i+\EMin)(a_i+z)} +\cEShp{m, n} \\
&+ (\EMin-z)^2 \bigg(\sum_{i=1}^m \frac{1}{(a_i+\EMin)^2(a_i+z)} + \sum_{j=1}^n \frac{1}{(b_j-\EMin)^2 (b_j-z)}\bigg). 
\end{align*}
Set $z = \EMin-\dfrac{c(k-1)\MinDis}{m}$ for some absolute constant $c \in (0, 1/2]$ to be selected below. This is a legitimate value for $z$ since $k \le m$ and $\MinDis \le a_m^{\min}+\EMin$. Then, using the monotonicity of $(a_i)_{i \in [m]}$ and bounding term by term, one obtains that  
\begin{align*}
&\cEA{k-1}{\EMin}+\cEShphs{m-k+1, n}{k-1} - \cEShp{m, n} \nonumber\\
&\le - \frac{k-1}{m}\sum_{i=1}^{m}\frac{\EMin-z}{(a_i+\EMin)(a_i+z)}+(\EMin-z)^2 \bigg(\sum_{i=1}^m \frac{1}{(a_i+\EMin)^2(a_i+z)} + \sum_{j=1}^n \frac{1}{(b_j-\EMin)^2 (b_j-z)}\bigg) \nonumber\\
&\le -\bigg(\frac{k-1}{m}\bigg)^2c\MinDis \sum_{i=1}^m \frac{1}{(a_i+\EMin)^2}+\bigg(\frac{k-1}{m}\bigg)^2c^2\MinDis^2\bigg(\sum_{i=1}^m \frac{2}{(a_i+\EMin)^3} + \sum_{j=1}^n \frac{1}{(b_j-\EMin)^3}\bigg) \nonumber\\
&\le -\bigg(\frac{k-1}{m}\bigg)^2c\MinDis \sum_{i=1}^m \frac{1}{(a_i+\EMin)^2}+\bigg(\frac{k-1}{m}\bigg)^2c^2\MinDis\bigg(\sum_{i=1}^m \frac{2}{(a_i+\EMin)^2} + \sum_{j=1}^n \frac{1}{(b_j-\EMin)^2}\bigg) \nonumber\\
&= -\bigg(\frac{k-1}{m}\bigg)^2(c-3c^2)\MinDis \sum_{i=1}^m \frac{1}{(a_i+\EMin)^2} = -\bigg(\frac{k-1}{m}\bigg)^2(c-3c^2)\MinDis\MinVar 
\end{align*}
The second inequality above uses $a_i+z \ge \frac{1}{2}(a_i+\EMin)$, which follows from $c\le 1/2$ in the choice of $z$. The subsequent steps use \eqref{EMinVar}. Let $c < 1/3$ and rename $c-3c^2$ as $c$.  
\end{proof}

With the aid of the preceding estimates, one can derive the upper bounds below for the exit probabilities. 

\begin{lem}
\label{LExitPr}
Let $m, n \in \bbZ_{>0}$, $\EMin = \cEMin{m, n}$, $\MinVar = \cMinVar{m, n}$ and $\MinDis = \cMinDis{m, n}$. Let $s > 0$ and $p \in \bbZ_{> 0}$. Then there exist an absolute constant $s_0 > 0$ and a constant $C_p > 0$ {\rm(}depending only on $p${\rm)} such that the following bounds hold subject to the indicated further assumptions. 
\begin{enumerate}[\normalfont (a)]
\item Let $k \in \bbZ$. Assume that $(a_i)_{i \in [m]}$ is nondecreasing, $s \ge s_0$ and $k \ge 1 + \dfrac{sm}{\MinDis^{1/2}\MinVar^{1/4}}$. Then 
\begin{align*}
\csPf{m, n}{\EMin}\{\cMiEh{m, n} \ge k\} \le C_pms^{-p}. 
\end{align*}
\item Let $l \in \bbZ$. Assume that $(b_j)_{j \in [n]}$ is nondecreasing, $s \ge s_0$ and $l \ge 1 + \dfrac{sn}{\MinDis^{1/2}\MinVar^{1/4}}$. Then 
\begin{align*}
\csPf{m, n}{\EMin}\{\cMiEv{m, n} \ge l\} \le C_pns^{-p}. 
\end{align*}
\end{enumerate}
\end{lem}
\begin{proof}
We prove only (a) since the proof of (b) is analogous. One may assume that $k \le m$ since the probability is zero otherwise. Set $s_0 = 2c^{-1/2}$ where $c > 0$ is the absolute constant in Lemma \ref{LShpComp}(a). Then, by the lemma, 
\begin{align*}
\cEShp{m, n}-\cEShphs{m-k+1, n}{k-1}-\cEA{k}{\EMin} &\ge c\MinDis\MinVar \bigg(\frac{k-1}{m}\bigg)^2-\frac{1}{a_k+\EMin} \\
&> cs^2\sqrt{\MinVar}-\sqrt{\MinVar} \ge  \tfrac{1}{2}cs^2 \sqrt{\MinVar}. 
\end{align*}
Then a union bound combined with the tail bounds in Lemmas \ref{LStLPPCon} and \ref{LBulkBdRTBnd} yields
\begin{align*}
\csPf{m, n}{\EMin}\{\cMiEh{m, n} = k\} &= \csPf{m, n}{\EMin}\{\sG(m, n) = \sG(k, 0) + \G_{k, 1}(m, n)\} \\
&\le \csPf{m, n}{\EMin}\{\sG(m, n) \le \cEShp{m, n}-\frac{cs^2}{4}\sqrt{\MinVar}\} \\
&+ \csPf{m, n}{\EMin}\{\sG(k, 0) + \G_{k, 1}(m, n) \ge \cEA{k}{\EMin}+\cEShphs{m-k+1, n}{k-1}+\frac{cs^2}{4}\sqrt{\MinVar}\} \\
&\le C_ps^{-2p}. 
\end{align*}
Now applying the last inequality with another union bound results in  
\begin{equation*}
\csPf{m, n}{\EMin}\{\cMiEh{m, n} \ge k\} \le \sum_{i=k}^m \csPf{m, n}{\EMin}\{\cMiEh{m, n} = i\} \le C_pms^{-2p}, 
\end{equation*}
which implies the claim in (a). 
\end{proof}

The next lemma is a provisional left tail bound for the last-passage times. 
\begin{lem}
\label{LLTailMinCon}
Let $m, n \in \bbZ_{>0}$, $\EMin = \cEMin{m, n}$, $\MinVar = \cMinVar{m, n}$ and $\MinDis = \cMinDis{m, n}$. Let $s > 0$ and $p \in \bbZ_{> 0}$. Then there exist an absolute constant $s_0 > 0$ and a constant $C_p > 0$ {\rm(}depending only on $p${\rm)} such that 
\begin{align*}
\cPf{m, n}\{\G(m, n) \le \cEShp{m, n}-s\MinDis^{-1/4}\MinVar^{3/8}(m+n)^{1/2}\} \le C_p(m+n)s^{-p} \quad \text{ for } s \ge s_0.  
\end{align*}
\end{lem}
\begin{proof}
Let $k, l \in \bbZ_{>0}$ and $x \in \bbR$. Note from the definitions \eqref{EhorExit}--\eqref{EverExit} that, on the event that $\cMiEh{m, n} \le k$ and $\cMiEv{m, n} \le l$, the inequality  
\begin{align*}\G(m, n) \ge \sG(m, n)-\sG(k, 0)-\sG(0, l)\end{align*} 
holds. Using this with the union bound leads to   
\begin{align}
\cPf{m, n}\{\G(m, n) \le x\} &= \csPf{m, n}{\EMin}\{\G(m, n) \le x\} \nonumber \\
&\le \csPf{m, n}{\EMin}\{\cMiEh{m, n} \ge k+1\} + \csPf{m, n}{\EMin}\{\cMiEv{m, n} \ge l+1\}\label{E8}\\
&+ \csPf{m, n}{\EMin}\{\sG(m, n)-\sG(k, 0)-\sG(0, l) \le x\}. \label{E9}
\end{align}

Let $s \ge \sqrt{s_0}$ where $s_0 > 0$ denotes the constant from Lemma \ref{LExitPr}. Choose 
\begin{align*}
k = \min \{\lc ms^2\MinDis^{-1/2}\MinVar^{-1/4}\rc, m\} \quad \text{ and } \quad l = \min\{\lc ns^2\MinDis^{-1/2}\MinVar^{-1/4}\rc, n\}.
\end{align*} 
By virtue of Lemma \ref{LGDisPerInv}, the sequences $(a_i)_{i \in [m]}$ and $(b_j)_{j \in [n]}$ can be assumed to be nondecreasing without loss of generality. Also since $s^2 \ge s_0$, in the case $s^2\MinDis^{-1/2}\MinVar^{-1/4} \le 1$, Lemma \ref{LExitPr} gives 
\begin{align*}
\eqref{E8} \le C_p(m+n)s^{-p}
\end{align*}
for some constant $C_p > 0$. The last bound also holds trivially in the case $s^2\MinDis^{-1/2}\MinVar^{-1/4} > 1$ because then $\eqref{E8} = 0$ since $k = m$ and $l = n$. 
Set $x = \cEShp{m, n}-y$ where $y = cs\MinDis^{-1/4}\MinVar^{3/8}(m+n)^{1/2}$ and $c > 0$ is an absolute constant to be determined below. Another union bound yields 
\begin{align*}
\eqref{E9} &\le \csPf{m, n}{\EMin}\{\sG(m, n) \le \cEShp{m, n}-y/2\} + \csPf{m, n}{\EMin}\{\sG(k, 0) \ge y/4\} \\
&\qquad 
+ \csPf{m, n}{\EMin}\{\sG(0, l) \ge y/4\}. 
\end{align*}
It follows from definitions \eqref{EMinVar} and \eqref{EMinDis} that $\MinVar \le \min\{m, n\}\MinDis^{-2}$. Using this bound with Lemma \ref{LStLPPCon}, one obtains that 
\begin{align*}
\csPf{m, n}{\EMin}\{\sG(m, n) \le \cEShp{m, n}-y/2\} \le \frac{C_p\MinVar^{p/2}}{y^p} = \frac{C_p\MinDis^{p/4} \MinVar^{p/8}}{s^p (m+n)^{p/2}} \le \frac{C_p}{s^p (m+n)^{3p/8}}  \le \frac{C_p}{s^p}. \label{E20}
\end{align*}
By the Cauchy-Schwarz inequality and and the choices of $k$ and $y$, 
\begin{align}
\cEA{k}{\EMin} = \sum_{i=1}^k \frac{1}{a_i+\EMin} \le \sqrt{k} \biggl\{\sum_{i=1}^k \frac{1}{(a_i+\EMin)^2}\biggr\}^{1/2} \le \sqrt{k}\sqrt{\MinVar} \le s\MinDis^{-1/4}\MinVar^{3/8}m^{1/2} \le \frac{y}{8}, 
\end{align}
provided that $c \ge 8$. Therefore, by Lemma \ref{LExpCon} and since $\sum_{i=1}^k (a_i+\EMin)^{-2} \le \MinVar$ (as noted in \eqref{E20}),
\begin{align*}
\csPf{m, n}{\EMin}\{\sG(k, 0) \ge y/4\} \le \csPf{m, n}{\EMin}\{\sG(k, 0) \ge \cEA{k}{\EMin}+y/8\} \le \frac{C_p}{(y\MinVar^{-1/2})^{p}} = \frac{C_p\MinDis^{p/4}\MinVar^{p/8}}{s^p (m+n)^{p/2}} \le \frac{C_p}{s^p}. 
\end{align*}
The last step uses $\MinVar \le m\Delta^{-2}$ once more and drops the factor $(m+n)^{3p/8}$. In the same vein, 
\begin{align*}
\csPf{m, n}{\EMin}\{\sG(0, l) \ge y/4\} \le C_ps^{-p}. 
\end{align*}
Putting the preceding bounds together results in $\eqref{E9} \le C_p s^{-p}$. Hence, 
\begin{align*}\cPf{m, n}\{\sG(m, n) \le x\} \le \eqref{E8} + \eqref{E9} \le C_p(m+n)s^{-p}.\end{align*} 
This implies the claim upon replacing $s$ with $s/c$ throughout.   
\end{proof}

We next derive a nontrivial left tail bound that does not depend on the location of the minimizer. One ingredient in our argument is the following set of elementary estimates on the minimizer with shifted parameters. 
\begin{lem}
\label{LMinShfEst}
Let $m, n \in \bbZ_{>0}$ and $\EMin = \cEMin{m, n}$.  
\begin{enumerate}[\normalfont (a)]
\item Assume that $m > 1$ and write $\wt{\EMin} = \cEMinhs{m-1, n}{1}$. If $a_1 = a_m^{\min}$ then $a_{2, m}^{\min}+\wt{\EMin} \ge \dfrac{a_1+\EMin}{\sqrt{2}}$. Also, if 
$b_n^{\min}-\EMin \le \dfrac{a_1+\EMin}{\sqrt{2}}$ then 
$b_n^{\min}-\wt{\EMin} \le \sqrt{2}(b_n^{\min}-\EMin)$. 
\item Assume that $n > 1$ and write $\wt{\EMin} = \cEMinvs{m, n-1}{1}$. If $b_1 = b_n^{\min}$ then $b_{2, n}^{\min}-\wt{\EMin} \ge \dfrac{b_1-\EMin}{\sqrt{2}}$. Also, if $a_m^{\min}+\EMin \le \dfrac{b_1-\EMin}{\sqrt{2}}$ then 
$a_m^{\min}+\wt{\EMin} \le \sqrt{2}(a_m^{\min}+\EMin)$. 
\end{enumerate}
\end{lem}
\begin{proof}
We only verify (a) since (b) is entirely analogous. A moment of inspecting \eqref{EShpMin} reveals that $\wt{\EMin} \le \EMin$. Hence, the second and last inequalities in the following derivation. The first inequality is by the assumption $a_1 \le a_{2, m}^{\min}$, and the equality comes again from \eqref{EShpMin}. 
\begin{align*}
&\frac{2}{(a_1+\EMin)^2}- \frac{1}{(a_{2, m}^{\min}+\wt{\EMin})^2} \ge \frac{1}{(a_1+\EMin)^2} + \frac{1}{(a_{2, m}^{\min}+\EMin)^2}-\frac{1}{(a_{2, m}^{\min}+\wt{\EMin})^2} \\ 
&\qquad\qquad
 \ge \sum_{i=1}^m \frac{1}{(a_i+\EMin)^2}-\sum_{i=2}^m \frac{1}{(a_i+\wt{\EMin})^2}  
= \sum_{j=1}^n \frac{1}{(b_j-\EMin)^2}-\sum_{j=1}^n \frac{1}{(b_j-\wt{\EMin})^2} \ge 0. 
\end{align*}
Hence, $a_{2, m}^{\min}+\wt{\EMin} \ge \dfrac{a_1+\EMin}{\sqrt{2}}$ as claimed. Using $\wt{\EMin} \le \EMin$ and \eqref{EShpMin} also yields 
\begin{align*}
\frac{1}{(a_1+\EMin)^2}+\frac{1}{(b_n^{\min}-\wt{\EMin})^2}-\frac{1}{(b_n^{\min}-\EMin)^2} &\ge \frac{1}{(a_1+\EMin)^2}+\sum_{j=1}^n \frac{1}{(b_j-\wt{\EMin})^2}-\sum_{j=1}^n \frac{1}{(b_j-\EMin)^2} \\
&= \sum_{i=2}^m \frac{1}{(a_i+\wt{\EMin})^2}-\sum_{i=2}^m \frac{1}{(a_i+\EMin)^2} \ge 0. 
\end{align*}
Now if $\dfrac{b_n^{\min}-\EMin}{a_1+\EMin} \le \dfrac{1}{\sqrt{2}}$ then $b_n^{\min}-\wt{\EMin} \le \sqrt{2} (b_n^{\min}-\EMin)$. 
\end{proof}

\begin{lem}
\label{LLTail}
Let $m, n \in \bbZ_{>0}$ and $\MinVar = \cMinVar{m, n}$. Let $p \in \bbZ_{>0}$. Then there exist absolute constants $s_0, s_1 > 0$ and a constant $C_p > 0$ {\rm(}depending only on $p${\rm)} such that 
\begin{align*}
\cPf{m, n}\{\G(m, n) \le \cEShp{m, n}-s\MinVar^{1/2}(m+n)^{2/5}\} \le C_ps^{-p}
\end{align*}
provided that $s_0 \le s \le s_1 (a_m^{\min}+b_n^{\min})\MinVar^{1/2}(m+n)^{-2/5}$. 
\end{lem}
\begin{proof}
Introduce a threshold $0 < \delta \le (a_m^{\min}+b_n^{\min})/2$ to be determined later. Let
\begin{align}
K &= \min\{k \in [m]: \cEMinhs{m-k+1, n}{k-1} \le b_n^{\min}-\delta\} \nonumber\\
L &= \min\{l \in [n]: \cEMinvs{m, n-l+1}{l-1} \ge -a_m^{\min}+\delta\}. \label{Ldef}
\end{align}
Both sets above are nonempty because 
\begin{align*}
\cEMinhs{1, n}{m-1} &\le b_n^{\min}-(a_m+b_n^{\min})/2 \le b_n^{\min}-\delta \\
\cEMinvs{m, 1}{n-1} &\ge -a_m^{\min}+(a_m^{\min}+b_n)/2 \ge -a_m^{\min}+\delta, 
\end{align*} 
where the first inequalities on both lines can be readily seen from definition \eqref{EShpMin}. Also, since the intersection $(-a_m^{\min}, -a_m^{\min}+\delta) \cap (b_n^{\min}-\delta, b_n^{\min})$ is empty, the inequalities $K > 1$ and $L > 1$ cannot both hold. Appealing to the row-column symmetry, let us assume that $K = 1$ for concreteness. 

On account of Lemma \ref{LGDisPerInv}, without loss in generality, the parameters $(a_i)_{i \in [m]}$ and $(b_j)_{j \in [n]}$ can be reordered to be nondecreasing. Work with 
\begin{align}
\delta \le (\sqrt{2}-1)(a_1+b_1) \label{E29}
\end{align} 
here on. Abbreviate $\EMin_l = \cEMinvs{m, n-l+1}{l-1}$ for $l \in [L]$. \emph{Claim:} 
\begin{align}
\label{E21}
\cMinDisvs{m, n-L+1}{L-1} = \min \{a_1+\EMin_L, b_L-\EMin_{L}\} \ge \delta. 
\end{align}
The inequality $a_1+\EMin_L \ge \delta$ holds by definition \eqref{Ldef}. If $L = 1$ then $b_L-\EMin_{L} = b_1-\EMin_1 \ge \delta$ since $K = 1$. Now consider the case $L > 1$. Then $a_1+\EMin_{L-1} < \delta$ and 
\begin{align*}b_{L-1}-\EMin_{L-1} \ge b_1-\EMin_{L-1} = (a_1+b_1)-(a_1+\EMin_{L-1}) \ge (a_1+b_1)-\delta \ge \sqrt{2}\delta \ge \sqrt{2}(a_1+\EMin_{L-1}).\end{align*} 
Hence, invoking Lemma \ref{LMinShfEst}(b) with the parameters $\bfa, \tau_{L-2}\bfb$ and the lattice coordinates $m$ and $n-L+2 > 1$ yields 
\begin{align}
\label{E15}
b_{L}-\EMin_{L} \ge \frac{b_{L-1}-\EMin_{L-1}}{\sqrt{2}} \ge \delta \quad \text{ and } \quad a_1+\EMin_{L} \le \sqrt{2}(a_1+\EMin_{L-1}) \le \sqrt{2}\delta.
\end{align}
The first inequality above completes the verification of \eqref{E21}. 

Write $\MinDis = \cMinDis{m, n}$. One can read off from definition \eqref{EShpMin} that the shifted minimizer $\Fas{\EMin}{m-k+1, n-l+1}{\tau_{k-1}\bfa, \tau_{l-1}\bfb}$ is nonincreasing in $k \in [m]$ and nondecreasing in $l \in [n]$. Hence, the inequality below. 
\begin{align}
\label{E22}
\MinVar = \sum_{i=1}^{m} \frac{1}{(a_i+\EMin_1)^2} \ge \sum_{i=1}^m \frac{1}{(a_i+\EMin_L)^2} = \cMinVarvs{m, n-L+1}{L-1}. 
\end{align}
Now conclude from Lemma \ref{LLTailMinCon} and the bounds in \eqref{E21} and \eqref{E22} that  
\be\begin{aligned}
&\cPf{m, n}\bigl\{\G_{1, L}(m, n) \le \cEShpvs{m, n-L+1}{L-1}-u\delta^{-1/4}\MinVar^{3/8}(m+n)^{1/2}\bigr\} \\
&\qquad\qquad
\le C_p(m+n)u^{-p} 
\label{E23}
\end{aligned} \ee
whenever $u \ge u_0$ for some constants $u_0, C_p > 0$. Let $v > 0$ and $q \in \bbZ_{\ge 0}$. By virtue of Lemma \ref{LExpCon} and since 
\begin{align}
\label{E111}
\MinVar = \sum_{j=1}^n (b_j-\EMin_1)^{-2} \ge \sum_{j=1}^n (a_1+b_j)^{-2}, 
\end{align}
one also has
\begin{align}
\cPf{m, n}\{\G(1, L-1) \le \cEB{L-1}{-a_1}-v\sqrt{\MinVar}\} \le C_q v^{-q} \label{E24}
\end{align}
for some constant $C_q > 0$. Here, interpret $\G(1, 0) = 0$ that arises in the case $L = 1$. 

The next task is to establish that 
\begin{align}
\cEShp{m, n} \le \cEShpvs{m, n-L+1}{L-1} + \cEB{L-1}{-a_1} + 4\delta \MinVar. \label{E18}
\end{align}
Assume that $L > 1$ since \eqref{E18} holds trivially otherwise. One obtains from the second set of inequalities in \eqref{E15} that 
\begin{align}
\label{E17}
b_1-\EMin_{L} = (a_1+b_1)-(a_1+\EMin_{L}) \ge (a_1+b_1)-\sqrt{2}\delta \ge (\sqrt{2}-1)(a_1+b_1) \ge \delta > 0. 
\end{align}
Thus, $\EMin_{L} \in (-a_1, b_1)$ is an admissible $z$-parameter in the next derivation. Furthermore, 
\begin{align}
\label{E110}
\frac{b_j-\EMin_{L}}{a_1+b_j} \ge \frac{b_1-\EMin_{L}}{a_1+b_1} \ge \sqrt{2}-1 \quad \text{ for } j \in [n]
\end{align}
by \eqref{E17} and the monotonicity of $(b_j)_{j \in [n]}$. One now develops from definition \eqref{EECenter}, the second bound in \eqref{E15} and estimates \eqref{E111} and \eqref{E110} that 
\begin{align*}
&\cEShp{m, n}-\cEShpvs{m, n-L+1}{L-1}= \cEShp{m, n}-\cEStShpvs{m, n-L+1}{\EMin_{L}}{L-1}\\
&\qquad\qquad\le \cEStShp{m, n}{\EMin_{L}}-\cEStShpvs{m, n-L+1}{\EMin_{L}}{L-1} = \cEB{L-1}{\EMin_{L}}\\
&\qquad\qquad= \cEB{L-1}{-a_1} + \sum_{j=1}^{L-1} \frac{a_1+\EMin_{L}}{(a_1+b_j)(b_j-\EMin_{L})} \\
&\qquad\qquad\le \cEB{L-1}{-a_1}+ \sqrt{2}\delta\sum_{j=1}^{L-1} \frac{1}{(a_1+b_j)(b_j-\EMin_{L})}  \\
&\qquad\qquad \le \cEB{L-1}{-a_1}+ 4\delta \sum_{j=1}^{L-1} \frac{1}{(a_1+b_j)^2}\\
&\qquad\qquad\le \cEB{L-1}{-a_1} + 4\delta \MinVar, 
\end{align*}
which verifies the claim in \eqref{E18}. 

Now combine \eqref{E23} and \eqref{E24} with a union bound, and then use \eqref{E18} and that $\G(m, n) \stackrel{\text{a.s.}}{\ge} \G(1, L-1) + \G_{1, L}(m, n)$ to arrive at 
\begin{align}
&\cPf{m, n}\{\G(m, n) \le \cEShp{m, n}-y\} \le C_p(m+n)u^{-p} + C_q v^{-q},  \label{E25}
\end{align}
where $y = u\delta^{-1/4}\MinVar^{3/8}(m+n)^{1/2}+v\MinVar^{1/2}+4\delta \MinVar$. The terms on the right-hand side become comparable upon setting $\delta = s\MinVar^{-1/2}(m+n)^{2/5}$ and $v = s(m+n)^{2/5}$ where $s = u^{4/5}$. Thus, $s_0$ can be chosen as $u_0^{4/5}$. For the choice of $\delta$ to meet condition \eqref{E29}, it suffices to have $s \le (\sqrt{2}-1)(a_1+b_1)\MinVar^{1/2}(m+n)^{-2/5}$. 
 With the preceding choices, $y = 6s\MinVar^{1/2}(m+n)^{2/5}$ and the bound in \eqref{E25} reads 
\begin{align*}C_p (m+n)s^{-5p/4} + C_q s^{-q}(m+n)^{-2q/5}.\end{align*} 
Enlarging $p$ by a factor $5/4$ and choosing $q \ge 5p/4$ makes the first term more dominant than the second and results in the bound 
\begin{equation*}
\cPf{m, n}\{\G(m, n) \le \cEShp{m, n}-6s\MinVar^{1/2}(m+n)^{2/5}\} \le C_p(m+n)s^{-p}. 
\end{equation*}
Redefining $s$ to be $s/6$ and adjusting the constant completes the proof. 
\end{proof}

\section{Proof of Theorem \ref{TEmpShp}}
\label{SEmpShp}

\begin{proof}[Proof of Theorem \ref{TEmpShp}]
Let $0 < q < p$. A Borel-Cantelli argument and Lemma \ref{LBulkRTBnd} imply that, $\cP$-a.s., there exists a random $M \in \bbZ_{>0}$ such that  
\begin{align}
\label{E27}
\cGp{m, n} \le \cEShp{m, n} + (m+n)^{q}\bigl(\cMinVar{m, n}\bigr)^{1/2} \quad \text{ whenever } m+n \ge M. 
\end{align}
Since $\EMin$ is at least $\frac{1}{2}|\cIz{m, n}|$ away from one of the endpoints of $\cIz{m, n} = (-\min \ca{m}{}, \min \cb{n}{})$, one concludes from definition \eqref{EMinVar} the trivial bound 
\begin{align}
\cMinVar{m, n} \le 4\max\{m, n\}|\cIz{m, n}|^{-2}. \label{EVarUB}
\end{align}
Inserting this into \eqref{E27} yields 
\begin{align*}
\cGp{m, n} \le \cEShp{m, n} + 2(m+n)^{q+1/2}|\cIz{m, n}|^{-1} \quad \text{ whenever } m+n \ge M. 
\end{align*}
Since $q < p$, the upper bound in the theorem follows. 

Now the lower bound. Since $\Gp \stackrel{\text{a.s.}}{\ge 0}$, it suffices to consider $m, n \in \bbZ_{>0}$ such that 
\begin{align*}
\cEShp{m, n} \ge |\cIz{m, n}|^{-1} (m+n)^{9/10+p}. 
\end{align*}
Then, by the Cauchy-Schwarz inequality, 
\begin{align}
\cMinVar{m, n} &= \frac{1}{2}\bigg\{\sum_{i=1}^m \frac{1}{(a_i+\EMin)^2} + \sum_{j=1}^n \frac{1}{(b_j-\EMin)^2}\bigg\} \ge \frac{1}{2(m+n)}\bigg\{\sum_{i=1}^m \frac{1}{a_i+\EMin} + \sum_{j=1}^n \frac{1}{b_j-\EMin}\bigg\}^2 \nonumber\\
&= \frac{\cEShp{m, n}^2}{2(m+n)} \ge \frac{1}{2}|\cIz{m, n}|^{-2}(m+n)^{4/5+2p}, \nonumber
\end{align}
where $\EMin = \cEMin{m, n}$. Hence, 
\begin{align}
|\cIz{m, n}| \MinVar^{1/2}(m+n)^{-2/5} &\ge \frac{1}{\sqrt{2}} (m+n)^{p}. \label{E55}
\end{align}
Let $s_0, s_1 > 0$ denote the absolute constants in Lemma \ref{LLTail}. By \eqref{E55} and since $0 < q < p$, 
\begin{align*}
s_0 \le (m+n)^{q} \le s_1 |\cIz{m, n}| \MinVar^{1/2}(m+n)^{-2/5} \quad \text{ whenever } m+n \ge N_0
\end{align*}
for some sufficiently large constant $N_0 \in \bbZ_{>0}$. Also, by \eqref{EVarUB}, 
\begin{align}
\MinVar^{1/2}(m+n)^{2/5} \le 2|\cIz{m, n}|^{-1}(m+n)^{9/10}. \label{E56}
\end{align}
Let $u > 0$. By \eqref{E56} and an application of Lemma \ref{LLTail} with $s = (m+n)^q$, one obtains that 
\begin{align*}
&\cP\{\cGp{m, n} \le \cEShp{m, n}-2|\cIz{m, n}|^{-1} (m+n)^{9/10+q}\} \\
&\le \cP\{\cGp{m, n} \le \cEShp{m, n}-(m+n)^q (\MinVar^{1/2}(m+n)^{2/5})\} \le \frac{C}{(m+n)^{qu}}
\end{align*}
whenever $m+n \ge N_0$ and for some constant $C > 0$ dependent only on $u$. The last bound is summable over $\bbZ_{>0}^2$ provided that $u$ is sufficiently large. 
Hence, by the Borel-Cantelli lemma, $\cP$-a.s., there exists a random $N \ge N_0$ such that 
\begin{align*}
\cGp{m, n} \ge \cEShp{m, n} - 2|\cIz{m, n}|^{-1} (m+n)^{9/10+q} \quad \text{ whenever } m+n \ge N.  
\end{align*} 
This implies the claimed lower bound since $q < p$. 
\end{proof}

\section{Proof of Theorem \ref{TRains}}
\label{SPrRains}

\begin{proof}[Proof of Theorem \ref{TRains}]
Fix $n \in \bbZ_{>0}$ for the moment. Since the weights in \eqref{ERainsWp} are a.s.\ nonnegative, the $\G_n$-process is a.s.\ coordinatewise nondecreasing. Therefore, $\cP$-a.s.,  
\begin{align}
\sup_{i, j \in \bbZ_{>0}} \G_n(i, j) = \sup_{l \in \bbZ_{>0}} \G_{n}(ln, ln) = \lim_{l \rightarrow \infty} \G_{n}(ln, ln)=\mathcal{G}_n \label{E91}
\end{align}
where the last equality defines  $\mathcal{G}_n$. 

Consider real parameters $\wt{\a}^n$ and $\wt{\b}^n$ subject to \eqref{AsmParam} (with $\a = \wt{\a}^n$ and $\b = \wt{\b}^n$) that also satisfy
\begin{align*}
\wt{\a}_{ln}^n(i) = a_{\lc i/n \rc} \quad \text{ and } \wt{\b}_{ln}^n(i) = b_{\lc i/n \rc} \quad \text{ for } l \in \bbZ_{>0} \text{ and } i \in [ln]. 
\end{align*}
Then $\w_n(i, j) = \w^{\wt{\a}^n, \wt{\b}^n}_{ln, ln}(i, j)$ for $i, j \in [ln]$ where the right-hand side is given by \eqref{EWPdef} (with $\w_{m, n}(i, j) = \w(i, j)$). Hence, 
\begin{align}\G_n(ln, ln) = \G^{\wt{\a}^n, \wt{\b}^n}(ln, ln) \quad \text{ for } l \in \bbZ_{>0}. \label{E92} \end{align} 

From definition \eqref{EECenter}, one has
\begin{align*}
\EShp^{\wt{\a}^n, \wt{\b}^n}(ln, ln) = n \inf_{z \in (-a_{l}^{\min}, b_{l}^{\min})} \bigg\{\sum_{i=1}^{l} \frac{1}{a_i+z} + \sum_{j=1}^{l} \frac{1}{b_j-z}\bigg\} = n\cEShp{l, l}. 
\end{align*}
Similarly, by definition \eqref{EMinVar}, 
\begin{align}
\MinVar^{\wt{\a}^n, \wt{\b}^n}(ln, ln) = n \cMinVar{l, l} \quad \text{ for } l \in \bbZ_{>0}. \label{E96}
\end{align}
Since $(-a_{l}^{\min}, b_{l}^{\min}) \supset (-\inf \a,  \inf \b)$, it is clear from \eqref{ERainsCent} that
\begin{align}
\label{E90}
\cEShp{l, l} \le \cEShp{\infty, \infty} \quad \text{ for } l \in \bbZ_{>0}. 
\end{align}
Write $z_0 = (\inf \a - \inf \b)/2$ and $\EMin_l = \cEMin{l, l}$ for $l \in \bbZ_{>0}$. Recalling definitions \eqref{EMinVar} and 
\eqref{EMinDis}, note also that 
\begin{align}
\cMinDis{l, l}^{-2} &= \max \{(a_{l}^{\min} + \EMin_l)^{-2}, (b_{l}^{\min}-\EMin_l)^{-2}\} \le \cMinVar{l, l} = \sum_{i=1}^{l} \frac{1}{(a_i+\EMin_l)^2} = \sum_{j=1}^{l} \frac{1}{(b_j-\EMin_l)^2} \nonumber \\ 
&\le \max \bigg\{\sum_{i=1}^l \frac{1}{(a_i+z_0)^2}, \sum_{j=1}^l \frac{1}{(b_j-z_0)^2}\bigg\} \label{E197}\\ 
&\le \frac{2}{\inf \a + \inf \b} \max \bigg\{\sum_{i=1}^\infty \frac{1}{a_i+z_0}, \sum_{j=1}^\infty \frac{1}{b_j-z_0}\bigg\} \label{E97}
\end{align}
Bound \eqref{E197} follows upon considering the cases $\EMin_l \le z_0$ and $\EMin_l \ge z_0$ separately. 
The last inequality uses that $z_0$ is the midpoint of $(-\inf \a, \inf \b)$. Denote the quantity in \eqref{E97} by $c_0 > 0$. By assumption \eqref{AsmRainsConv}, $c_0 < \infty$. 
Now, because $\cMinDis{l, l} \ge c_0^{-1/2}$ for $l \in \bbZ_{>0}$, there exist $\delta > 0$ and $L_0 \in \bbZ_{>0}$ such that $\EMin_l \in I = (-\inf \a+\delta, \inf \b-\delta)$ for $l \ge L_0$. Then
\begin{align}
\cEShp{\infty, \infty} &\le \inf_{z \in I} \cEStShp{\infty, \infty}{z} \nonumber \\
&\le \inf_{z \in I} \cEStShp{l, l}{z} + \sum_{i=l+1}^\infty \frac{1}{a_i-\inf \a + \delta} + \sum_{j=l+1}^{\infty} \frac{1}{b_j-\inf \b + \delta} \label{E94}
\end{align}
The first term equals $\cEShp{l, l}$ for $l \ge L_0$ while both series vanish as $l \rightarrow \infty$ in \eqref{E94}. Combining with \eqref{E90}, one concludes that 
\begin{align}
\lim_{l \rightarrow \infty} \cEShp{l, l} = \cEShp{\infty, \infty}. \label{E95}
\end{align}

Let $p > 0$. Apply Lemma \ref{LBulkRTBnd} using \eqref{E90} and \eqref{E97} to obtain
\begin{align*}
&\cP\{\G_{n}(ln, ln) \ge n \cEShp{\infty, \infty} + sc_0^{1/2}n^{1/2}\}  \\
&\le \cP\{\G_{n}(ln, ln) \ge n \cEShp{l, l} + s(\cMinVar{l, l})^{1/2}n^{1/2}\} \le C_ps^{-p} \quad \text{ for } s > 0
\end{align*}
for some constant $C_p > 0$ depending only on $p$. Passing to the limit $l \rightarrow \infty$ results in 
\begin{align}
\cP\{\mathcal{G}_n > n\cEShp{\infty, \infty} + sc_0^{1/2}n^{1/2}\} \le C_ps^{-p} \quad \text{ for } s > 0 \label{E98}
\end{align}
on account of \eqref{E91}. Let $\eta > 0$. By \eqref{E98} and the Borel-Cantelli argument, $\cP$-a.s., there exists a random $L_1 \in \bbZ_{>0}$ such that 
\begin{align*}
\mathcal{G}_n \le n\cEShp{\infty, \infty}  + n^{1/2+\eta} \quad \text{ for } n \ge L_1.
\end{align*}
In particular, with $\eta < 1/2$, 
\begin{align}\limsup_{n \rightarrow \infty} n^{-1}\mathcal{G}_n \le \cEShp{\infty, \infty}. \label{E99}\end{align}

Note from the bounds culminating in \eqref{E97} that 
\begin{align*}
\max \{\cMinDis{l, l}^{-2}, \cMinVar{l, l}\} \le c_0 \quad \text{ for } l \in \bbZ_{>0}, 
\end{align*}
which justifies the first step in the next derivation. 
\begin{align*}
&\cP\{\G_{n}(ln, ln) \le n\cEShp{l, l}-sc_0^{1/2}l^{1/2}n^{7/8}\} \\ 
&\le \cP\{\G_{n}(ln, ln) \le n\cEShp{l, l}-s(\cMinDis{l, l})^{-1/4}(\cMinVar{l, l})^{3/8}l^{1/2}n^{7/8}\} \\
&= \cP\{\G^{\wt{\a}^n, \wt{\b}^n}(ln, ln) \le \EShp^{\wt{\a}^n, \wt{\b}^n}(ln, ln)-s(\MinDis^{\wt{\a}^n, \wt{\b}^n}(ln, ln))^{-1/4}(\MinVar^{\wt{\a}^n, \wt{\b}^n}(ln, ln))^{3/8}l^{1/2}n^{1/2}\} \\
&\le C_plns^{-p} \quad \text{ for } s \ge s_0.  
\end{align*}
The second step is by \eqref{E92}-\eqref{E96}. The final bound is an application of Lemma \ref{LLTailMinCon} and $s_0 > 0$ denotes the absolute constant from there. Hence, by the Borel-Cantelli, $\cP$-a.s.,
\begin{align*}
\mathcal{G}_n \ge \G_n(ln, ln) \ge n\cEShp{l, l}-l^{1/2+\eta}n^{7/8+ \eta} \quad \text{ whenever } l+n \ge L_1
\end{align*}
after increasing $L_1$ if necessary. With $\eta < 1/8$, one obtains 
\begin{align}
\liminf_{n \rightarrow \infty} n^{-1}\mathcal{G}_n \ge \cEShp{l, l} \label{E100}
\end{align}
By \eqref{E95}, letting $l \rightarrow \infty$, the lower bound in \eqref{E100} matches the upper bound in \eqref{E99}. 
\end{proof}

\section{Proof of Theorem \ref{TShpFun}} 
\label{SPrShpFun}

The proof is based on Corollary \ref{CEmpShp} and an approximation of the centering $\cEShp{}$ by the shape function. 

Let $\alpha$ and $\beta$ be finite, nonnegative Borel measures on $\bbR$. Fix $x, y > 0$. The derivatives of \eqref{EStShpFun} are given by 
\begin{align*}
\partial_z^n \cStShp{x, y}{z} = x(-1)^n n!\int_\bbR \frac{\alpha(\dd a)}{(a+z)^{n+1}} + y n! \int_{\bbR} \frac{\beta(\dd b)}{(b-z)^{n+1}} \quad \text{ for } z \in \bbC \smallsetminus (-\supp \alpha \cup \supp \beta)
\end{align*}
on account of \eqref{E42}. Assume that $\alpha$ and $\beta$ are nonzero, that \eqref{EInfSuppCond} holds, and suppose  $\aml, \bml \in \bbR$ satisfy \eqref{AsmLB} and \eqref{Eab}. Then 
$
\partial_z^2 \cStShp{x, y}{z} > 0 \text{ for } z \in (-\aml, \bml). 
$
Hence, the function $z \mapsto \cStShp{x, y}{z}$ is strictly convex on $(-\aml, \bml)$. Consequently, there exists a unique $z$-value, denoted with $\cMaMin{x, y}$, in the closed interval $[-\aml, \bml]$ such that the infimum in \eqref{EShpFun} is given by  $\cShp{x, y} = \cStShp{x, y}{z}$. Examining the sign of the first derivative reveals that 
\begin{align}
\cMaMin{x, y} = \begin{cases} -\aml \quad \text{ if } \displaystyle x\int_{\bbR} \frac{\alpha(\dd a)}{(a-\aml)^{2}}  \le y \int_{\bbR} \frac{\beta(\dd b)}{(b+\aml)^{2}}  \\ \\ \bml \quad \text{ if } x \displaystyle \int_{\bbR} \frac{\alpha(\dd a)}{(a+\bml)^{2}} \ge y \int_{\bbR} \frac{\beta(\dd b)}{(b-\bml)^{2}} \\ \\ 
\text{otherwise, the unique $z$-value with } \displaystyle x\int_{\bbR} \frac{\alpha(\dd a)}{(a+z)^2} = y\int_{\bbR} \frac{\beta(\dd b)}{(b-z)^2} . \label{EMaMin}
\end{cases}
\end{align} 
 
Recall the definition of $\cEMin{m, n}$ as the unique minimizer in \eqref{EECenter}, described in \eqref{EShpMin}. 

\begin{lem}
\label{LMinConv}
Assume \eqref{AsmLB}, \eqref{AsmVagueConv}, \eqref{AsmMinConv} and that $\alpha$ and $\beta$ are not zero measures. Let $\epsilon > 0$. Then there exists $L \in \bbZ_{>0}$ such that 
\begin{align*}
|\cEMin{m, n}-\cMaMin{m, n}| < \epsilon \quad \text{ whenever } m, n \ge L. 
\end{align*}
\end{lem}
\begin{proof}
Note that $\aml, \bml < \infty$ since $\alpha, \beta \neq 0$. Pick $\delta > 0$ such that $2\delta < \aml + \bml$. Then the interval $[-\aml+\delta, \bml-\delta]$ is nonempty.    By \eqref{AsmMinConv}, there exists $L \in \bbZ_{>0}$ such that 
\begin{align*}
|\min \ca{m}{}-\aml| < \delta \quad \text{ and } \quad |\min \cb{n}{}-\bml| < \delta \quad \text{ for } m, n \ge L. 
\end{align*}
In particular, $[-\aml+\delta, \bml-\delta] \subset \cIz{m, n}$ for $m, n \ge L$. Hence, it follows from assumption \eqref{AsmVagueConv} and Lemma \ref{LCatCont} that 
\begin{align*}
\lim_{m \rightarrow \infty}\frac{1}{m}\sum_{i=1}^m \frac{1}{\ca{m}{i}+z} = \int_{\bbR} \frac{\alpha(\dd a)}{a+z} \quad \text{and} \quad  \lim_{n \rightarrow \infty} \frac{1}{n} \sum_{j=1}^n \frac{1}{\cb{n}{j}-z} = \int_{\bbR} \frac{\beta(\dd b)}{b-z}
\end{align*}
uniformly in $z \in [-\aml + \delta, \bml-\delta]$. Since the Cauchy transform produces holomorphic functions, the uniform convergence on compact sets also holds for derivatives. 
Thus, for any $k \in \bbZ_{>0}$, 
\begin{align*}
\lim_{m \rightarrow \infty}\frac{1}{m}\sum_{i=1}^m \frac{1}{(\ca{m}{i}+z)^k} = \int_{\bbR} \frac{\alpha(\dd a)}{(a+z)^k} \text{ and } \lim_{n \rightarrow \infty} \frac{1}{n} \sum_{j=1}^n \frac{1}{(\cb{n}{j}-z)^k} = \int_{\bbR} \frac{\beta(\dd b)}{(b-z)^k}
\end{align*}
uniformly in $z \in [-\aml + \delta, \bml-\delta]$. Increase $L \in \bbZ_{>0}$ if necessary to have 
\be\begin{aligned}
&\bigg|\sum_{i=1}^m \frac{1}{(\ca{m}{i}+z)^2}-m\int_\bbR \frac{\alpha(\dd a)}{(a+z)^2}\bigg| < mc\epsilon \\
 \text{and} \quad 
 &\bigg|\sum_{j=1}^n \frac{1}{(\cb{n}{j}-z)^2}-\int_{\bbR} \frac{\beta(\dd b)}{(b-z)^2}\bigg| < nc\epsilon 
\end{aligned} \label{E40}\ee
whenever $m, n \ge L$ and $z \in [-\aml + \delta, \bml - \delta]$, where $c > 0$ is a constant to be chosen below. 

Work with $m, n \ge L$ below. To prove the claim of the lemma, argue by contradiction. Consider the case $\cEMin{m, n} \ge \cMaMin{m, n}+\epsilon$ first. Put $z = \cMaMin{m, n}+\epsilon/2$. Then $z \ge -\aml + \epsilon/2$ and $z \le \cEMin{m, n}-\epsilon/2 \le \min\cb{n}{}-\epsilon/2 \le \bml+\delta-\epsilon/2$. Therefore, $z \in [-\aml + \delta, \bml-\delta]$ provided that $\delta \le \epsilon/4$. Recalling \eqref{EShpMin},    $z < \cEMin{m, n}$ and \eqref{E40} imply that 
\be\begin{aligned}
0 &\ge -\sum_{i=1}^m \frac{1}{(\ca{m}{i}+z)^2}+\sum_{j=1}^n \frac{1}{(\cb{n}{j}-z)^2} \\&\ge  -m\int_\bbR \frac{\alpha(\dd a)}{(a+z)^2}+n\int_\bbR \frac{\beta(\dd b)}{(b-z)^2} - c\epsilon(m+n). 
\end{aligned}  \label{E41}\ee
The function $f:[\cMaMin{m, n}, z] \rightarrow \bbR$ given by 
\begin{align*}
f(s) = \partial_s \cStShp{m, n}{s} = -m\int_{\bbR} \frac{\alpha(\dd a)}{(a+s)^{2}} + n \int_{\bbR} \frac{\beta(\dd b)}{(b-s)^{2}}
\end{align*}
is differentiable in the interior and continuous up to the boundary of its domain. Continuity at the endpoint $\cMaMin{m, n}$ holds even when $\cMaMin{m, n} = -\aml$ because then $\int_{\bbR} (a-\aml)^{-2} \alpha(\dd a) < \infty$, see \eqref{EMaMin}.   By the mean-value theorem, there exists $s \in (\cMaMin{m, n}, z)$ such that
\begin{align*}
f(z) &= f(\cMaMin{m, n}) + \frac{1}{2}\epsilon f'(s) \; \ge \; m \epsilon \int_{\bbR} \frac{\alpha(\dd a)}{(a+s)^{3}}+n \epsilon \int_{\bbR} \frac{\beta(\dd b)}{(b-s)^{3}} \\
&\ge m \epsilon \int_{\bbR} \frac{\alpha(\dd a)}{(a+\bml)^{3}}+n \epsilon \int_{\bbR} \frac{\beta(\dd b)}{(b+\aml)^{3}} \; \ge\; 2c\epsilon (m+n). 
\end{align*} 
The first inequality above holds because $f(\cMaMin{m, n}) = 0$ if $\cMaMin{m, n} \in (-\aml, \bml)$, and $f(\cMaMin{m, n}) \ge 0$ if $\cMaMin{m, n} = -\aml$. Note that $\cMaMin{m, n} < \bml$ in the present case. 
The second inequality is monotonicity, and the last inequality comes from choosing $c>0$ small enough. 
Combining this with  \eqref{E41} gives  the contradiction 
\begin{align*}
0 &\ge f(z)-c\epsilon(m+n) \ge c\epsilon (m+n). 
\end{align*}

Likewise, the case $\cEMin{m, n} \le \cMaMin{m, n}-\epsilon$ results in a contradiction. The proof of Lemma \ref{LMinConv} is complete.  
\end{proof}

\begin{lem}
\label{LEShpConv}
Assume \eqref{AsmLB}, \eqref{AsmVagueConv} and \eqref{AsmMinConv}. Let $\epsilon > 0$. Then there exists $L \in \bbZ_{>0}$ such that 
\begin{align*}
|\cEShp{m, n}-\cShp{m, n}| \le \epsilon(m+n) \quad \text{ whenever } m, n \in \bbZ_{\ge L}. 
\end{align*}
\end{lem}
\begin{proof}
Abbreviate $\EMin = \cEMin{m, n}$ and $\MaMin = \cMaMin{m, n}$. Assume further that $\aml, \bml < \infty$. As in the preceding proof, pick $\delta > 0$ and $L \in \bbZ_{>0}$ such that $2\delta < \aml + \bml$ and $[-\aml+\delta, \bml-\delta] \subset \cIz{m, n}$ for $m, n \ge L$. Let $z \in (-\aml + \delta, \bml-\delta)$. By \eqref{AsmMinConv}, after increasing $L$ if necessary, 
\be\label{E365} 
\ca{m}{i}+z \ge \delta\quad\text{and}\quad \cb{n}{j}-z \ge \delta \quad\text{for }m, n \ge L \text{ and }i \in [m], j \in [n].
\ee 
Hence, using definition \eqref{EMinVar}, identity \eqref{E35} and bound \eqref{EVarUB}, one obtains that 
\begin{align}
0 &\le \cEStShp{m, n}{z}-\cEShp{m, n} \label{E366}\\
&= (z-\EMin)^2 \bigg\{\sum_{i=1}^m \frac{1}{(\ca{m}{i}+z)(\ca{m}{i}+\EMin)^2} + \sum_{j=1}^n \frac{1}{(\cb{n}{j}-z)(\cb{n}{j}-\EMin)^2}\bigg\}\nonumber\\
&\le 2\delta^{-1}(z-\EMin)^2\cMinVar{m, n} \nonumber\\
&\le \frac{8(z-\EMin)^2(m+n)}{\delta(\inf \bfa + \inf \bfb)^2}. \label{E36}
\end{align}
The denominator is nonzero by \eqref{AsmLB} and \eqref{AsmMinConv}. Analogously  one  derives 
\begin{align}
0 &\le \cStShp{m, n}{z}-\cShp{m, n} \label{E377} \\
&= (z-\MaMin)^2 \bigg\{m\int_{\bbR} \frac{\alpha(\dd a)}{(a+z)(a+\MaMin)^2} + n \int_{\bbR} \frac{\beta(\dd b)}{(b-z)(b-\MaMin)^2}\bigg\} \nonumber\\
&\le 2\delta^{-1}(z-\MaMin)^2 \bigg\{m\int_{\bbR} \frac{\alpha(\dd a)}{(a+\MaMin)^2} + n \int_{\bbR} \frac{\beta(\dd b)}{(b-\MaMin)^2}\bigg\} \nonumber\\
&\le \frac{8(z-\MaMin)^2(m+n)}{\delta(\inf \bfa + \inf \bfb)^2}. \label{E37}
\end{align}
The second inequality above uses $(\inf\supp\alpha)+z\ge\delta$ and $(\inf\supp\beta)-z\ge\delta$,  which comes from combining \eqref{E365} with \eqref{Eab},   \eqref{AsmVagueConv} and \eqref{AsmMinConv}.
For the last inequality, observe from \eqref{EMaMin} that 
\begin{align*}
\int_{\bbR} \frac{\alpha(\dd a)}{(a+\MaMin)^2} + \int_{\bbR} \frac{\beta(\dd b)}{(b-\MaMin)^2} 
&\le 2\bigg(\one{\{\MaMin \le \tfrac{\bml-\aml}{2}\}}
\int_{\bbR} \frac{\beta(\dd b)}{(b-\MaMin)^2} + \one{\{\MaMin > \tfrac{\bml-\aml}{2}\}}\int_{\bbR} \frac{\alpha(\dd a)}{(a+\MaMin)^2}\bigg) \le \frac{8}{(\aml+\bml)^2}. 
\end{align*}

To verify the claim of the lemma, first consider the case that $\alpha$ and $\beta$ are not zero measures. Choose $\delta$ sufficiently small such that $40\delta(\inf \bfa + \inf \bfb)^{-2} < \epsilon/2$. Then pick $L \in \bbZ_{>0}$ sufficiently large such that 
\begin{align}
|\EMin-\MaMin| \le \delta \quad \text{ and } |\cEStShp{m, n}{z}-\cStShp{m, n}{z}| \le \frac{1}{2}\epsilon(m+n) \label{E38}
\end{align}
whenever $m, n \ge L$ and $z \in [-\aml +\delta, \bml - \delta]$. Such $L$ exists by virtue of Lemmas \ref{LMinConv} and \ref{LCatCont}. 

Set $z = \min\{\max\{\MaMin, -\aml + \delta\}, \bml-\delta\}$. Then $z \in [-\aml + \delta, \bml-\delta]$ and $|z-\MaMin| \le \delta$. Furthermore, 
$|z-\EMin| \le |z-\MaMin|+ |\MaMin-\EMin| \le 2\delta$ if $m, n \ge L$. Now putting together \eqref{E366}--\eqref{E38} via the triangle inequality leads to 
\begin{align*}
|\cEShp{m, n}-\cShp{m, n}| &\le |\cEStShp{m, n}{z}-\cEShp{m, n}| + |\cStShp{m, n}{z}-\cShp{m, n}| \\
&+ |\cEStShp{m, n}{z}-\cStShp{m, n}{z}| \\
&\le \frac{40\delta(m+n)}{(\inf \bfa + \inf \bfb)^2}+\frac{\epsilon}{2}(m+n) \le \epsilon (m+n)
\end{align*}
whenever $m, n \ge L$.

Now assume that $\beta = 0$. Then, as noted in \eqref{E46},  
\begin{align*}
\cShp{x,  y} = x \int_{\bbR} \frac{\alpha(\dd a)}{a+\bml} \quad \text{ for } x, y > 0.
\end{align*}
By another appeal to Lemma \ref{LCatCont}, increase $L$ if necessary to have 
\begin{align*}
\bigg|\sum_{i=1}^m \frac{1}{\ca{m}{i}+\bml+\delta}-m\int_{\bbR} \frac{\alpha(\dd a)}{a+\bml+\delta}\bigg| \le \frac{\epsilon m}{2} \quad \text{ for } m \ge L. 
\end{align*}
  Then 
\begin{align*}
\cEShp{m, n} &\ge \sum_{i=1}^m \frac{1}{\ca{m}{i}+\min \cb{n}{}} \ge \sum_{i=1}^m \frac{1}{\ca{m}{i}+\bml + \delta} \ge m \int_{\bbR}\frac{\alpha(\dd a)}{a+\bml+\delta} -\frac{\epsilon m}{2} \\
&= m \int_{\bbR} \frac{\alpha(\dd a)}{a+\bml} -m \delta \int_{\bbR} \frac{\alpha(\dd a)}{(a+\bml)(a+\bml+\delta)}-\frac{\epsilon m}{2}\\
&\ge m \int_{\bbR} \frac{\alpha(\dd a)}{a+\bml} -\frac{m \delta}{(\inf \a + \inf \b)^2}-\frac{\epsilon m}{2}\\
&\ge m \int_{\bbR} \frac{\alpha(\dd a)}{a+\bml} - \epsilon m = \cShp{m, n}-\epsilon m
\end{align*}
for $m, n \ge L$. The second-last inequality comes from $\inf \a\le \inf\supp\alpha$ and  $\inf \b\le \min \cb{n}{}\to\bml$.

 For the complementary bound, choose $L$ larger if necessary such that 
\begin{align}
\min\cb{n}{} \in  [\bml-\delta/2, \bml+\delta/2] \quad \text{ and } \quad \sum_{j=1}^n \frac{1}{\cb{n}{j}-\bml+\delta/2} \le \frac{\epsilon n}{2} \quad \text{ for } n \ge L \label{E500} \\
\bigg|\sum_{i=1}^m \frac{1}{\ca{m}{i}+\bml-\delta/2}-m \int_{\bbR} \frac{\alpha(\dd a)}{a+\bml-\delta/2}\bigg| \le \frac{\epsilon m}{2} \quad \text{ for } m \ge L \label{E501}, 
\end{align}
where we invoke Lemma \ref{LCatCont} again. Then, for $m, n \ge L$, one obtains that 
\begin{align*}
\cEShp{m, n} &\le \sum_{i=1}^m \frac{1}{\ca{m}{i}+\min \cb{n}{}-\delta} + \sum_{j=1}^n \frac{1}{\cb{n}{j}-\min \cb{n}{}+\delta}  \\
&\le \sum_{i=1}^m \frac{1}{\ca{m}{i}+\bml-3\delta/2}+ \sum_{j=1}^n \frac{1}{\cb{n}{j}-\bml+\delta/2} \\
&\le m \int_{\bbR} \frac{\alpha(\dd a)}{a+\bml-3\delta/2} + \frac{\epsilon}{2}(m+n) \\
&= \cShp{m, n} + \frac{3\delta m}{2} \int_{\bbR} \frac{\alpha (\dd a)}{(a+\bml)(a+\bml-3\delta/2)} + \frac{\epsilon}{2}(m+n) \\
&\le \cShp{m, n} + \frac{6\delta m}{(\inf \bfa + \inf \bfb)^2} + \frac{\epsilon}{2}(m+n) \le \cShp{m, n} + \epsilon (m+n).
\end{align*}
The first inequality above is by definition \eqref{EECenter}. The next two inequalities are due to \eqref{E500}-\eqref{E501}. For the second-last inequality, first use $a \ge \inf \supp \alpha \ge \aml$ and $\delta \le \aml + \bml$, and then recall $\aml + \bml \ge \inf \a + \inf \b$. 
This completes the case $\beta = 0$. The case $\alpha = 0$ is handled similarly. 

The preceding paragraph also includes the cases $\aml = \infty, \bml < \infty$ and $\aml < \infty, \bml = \infty$. Therefore, the only remaining case is $\aml = \bml = \infty$ which implies that $\cShp{}=0$. Pick any $z_0 \in (-\inf \a, \inf \b)$. Since $\alpha, \beta = 0$ now, by Lemma \ref{LCatCont}, 
\begin{align*}
\cEShp{m, n} \le \cEStShp{m, n}{z_0} \le \epsilon (m+n) \quad \text{ for } m,n \ge L 
\end{align*}
after increasing $L$ if necessary. This completes the proof. 
\end{proof}

\begin{proof}[Proof of Theorem \ref{TShpFun}]
This is immediate from Theorem \ref{TEmpShp} and Lemma \ref{LEShpConv}. 
\end{proof}

\section{Proofs of Theorem \ref{TNarShp}}
\label{SThinRec}

\begin{proof}[Proof of Theorem \ref{TNarShp}]
Due to the symmetry, only (a) is proved below. Write $p = \inf_{n \in S} \min \cb{n}{}$. Introduce a small parameter $\eta > 0$ such that $2\eta < \inf \a + p$. Since $|\cIz{m, n}| = \min \ca{m}{} + \min \cb{n}{} \ge \inf \a + p > \eta$ for $m \in \bbZ_{>0}$ and $n \in S$, it follows from Theorem \ref{TEmpShp} that, $\cP$-a.s., there exists a random $L \in \bbZ_{>0}$ such that 
\begin{align}
|\cGp{m, n}-\cEShp{m, n}| \le \epsilon (m+n) \quad \text{ for } m \in \bbZ_{>0} \text{ and } n \in S \text{ with } m+n \ge L. \label{E202}
\end{align}

By the first limit in \eqref{AsmVagueConv} and Lemma \ref{LCatCont}, there exists $K \in \bbZ_{>0}$ such that 
\begin{align}
\bigg|\sum_{i=1}^m \frac{1}{\ca{m}{i}+z}-m\int_{\bbR} \frac{\alpha(\dd a)}{a+z}\bigg| \le \epsilon m\quad \text{ whenever } m \ge K \text{ and } z \in [p-\eta, p+4\epsilon^{-1}]. \label{E43}
\end{align}
The range of $z$ in \eqref{E43} can be extended to $[p-\eta, \infty)$ since the terms $(\ca{m}{i}+z)^{-1}$ for $i \in [m]$ and the integrand $(a+z)^{-1}$ are bounded from above by $(\inf \a+z)^{-1}$ which is less than $\epsilon/4$ for $z \ge p+4/\epsilon$. 

Choose $K \ge L$, take $m \ge K$ and $n \in S$ below. Write $\EMin = \cEMin{m, n}$. By definition \eqref{EECenter}, \eqref{E43} and since $\max \{\EMin, p\} \le \min \cb{n}{}$, 
one has the lower bound
\begin{align}
\cEShp{m, n} \ge \sum_{i=1}^m \frac{1}{\ca{m}{i}+\EMin} \ge \sum_{i=1}^m \frac{1}{\ca{m}{i}+\min \cb{n}{}} \ge m \int_{\bbR} \frac{\alpha(\dd a)}{a+\min \cb{n}{}}-\epsilon m. \label{E200}
\end{align}
For a complementary upper bound, noting that $\min \cb{n}{}-\eta \in \cIz{m, n}$, develop 
\begin{align*}
\cEShp{m, n} &\le \cEStShp{m, n}{\min\cb{n}{}-\eta} = \sum_{i=1}^m \frac{1}{\ca{m}{i}+\min \cb{n}{}-\eta} + \sum_{j=1}^n \frac{1}{\cb{n}{j}-\min \cb{n}{}+\eta} \\
&\le m\int_{\bbR} \frac{\alpha(\dd a)}{a+\min\cb{n}{}-\eta} + \epsilon m + \frac{n}{\eta} \\
&= m\int_{\bbR} \frac{\alpha(\dd a)}{a+\min\cb{n}{}} + m\eta \int_{\bbR} \frac{\alpha(\dd a)}{(a+\min\cb{n}{})(a+\min\cb{n}{}-\eta)} + \epsilon m + \frac{n}{\eta} \\
&\le m\int_{\bbR} \frac{\alpha(\dd a)}{a+\min\cb{n}{}} +2m\eta \int_{\bbR} \frac{\alpha(\dd a)}{(a+\min\cb{n}{})^2}+\epsilon m+\frac{n}{\eta}.  
\end{align*}
The first two inequalities above come again from \eqref{EECenter} and \eqref{E43}. The last inequality comes from  $2 \eta < \inf \a + \min \cb{n}{} \le \inf \supp \alpha + \min \cb{n}{}$ where we appealed to  the assumption that the first limit in \eqref{AsmVagueConv} holds. With $\eta \le \epsilon$ and $n \le \epsilon \eta m$, one then has 
\begin{align}\cEShp{m, n} \le m \int_{\bbR} \frac{\alpha(\dd a)}{a+\min\cb{n}{}} + C\epsilon m, \label{E201}\end{align} 
where $C = 2 + 2 \int_{\bbR} (a+\min\cb{n}{})^{-2}\alpha(\dd a)$. 

Choose $\delta = \min \{1, \epsilon \eta\}$. Combining the bounds from \eqref{E202}, \eqref{E200} and \eqref{E201} with the triangle inequality results in 
\begin{align*}
\bigg|\cGp{m, n}-m \int_{\bbR} \frac{\alpha(\dd a)}{a+\min\cb{n}{}}\bigg|  \le \epsilon (m+n) + C\epsilon m \le \epsilon m (1+\delta + C) \le \epsilon (2+C) m
\end{align*}
for $m \in \bbZ_{\ge K}$ and $n \in S$ with $n \le \delta m$. The result follows upon replacing $\epsilon$ with $\epsilon (2+C)^{-1}$ throughout. 
\end{proof}

\section{Proof of Theorem \ref{TLimShp}}
\label{SPrLimShp}

\begin{lem}
\label{LLimShpBd}
Assume \eqref{EInfSuppCond}-\eqref{Eab}. Then 
\begin{enumerate}[\normalfont (a)]
\item $\cLimShp \subset \{(x, y) \in \bbR_{>0}^2: x\cA{\bml} \le 1 \text{ and } y \cB{-\aml} \le 1\}$ 
\item $\cLimShp \supset \{(x, y) \in \bbR_{>0}^2: 2(x + y) \le \aml + \bml\}$. 
\end{enumerate}
\end{lem}
\begin{proof}
\eqref{EShpBdLB} implies (a). If $\aml, \bml < \infty$ then taking $z = \dfrac{1}{2}(\bml-\aml)$ in \eqref{EShpFun} gives 
\begin{align*}
\cShp{x, y} \le \frac{2(x+y)}{\aml + \bml} \quad \text{ for } x, y \in \bbR_{>0}. 
\end{align*}
This bound also trivially holds if $\aml = \infty$ or $\bml = \infty$ by \eqref{EShpFunZero}. Hence, (b). 
\end{proof}

\begin{lem}
\label{LClusBd}
Assume \eqref{AsmLB}, \eqref{AsmVagueConv} and \eqref{AsmMinConv}. Let $S$ denote the set from \eqref{ETrU}, and $\ams$ and $\bms$ be given by \eqref{EAB}. Then there exists a constant $C_0 > 0$ such that the following hold .
\begin{enumerate}[\normalfont (a)]
\item $S \cap (\cLimShp \cup \cLimShpBd{\ams, \bms}) \subset [0, C_0]^2$. 
\item $\cP$-a.s., $S \cap t^{-1}\cClus{t} \subset (0, \max\{C_0, Kt^{-1}\}]^2$ for $t > 0$ and some random $K \in \bbZ_{>0}$.  
\end{enumerate}
\end{lem}
\begin{proof}
Let $t > 0$ and $(x, y) \in S \cap (\cLimShp \cup \cLimShpBd{\ams, \bms} \cup t^{-1}\cClus{t})$. By symmetry, it suffices to bound $x$ from above assuming $x \ge y$. If $\alpha = 0$ or $\bml = \infty$ then $x \le C$ by definition \eqref{ETrU}. Assume now that $\alpha \neq 0$ and $\bml < \infty$. The latter implies that $\bms < \infty$ in view of \eqref{AsmMinConv}. This and monotonicity give $0 < \cA{\bms} \le \cA{\bml} $. If $(x, y) \in \cLimShp$ then, by Lemma \ref{LLimShpBd}(a),  $x \le \cA{\bml}^{-1} < \infty$. If $(x, y) \in \cLimShpBd{\ams, \bms}$ then $x \le \cA{\bms}^{-1} < \infty$ by definition \eqref{ELimShpBd}. Hence, (a). 

Turn to the remaining case $(x, y) \in t^{-1}\cClus{t}$. Let $\eta > 0$ to be chosen below. By Theorems \ref{TShpFun} and \ref{TNarShp}, $\cP$-a.s., there exist random $K, L \in \bbZ_{>0}$ with $K \ge L$ such that 
\begin{alignat}{2}
\cGp{m, n} &\ge \cShp{m, n} + \eta (m+n) \quad &&\text{ for } m, n \in \bbZ_{> L} \label{E400} \\
\cGp{m, n} &\ge m\cA{\min \cb{n}{}} + \eta m \quad &&\text{ for } m \in \bbZ_{\ge K} \text{ and } n \in [L] \label{E401}.
\end{alignat}
To prove (b), one may assume that $tx > K$. If $ty > L$ then, by monotonicity, homogeneity, \eqref{E400} and since $x \ge y$, 
\begin{align}
\cShp{x, y} &= t^{-1}\cShp{tx, ty} \le t^{-1} \cShp{\lc tx \rc, \lc ty \rc} \le t^{-1} \{\cGp{\lc tx \rc, \lc ty \rc} + \eta(\lc tx \rc + \lc ty \rc)\} \nonumber \\
&\le 1 + \eta (1+1/L)(x+y) \le 1 + 4\eta x. \label{E403}
\end{align}
Combining the last bound with \eqref{EShpBdLB} and using monotonicity yield 
\begin{align*}
x\cA{\bms} \le x \cA{\bml} \le 1 + 4\eta x. 
\end{align*}
If $ty \le L$ then, by monotonicity and \eqref{E401},  
\begin{align}
x\cA{\bms} \le t^{-1}\lc tx \rc \cA{\min \cb{\lc ty \rc}{}}\le t^{-1}\cGp{\lc tx \rc, \lc ty \rc}+\eta t^{-1}\lc tx \rc \le 1 + 2\eta x. \label{E310}
\end{align}
Since $\cA{\bms} > 0$, one then has $x \le 2\cA{\bms}^{-1}$ in both cases upon taking $\eta$ sufficiently small. Hence, (b). 
\end{proof}

\begin{lem}
\label{LConBlk}
Assume \eqref{AsmLB}, \eqref{AsmVagueConv} and \eqref{AsmMinConv}. Let $S$ be given by \eqref{ETrU}, and define $S_t = S \cap \bbR_{\ge t}^2$ for $t > 0$. Let $0 < \epsilon < 1$. Then, $\cP$-a.s., there exists a random $L \in \bbZ_{>0}$ such that 
\begin{align*}
S_{L/t} \cap (1-\epsilon) \cLimShp \subset S_{L/t} \cap t^{-1}\cClus{t} \subset S_{L/t} \cap (1+\epsilon) \cLimShp \quad \text{ for } t \in \bbR_{>0}. 
\end{align*}
\end{lem}
\begin{proof}
Introduce $\delta > 0$ to be tuned later. By Theorem \ref{TShpFun}, $\cP$-a.s., there exists a random $L \in \bbZ_{>0}$ such that 
\begin{align}
|\cGp{m, n}-\cShp{m, n}| \le \delta (m+n) \quad \text{ for } m, n \in \bbZ_{\ge L}. \label{E300}
\end{align}

Fix $t \in \bbR_{>0}$. Assume that $(x, y) \in S_{L/t} \cap (1-\epsilon) \cLimShp$. Choosing $L$ large enough, one has $x, y \le C_0$ where $C_0 > 0$ is the constant from Lemma \ref{LClusBd}. By \eqref{E300}, monotonicity, homogeneity and that $L/t \le x, y \le C_0$,  
\begin{align*}
\cGp{\lc tx \rc, \lc ty \rc} &\le \cShp{\lc tx \rc, \lc ty \rc}+\delta (\lc tx \rc + \lc ty \rc) \le t(1+1/L)(\cShp{x, y}+\delta(x+y)) \\
&\le t(1+1/L)(1-\epsilon+2\delta C_0) \le t.
\end{align*}
The last line inequality holds provided that $\delta$ is sufficiently small and $L$ is sufficiently large. Hence, $(x, y) \in S_{L/t} \cap t^{-1}\cClus{t}$. Assuming this and arguing as in \eqref{E403} also lead to 
\begin{align*}
\cShp{x, y} 
&\le 1 + \delta (1+1/L)(x+y) \le 1 + 4\delta C_0 \le 1+\epsilon 
\end{align*}
for sufficiently small $\delta$. Hence, $(x, y) \in S_{L/t} \cap (1+\epsilon)\cLimShp$. 
\end{proof}

\begin{lem}
\label{LConBd}
Assume \eqref{AsmLB}, \eqref{AsmVagueConv} and \eqref{AsmMinConv}. Let $S$, $\ams$ and $\bms$ be given by \eqref{ETrU} and \eqref{EAB}. Let $0 < \epsilon < 1$. Then, $\cP$-a.s., there exist deterministic $M, N \in \bbZ_{>0}$ and random $K \in \bbZ_{>0}$, $T \in \bbR_{>0}$ such that the following hold for $(x, y) \in \bbR_{>0}^2$ and $t \in \bbR_{\ge T}$. 
\begin{enumerate}[\normalfont (a)]
\item If $(x, y) \in S \cap t^{-1}\cClus{t}$ and $ty \le K$ then $x\cA{\bms} \le 1+\epsilon$. If $x \cA{\bms} \le 1-\epsilon$ and $x\one_{\{\alpha = 0\} \cup \{\bml = \infty\}} \le C$ then $(x, Nt^{-1}) \in S \cap t^{-1} \cClus{t}$. 
\item If $(x, y) \in S \cap t^{-1}\cClus{t}$ and $tx \le K$ then $y\cB{\mathfrak{-A}} \le 1+\epsilon$. If $y \cA{\bms} \le 1-\epsilon$ and $y\one_{\{\beta = 0\} \cup \{\aml = \infty\}}$ then $(Mt^{-1}, y) \in S \cap t^{-1} \cClus{t}$. 
\end{enumerate}
\end{lem}
\begin{proof}
Due to the symmetry, we only prove (a). Let $\eta > 0$ to be chosen below. By Theorem \ref{TNarShp} and Corollary \ref{CThinShp}, $\cP$-a.s., there exist a deterministic $0 < \delta \le 1$ and a random $K \in \bbZ_{>0}$ such that 
\begin{align}
|\cGp{m, n}-m\cA{\min \cb{n}{}}| &\le \eta m &&\quad \text{ for } m \in \bbZ_{\ge K}, n \in \bbZ_{>0} \text{ with } n \le \delta m \label{E301}\\
|\cGp{m, n}-n\cB{-\min \ca{m}{}}| &\le \eta n &&\quad \text{ for } n \in \bbZ_{\ge K}, m \in \bbZ_{>0} \text{ with } m \le \delta n. \nonumber 
\end{align}

Fix $t \in \bbR_{\ge T}$. Assume that $(x, y) \in S \cap t^{-1}\cClus{t}$ and $ty \le K$. Let $C_0 > 0$ denote the constant from Lemma \ref{LClusBd}. If $\delta tx \ge K$ then, by \eqref{E301} and arguing as in \eqref{E310}, 
\begin{align*}
x\cA{\bms} \le 1 + 2\eta x \le 1 + 2\eta C_0 \le 1 + \epsilon 
\end{align*}
for sufficiently large $K$ and sufficiently small $\eta$.  

To prove the second claim in (a), assume now that $x \cA{\bms} \le 1-\epsilon$ and $x\one_{\{\alpha = 0\} \cup \{\bml = \infty\}} \le C$. The assumptions again imply that $x \le C_0$ by Lemma \ref{LClusBd}.  By continuity, there exists $N \in \bbZ_{>0}$ such that 
$\cA{\min \cb{N}{}} \le \cA{\bms} + \eta$. Work with $K \ge N$. If $\delta tx \ge K$ then  
\begin{align*}
\cGp{\lc tx \rc, N} &\le \lc tx \rc \cA{\min \cb{N}{}} + \eta \lc tx \rc \le (1+1/K)tx (\cA{\min \cb{N}{}} + \eta) \\ 
&\le (1+1/K) tx (\cA{\bms} + 2\eta) \le (1+1/K) t (1-\epsilon + 2\eta C_0). 
\end{align*}
After choosing $\eta$ sufficiently small and $K$ sufficiently large, the last term is at most $t$. If $\delta t x < K$ then $x \le K\delta^{-1}t^{-1} \le K\delta^{-1}T^{-1} \le T \le t$ upon taking $T$ sufficiently large. 
\end{proof}

\begin{proof}[Proof of Theorem \ref{TLimShp}]
Corresponding to the given $\epsilon > 0$, $\cP$-a.s., there exist $L \in \bbZ_{>0}$ as in Lemma \ref{LConBlk} and $K, M, N \in \bbZ_{>0}, T \in \bbR_{>0}$ as in Lemma \ref{LConBd}. Take $t \ge T$ below. Let $C_0 > 0$ denote the constant from Lemma \ref{LClusBd}. To prove (a), it suffices to show that 
\begin{align}
\cHDis{S \cap \overline{t^{-1}\cClus{t}}}{S \cap (\cLimShp \cup \cLimShpBd{\ams, \bms})} \le \epsilon C_0 \label{EHDisBnd}
\end{align}
for sufficiently large $T$. The arguments in \eqref{EHDisBnd} are closed subsets of $\bbR_{\ge 0}^2$. The second argument is nonempty and bounded by Lemmas \ref{LLimShpBd}(b) and \ref{LClusBd}(a). The same is true for the first argument with large enough $T$ by Lemma \ref{LClusBd}(b) and the first containment in Lemma \ref{LConBlk}. Hence, the left-hand side in \eqref{EHDisBnd} makes sense per definition of the Hausdorff metric in Subsection \ref{SsHaus}. 

To obtain \eqref{EHDisBnd}, first pick $(x, y) \in S \cap t^{-1}\cClus{t}$. Put 
\begin{align*}x' = (1+\epsilon)^{-1}x \one_{\{tx \ge K\}} \quad \text{ and } \quad y' = (1+\epsilon)^{-1}y \one_{\{ty \ge K\}}.\end{align*} 
If $tx, ty \ge K$ then Lemma \ref{LConBlk} implies that $(x', y') \in S \cap \cLimShp$. If $tx < K$ or $ty < K$ then Lemma \ref{LConBd} yields $(x', y') \in S \cap \cLimShpBd{\ams, \bms}$. Furthermore, 
$0 \le x-x', y-y' \le \epsilon C_0$. Assume now that $(x, y) \in S \cap (\cLimShp \cup \cLimShpBd{\ams, \bms})$. Put 
\begin{align*}
x' &= (1-\epsilon)x\one_{\{tx \ge K\}} + \frac{M}{t} \one_{\{tx < K\}} \quad \text{ and } \quad y' = (1-\epsilon)y\one_{\{ty \ge K\}} + \frac{N}{t}\one_{\{ty < K\}}. 
\end{align*} 
Then $(x', y') \in S \cap t^{-1}\cClus{t}$ by Lemma \ref{LConBlk} if $tx, ty \ge K$ with $K \ge (1-\epsilon)L$ and by Lemma \ref{LConBd} otherwise. Furthermore, $|x-x'|, |y-y'| \le \epsilon C_0$ by taking $T$ large enough. Hence, \eqref{EHDisBnd} is proved. 
\end{proof}

\section{Proofs of Theorems \ref{TLimFlux} and \ref{TNarHeight}}
\label{SPrFlux}

\begin{proof}[Proof of Theorem \ref{TLimFlux}]
We prove only the estimate for the flux process. The estimate for the height process can be obtained in the same manner, and then the estimate for the particle locations is immediate. 

Introduce a constant $c > 0$ to be chosen later. By Theorem \ref{TShpFun}, $\cP$-a.s., there exists $L_0 \in \bbZ_{>0}$ such that 
\begin{align}
|\cGp{m, n}-\cShp{m, n}| \le c\epsilon(m+n) \quad \text{ for } m, n \in \bbZ_{\ge L_0}. \label{E80}
\end{align}

Choose $L \in \bbZ_{>0}$ with $L > 2L_0 \max \{1, \epsilon^{-1}\}$. Fix $m \ge L$ and $t \in \bbR_{\ge 0}$ below. Consider $j \in \bbZ_{>0}$ with $j > \clFlux{m}{t}+\epsilon (m+t)$. Then, by definition of the flux process and rearranging terms, 
\begin{align}
(m+j)\cA{z} + j \cB{z} > t + \epsilon (\cA{z}+\cB{z})(t+j) \quad \text{ for } z \in (-\aml, \bml). \label{E81}
\end{align}
In particular, $t < c_0 (m+j)$ for some constant $c_0 > 0$. Now setting $z = \cMaMin{m+j, j}$ in \eqref{E81}, using monotonicity of $\cA{}$ and $\cB{}$ and the bound on $t$, one obtains that 
\begin{align*}
\cShp{m+j, j} = (m+j) \cA{\EMin} + j \cB{\EMin} &> t + \epsilon (\cA{\bml}+\cB{-\aml})(c_0(m+j) + j) \\
&> t + c_1 \epsilon (m+2j)
\end{align*}
for some constant $c_1 > 0$. Since $j \ge L_0$, choosing $c \le c_1$ yields 
\begin{align*}\cGp{m+j, j} \ge \cShp{m+j, j}-c_1 \epsilon (m+2j) >  t\end{align*}
in view of \eqref{E80}. Then it follows from the choice of $m$ that 
\begin{align*}
\cFluxp{m}{t} \le \clFlux{m}{t}+\epsilon (m+t). 
\end{align*}

For the lower bound, it suffices to consider the case $\clFlux{m}{t} \ge \epsilon (m+t)$ since the flux process is nonnegative. Now consider $j \in \bbZ_{>0}$ with $L_0 \le j < \clFlux{m}{t}-\epsilon (m+t)/2$. Such $j$ exists since the right-hand side is greater than $L_0$. By definition \eqref{ELimHeight}, there exists $z \in (-\aml, \bml)$ such that 
\begin{align}
(m+j)\cA{z} + j \cB{z} < t - \frac{\epsilon}{2} (\cA{z}+\cB{z})(t+j). \label{E83}
\end{align}
In particular, $t \ge j (\cA{z}+\cB{z}) \ge j (\cA{\bml}+\cB{-\aml})$. Then, by \eqref{E83}, 
\begin{align*}
\cShp{m+j, j} < t - c_2 \epsilon (m+2j) 
\end{align*}
for some constant $c_2 > 0$. Now choosing $c \le c_2$ yields 
\begin{align*}
\cGp{m+j, j} \le \cShp{m+j, j}+c_2 \epsilon (m+2j) < t. 
\end{align*}
One then concludes from the choice of $m$ that 
\begin{equation*}
\cFluxp{m}{t} \ge \clFlux{m}{t}-\frac{\epsilon}{2} (m+t). \qedhere
\end{equation*}
\end{proof}

\begin{proof}[Proof of Theorem \ref{TNarHeight}]
Abbreviate $A = \cA{\min \cb{n}{}}$. Assume that $\epsilon \le 1$ and let $0 < c < \min \{A, A^2\}$. By Theorem \ref{TNarShp}, there exists $K \in \bbZ_{>0}$ such that 
\begin{align}
|\cGp{m, n}-m A| \le c \epsilon m \quad \text{ for } m \ge \bbZ_{\ge K}. \label{E84}
\end{align}
Pick $T \ge K/\epsilon$ and let $t \ge T$ below. 

If $m \in \bbZ_{>0}$ with $m > t (1/A + \epsilon)$ then 
$m \ge K$ and, by \eqref{E84}, 
\begin{align*}
\cGp{m, n} \ge m (A-c\epsilon) > t (\epsilon + 1/A)(A-c\epsilon) = t + \epsilon t(A - c/A - c\epsilon) > t
\end{align*}
Hence, $\cHp{n}{t} \le t(1/A + \epsilon)$. Now assuming $2\epsilon < 1/A$, let $m \in \bbZ_{\ge K}$ with $m \le t(1/A-\epsilon)$ noting that the right-hand side is at least $K$. Therefore,  by \eqref{E84}, 
\begin{align*}
\cGp{m, n} \le m(A+c\epsilon) < t (1/A-\epsilon) (A+c\epsilon) = t - \epsilon t (A - c/A - c\epsilon) < t. 
\end{align*}
Hence, $\cHp{n}{t} \ge t(1/A-\epsilon)$. 
\end{proof}

\appendix 

\section{}
\label{SApdx}

\subsection{Concentration bounds for sums of independent exponential random variables}

\begin{lem}
\label{LComBnd}
Let $m, p \in \bbZ_{>0}$ and $x_1, \dotsc, x_m > 0$. Then 
\begin{align*}
\sum_{\substack{(k_1, \dotsc, k_m) \,\in\, \bbZ_{\ge 0}^m \\ \sum_{i \in [m]} k_i = 2p \\ k_i \neq 1\,\forall  i \in [m]}} \; \prod_{i \in [m]} x_i^{k_i} \le 3^p \bigg(\sum_{i \in [m]} x_i^2\bigg)^p. 
\end{align*}
\end{lem}
\begin{proof}
For $n \in \bbZ_{> 0}$, write $S_n$ for the set of all $\bfk = (k_1, \dotsc, k_m) \in \bbZ_{\ge 0}^m$ such that $\sum_{i \in [m]} k_i = n$. Let $S_n'$ denote the set of all $\bfk \in S_n$ with $k_i \neq 1$ for $i \in [m]$. Define two maps $f, g: S_{2p}' \rightarrow S_p$ as follows: For each $\bfk \in S_{2p}'$, there is an even number $2l$ of indices $i \in [m]$ for which $k_i$ is odd. Let $i_1 < \dotsm < i_{2l}$ denote these indices. Define $f(\bfk) \in S_{p}$ and $g(\bfk) \in S_{p}$ by setting 
\begin{align*}
f(\bfk)_{i} = \begin{cases}(k_i+1)/2 \quad &\text { if } i = i_s \text{ for some } s \in [l] \\ (k_i-1)/2 \quad &\text{ if } i = i_s \text{ for some } s \in [2l] \smallsetminus [l] \\ k_i/2 \quad &\text{ otherwise. }\end{cases} \\
g(\bfk)_{i} = \begin{cases}(k_i-1)/2 \quad &\text { if } i = i_s \text{ for some } s \in [l] \\ (k_i+1)/2 \quad &\text{ if } i = i_s \text{ for some } s \in [2l] \smallsetminus [l] \\ k_i/2 \quad &\text{ otherwise. }\end{cases} 
\end{align*}
The inequality $2 \le t+1/t$ applied with $t = \prod_{s=1}^l x_{i_s} \prod_{s=l+1}^{2l} 1/x_{i_s}$ yields  
\begin{align*}
2\prod_{i \in [m]} x_i^{k_i} \le \prod_{i \in [m]} x_i^{2f(\bfk)_i}+\prod_{i \in [m]} x_i^{2g(\bfk)_i}, 
\end{align*}
which justifies the first inequality below. 
Note also that $k_i \in \{2f(\bfk)_i, 2f(\bfk)_i+1, 2f(\bfk)_i-1\}$ and, because $k_i \neq 1$, one has $k_i > 0$ if and only if $f(\bfk)_i > 0$ for $i \in [m]$. Then, since each $\bfu = (u_1, \dotsc, u_m) \in S_p$ has at most $p$ nonzero coordinates, the number of elements in the preimage $f^{-1}\{\bfu\}$ is bounded by $3^p$. The same for $g^{-1}\{\bfu\}$. Hence, the second inequality below. The final line inserts the multinomial coefficients and appeals to the multinomial theorem. 
\begin{align*}
\sum_{\bfk = (k_1, \dotsc, k_m) \in S_{2p}'} \prod_{i \in [m]} x_i^{k_i} &\le \frac{1}{2}\sum_{\bfk = (k_1, \dotsc, k_m) \in S_{2p}'} \bigg\{\prod_{i \in [m]} x_i^{2f(\bfk)_i}+\prod_{i \in [m]} x_i^{2g(\bfk)_i}\bigg\} \\
&\le 3^p \sum_{\bfu = (u_1, \dotsc, u_m) \in S_{p}} \prod_{i \in [m]} x_i^{2u_i} \\
&\le 3^p \sum_{\bfu = (u_1, \dotsc, u_m) \in S_{p}} \frac{p!}{\prod_{i \in [m]} u_i!}\prod_{i \in [m]} x_i^{2u_i} = 3^p \bigg(\sum_{i \in [m]} x_i^2\bigg)^p. \qedhere
\end{align*}
\end{proof}

An application of the preceding lemma yields the concentration inequality below. 
\begin{lem}
\label{LExpCon}
Let $X_i\sim$ {\rm Exp}$(\lambda_i)$ be mutually independent exponential random variables.
Then for each $p \in \bbZ_{> 0}$ there exists a constant $C_p > 0$ {\rm(}depending only on $p${\rm)} such that, for all  $s > 0$ and $n\in\Z_{>0}$, 
\begin{align*}
\bfP\bigg\{\bigg|\sum_{i=1}^n X_i-\sum_{i=1}^n\frac{1}{\lambda_i}\bigg| \ge s\sqrt{\sum_{i=1}^n \frac{1}{\lambda_i^2}}\bigg\} \le \frac{C_p}{s^{p}}. 
\end{align*}
\end{lem}
\begin{proof}
With $S = \sum_{i=1}^n X_i$ the claimed inequality is 
\begin{align}
\P\{|S-\E S| \ge s\sqrt{\Var S}\} \le C_p s^{-p}. \label{E19}
\end{align}
For $i \in [n]$ and $q \in \bbZ_{\ge 0}$, a brief computation gives the $q$th central moment of $X_i$ as 
\begin{align*}
\E [(X_i-\lambda_i^{-1})^q] = c_q\lambda_i^{-q} \quad \text{ where } \quad c_q = q!\displaystyle \sum_{l=0}^q \frac{(-1)^l}{l!}. 
\end{align*}
From this, the multinomial theorem and independence, 
\begin{align}
\E[|S-\E S|^{2p}] = \E\bigg[\bigg(\sum_{i=1}^{n} X_i-\sum_{i=1}^{n}\frac{1}{\lambda_i}\bigg)^{2p}\bigg] = \sum \limits_{\substack{(u_1, \dotsc, u_n) \in \bbZ_{\ge 0} \\ \sum_{i \in [n]}u_i = 2p}} \frac{(2p)!}{\prod_{i=1}^n u_i!}\prod_{i=1}^{n} \frac{c_{u_i}}{\lambda_i^{u_i}}. \label{E1}
\end{align}
Since $c_1 = 0$, the condition $u_i \neq 1$ for $i \in [k]$ can be slipped into the outer sum without breaking the equality. Then, by virtue of Lemma \ref{LComBnd} and since $0\le c_{u_i} \le u_i!$, the right-hand side of \eqref{E1} is at most 
\begin{align*}
(2p)! \sum \limits_{\substack{(u_1, \dotsc, u_n) \in \bbZ_{\ge 0} \\ \sum_{i \in [n]}u_i = 2p \\ u_i \neq 1, i \in [n]}} \prod_{i=1}^{n} \frac{1}{\lambda_i^{u_i}} \le (2p)! 3^p \bigg(\sum_{i =1}^n \frac{1}{\lambda_i^2}\bigg)^p = C_p (\Var S)^p, 
\end{align*}
where $C_p = (2p)!3^p$.   By Markov's inequality, 
\begin{align*}
\P\{|S-\E S| \ge s \sqrt{\Var S}\} \le \frac{\E[|S-\E S|^{2p}]}{s^{2p} (\Var S)^p} \le \frac{C_p}{s^{2p}}, 
\end{align*}
which implies \eqref{E19} if $s \ge 1$. Since $C_p > 1$, \eqref{E19} also trivially holds if $s < 1$. 
\end{proof}

\subsection{Vague convergence}
\label{SsVagConv}
Let $\sM(\bbR)$ and $\sC_0(\bbR)$, respectively, denote the spaces of real-valued Borel measures on $\bbR$ and real-valued, continuous functions on $\bbR$ vanishing at infinity. The \emph{vague topology} on $\sM(\bbR)$ is the minimal topology such that the maps $\mu \mapsto \int_{\bbR} f\dd \mu$ for $f \in \sC_0(\bbR)$ are continuous. With this definition, a sequence $(\mu_n)_{n \in \bbZ_{>0}}$ in $\sM(\bbR)$ converges to $\mu \in \sM(\bbR)$ in the vague topology (\emph{vaguely}) if and only if $\int_{\bbR} f \dd \mu_n \rightarrow \int_{\bbR} f\dd \mu$ as $n \rightarrow \infty$ for any $f \in \sC_0(\bbR)$. We recall the following standard facts \cite{durr, foll-ra}. 

\begin{lem}
\label{LVagConBnd}
Let $(\mu_n)_{n \in \bbZ_{>0}}$ be a sequence in $\sM(\bbR)$ such that $\mu_n \rightarrow \mu$ vaguely as $n \rightarrow \infty$ for some $\mu \in \sM(\bbR)$. Then there exists $C > 0$ such that $\mu_n(\bbR) \le C$ for $n \in \bbZ_{>0}$.
\end{lem}

Denote the subspace of subprobability measures on $\bbR$ by 
\begin{align*}\sM^+_{\le 1}(\bbR) = \{\mu \in \sM(\bbR): \mu \ge 0 \text{ and } \mu(\bbR) \le 1\}.\end{align*} 
\begin{lem}
\label{LSpmVagComp} 
For any sequence $(\mu_n)_{n \in \bbZ_{>0}}$ in $\sM^+_{\le 1}(\bbR)$, there exists a subsequence $(\mu_{n_k})_{k \in \bbZ_{>0}}$ in $\sM^+_{\le 1}(\bbR)$ and $\mu \in \sM^+_{\le 1}(\bbR)$ such that $\mu_{n_k} \rightarrow \mu$ vaguely as $k \rightarrow \infty$. 
\end{lem}

\subsection{Cauchy transform}
\label{SsCauTr}
Write $\sM^+(\bbR)$ for the space of $\bbR_{\ge 0}$-valued Borel measures on $\bbR$. Let $\mu \in \sM^+(\bbR)$ and $D = \bbC \smallsetminus \supp \mu$. The \emph{Cauchy transform} of $\mu$ is defined as the convolution 
\begin{align}
\label{ECatDef}
\bCat{\mu}{z} = \int_\bbR \frac{\mu(\dd t)}{z-t} \quad \text{ for } z \in D. 
\end{align}
The integral 
is well-defined since $\mu(\bbR) < \infty$, $\mu(\bbR \smallsetminus \supp \mu) = 0$ and the distance $\Delta(z) = \inf \limits_{t \,\in\, \supp \mu}|z-t| > 0$ for $z \in D$, the latter due to $\supp \mu$ being closed.
By direct computation, 
\begin{align}
\label{ELip}
|\bCat{\mu}{z}-\bCat{\mu}{w}| = \left|(w-z) \int_{\bbR} \frac{\mu(\dd t)}{(z-t)(w-t)} \right| \le \frac{|z-w|\mu(\bbR)}{\Delta(z)\Delta(w)} \quad \text{ for } z, w \in D, 
\end{align}
and $\Cat_{\mu}$ is holomorphic on $D$ with 
\begin{align}
\label{E42}
\partial_z^n \bCat{\mu}{z} = (-1)^n n! \int_{\bbR} \frac{\mu(\dd t)}{(z-t)^{n+1}} \quad \text{ for } z \in D \text{ and } n \in \bbZ_{\ge 0}. 
\end{align}
\begin{lem}
\label{LCatCont}
Let $\mu \in \sM^+(\bbR)$ and $(\mu_n)_{n \in \bbZ_{>0}}$ be a sequence in $\sM^+(\bbR)$ such that $\lim \limits_{n \rightarrow \infty}\mu_n = \mu$ vaguely. Let $S \subset \bbC$ denote the closure of $\supp \mu \cup \bigcup_{n \in \bbZ_{>0}} \supp \mu_n$. Let $K \subset \bbC \smallsetminus S$ be compact. Then, for any $\epsilon > 0$, there exists $n_0 \in \bbZ_{>0}$ such that 
\begin{align*}
|\bCat{\mu_n}{z}-\bCat{\mu}{z}| < \epsilon \quad \text{ for all } n \ge n_0 \text{ and } z \in K. 
\end{align*}
\end{lem}
\begin{proof}
Since 
the real and imaginary parts of the integrand in \eqref{ECatDef} belong to $\sC_0(\bbR)$, it is immediate from the definition of vague convergence that $\bCat{\mu_n}{z} \rightarrow \bCat{\mu}{z}$ as $n \rightarrow \infty$ pointwise for $z \in \bbC \smallsetminus S$. To upgrade to uniform convergence, first pick a constant $C > 0$ as in Lemma \ref{LVagConBnd} such that $\mu_n(\bbR) \le C$ for $n \in \bbZ_{>0}$. Also, since $K$ is compact, $S$ is closed and $K \cap S = \emptyset$, there exists $\delta > 0$ such that $|z-x| \ge \delta$ for $z \in K$ and $x \in S$. 
Hence, it follows from \eqref{ELip} that 
\begin{align*}
|\bCat{\mu_n}{z}-\bCat{\mu_n}{w}| \le \frac{C}{\delta^2}|z-w| \quad \text{ for } z, w \in K \text{ and } n \in \bbZ_{>0}. 
\end{align*}
In particular, the sequence of functions $(\Cat_{\mu_n})_{n \in \bbZ_{>0}}$ is equicontinuous on $K$. 
This property together with the pointwise convergence on the compact set $K$ implies that $\bCat{\mu_n}{z} \rightarrow \bCat{\mu}{z}$ uniformly in $z \in K$ as $n \rightarrow \infty$ \cite[Chapter 7]{rudi-pma}. 
\end{proof}

\subsection{Hausdorff metric}
\label{SsHaus}
Let $(X, d)$ be a metric space. For $\epsilon > 0$ and $A \subset X$, the $\epsilon$\emph{-fattening} of $A$ is 
$A^{\epsilon} = \{x \in X: d(x, y) < \epsilon \text{ for some } y \in A\}$. Let $\sX$ be the space of nonempty, bounded, closed subsets of $X$. The Hausdorff metric (\cite[p.\ 280]{munk}) on $\sX$ is defined by 
\begin{align*}
\cHDis{A}{B} = \inf \{\epsilon > 0: A \subset B^\epsilon \text{ and } B \subset A^{\epsilon}\} \quad \text{ for } A, B \in \sX. 
\end{align*}

\bibliographystyle{plain}
\bibliography{timorefs}

\end{document}